\documentclass[12pt, twoside, leqno]{article}

\usepackage{amsmath,amsthm}
\usepackage{amssymb}

\usepackage{enumerate}
\usepackage{multirow}

\usepackage{graphicx}

\newtheorem{thm}{Theorem}
\newtheorem{lem}{Lemma}[section]
\newtheorem{prop}[lem]{Proposition}
\newtheorem{cor}[lem]{Corollary}

\theoremstyle{definition}
\newtheorem{defin}[lem]{Definition}

\newtheorem{exa}[lem]{Example}

\numberwithin{equation}{section}

\newcommand{\IE}{\textit{i.e., }}
\newcommand{\CF}{\textit{cf.\ }}	% Compare to

\newcommand{\IntegerRing}{{\mathbb Z}}

\newcommand{\RationalField}{{\mathbb Q}}
\newcommand{\RealField}{{\mathbb R}}

\newcommand{\tr}[1]{{\hspace{0.5mm}{}^t\hspace{-0.5mm}#1}}
\newcommand{\trace}[1]{{{\rm Tr}(#1)}}
\newcommand{\abs}[1]{ {\left\lvert #1 \right\rvert} }

\newcommand{\QA}[3]{{Q_{#1}(#2, #3)}}
\newcommand{\RA}[3]{{R_{#1}(#2, #3)}}
\newcommand{\Mod}[2]{{\lfloor #1 \rfloor_{#2}}}

\begin{document}

%%%%% To ease editing, for IMPAN journals add:

\baselineskip=17pt

%%%%%%%%%%%%%%%%

\title{On positive-definite ternary quadratic forms with the same representations over $\IntegerRing$}

\author{Ryoko Oishi-Tomiyasu\\
Graduate School of Science and Engineering,\\ Yamagata University / JST PRESTO \\
990-8560 1-4-12, Kojirakawa-cho, Yamagata-shi, Yamagata, Japan\\
E-mail: tomiyasu@imi.kyushu-u.ac.jp\footnote{Current affiliation: Institute of Mathematics for Industry (IMI), Kyushu University}
}

\date{}

\maketitle

%% Classification and key words; note that the 2010 classification is used:

\renewcommand{\thefootnote}{}



\renewcommand{\thefootnote}{\arabic{footnote}}
\setcounter{footnote}{0}

%%%%%%%%

\begin{abstract}
Kaplansky conjectured that 
if two positive-definite ternary quadratic forms have perfectly identical representations over $\IntegerRing$,
they are equivalent over $\IntegerRing$ or constant multiples of regular forms, or 
is included in either of two families parameterized by $\RealField^2$. %(so called, hexagonal and rhombohedral families).
Our results aim to clarify the limitations imposed to such a pair by computational and theoretical approaches.
Firstly, the result of an exhaustive search for such pairs of integral quadratic forms 
is presented, in order to provide a concrete version of the Kaplansky conjecture.
The obtained list contains a small number of non-regular forms that were confirmed to have the identical representations up to 3,000,000 by computation.
However, a strong limitation on the existence of such pairs is still observed, regardless of whether the coefficient field is $\RationalField$ or $\RealField$.
Secondly, we prove that if two pairs of ternary quadratic forms have the identical simultaneous representations over $\RationalField$, their constant multiples are equivalent over $\RationalField$.
This was motivated by the question why the other families were not detected in the search.
In the proof, the parametrization of quartic rings and their resolvent rings by Bhargava is used to discuss pairs of ternary quadratic forms.
\end{abstract}

\section{Introduction}

The aim of this article is to investigate the pairs of ternary positive-definite quadratic forms $f$ and $g$ with perfectly identical representations over $\IntegerRing$. 
If such a pair also has the identical multiplicities (\IE theta series), 
it is known that $f$ and $g$ are equivalent over $\IntegerRing$ \cite{Schiemann97}. 
Therefore, ``identical representations'' means $q_\IntegerRing(f) = q_\IntegerRing(g)$ herein,
where
$q_\IntegerRing(f) := \{ f(x) : 0 \neq x \in \IntegerRing^3 \}$ is the set of \textit{representations over $\IntegerRing$}.
The same problem was also discussed in \cite{Hsia81}, and crystallography as mentioned in the following.
 
For any $N$-ary quadratic forms $f$ and $g$ with real coefficients,
we will use the notation $f \sim g$ when 
$f$ and $g$ are \textit{equivalent over $\IntegerRing$}, \IE $f(\mathbf{x} w) = g(\mathbf{x})$ for some $w \in GL_N(\IntegerRing)$.
With regard to the binary case, it was proved by a number of mathematicians
that if two positive-definite quadratic forms $f \not\sim g$ have the identical representations over $\IntegerRing$,
then such a pair is provided by $d (x_1^2 - x_1 x_2 + x_2^2)$ and $d (x_1^2 + 3 x_2^2)$ for some $d > 0$ \cite{Watson80}.
All the indefinite binary cases were also determined in \cite{Delang82}, \cite{Delang87}.

As an immediate consequence, 
%positive-definite ternary forms $f \not\sim g$ such that $q_\IntegerRing(f) = q_\IntegerRing(g)$, are obtained. In fact, 
it is not difficult to verify that $f$ and $g$ in each of the following families,  satisfy $q_\IntegerRing(f) = q_\IntegerRing(g)$, regardless of the values of $c, d$ (see Section~\ref{Case of Disc(A_i, B_i) = 0 (proofs of Propositions 1, 2)}):
\begin{enumerate}[(i)]
\item $\{ f, g \} = \{ c(x_1^2 - x_1 x_2 + x_2^2) + d x_3^2, c(x_1^2 + 3 x_2^2) + d x_3^2 \}$,

\item $\{ f, g \} = \{ c(x_1^2 - x_1 x_2 + x_2^2) + d (x_1 + x_2 + 3 x_3)^2, c(x_1^2 + 3 x_2^2) + d (x_1 + 3 x_3)^2 \}$.
\end{enumerate}
Kaplansky conjectured in his letter to Schiemann in 1997 that 
all the pairs of non-regular forms $f, g$ satisfy
$f \not\sim g$ and $q_\IntegerRing(f) = q_\IntegerRing(g)$,
as long as they belong to either of the above (i), (ii).
A quadratic form $f$ over $\RationalField$ is said to be \textit{regular}, if 
any $m \in \RationalField$ that is
represented by $f$ over $\IntegerRing_v$ for any primes $v$,
is also represented by $f$ over $\IntegerRing$,
where $v = \infty$ is also included, and $\IntegerRing_{\infty} = \RealField$.

It was proved in \cite{Do2012} that the conjecture holds if only diagonal quadratic forms are considered
(all such cases are provided by No.34 and No.50 in Tables~\ref{Forty-nine groups of ternary positive definite quadratic forms representing the same numbers(1/3)}--\ref{Forty-nine groups of ternary positive definite quadratic forms representing the same numbers(3/3)}). 

In order to obtain more detailed information about this problem, 
an exhaustive search for such $f, g$ with integral quadratic coefficients were carried out.
It can be proved that if $f, g$ over $\RealField$ satisfy $q_\IntegerRing(f) = q_\IntegerRing(g)$,
infinitely many $f_2, g_2$ over $\IntegerRing$ with $q_\IntegerRing(f_2) = q_\IntegerRing(g_2)$ are generated from these $f, g$ (Lemma \ref{lem:decomposition of positive definite symmety matrices}). Hence, the search also provides information about the case of real forms.
%As a result, more detailed information about the Kaplansky conjecture was obtained.
The result is presented in Tables~\ref{Forty-nine groups of ternary positive definite quadratic forms representing the same numbers(1/3)}--\ref{Forty-nine groups of ternary positive definite quadratic forms representing the same numbers(3/3)}
in Section~\ref{A table of quadratic forms with the same representations over Z}, 
which indicates that
the existence of such pairs is rather limited as conjectured by Kaplansky, although the current list includes some non-regular cases.

If the quadratic forms contained in the above (i), (ii) are excluded, 
our exhaustive search finds only 151 equivalence classes of quadratic forms 
that have perfectly identical representations over $\IntegerRing$ as another class.
Among the 151 classes, 36 are not regular.
In addition, the list includes a case that has been proved to be regular only under the generalized Riemann hypothesis \cite{Oliver2014}.

In what follows, $\{ f, g \} \sim \{ f_2, g_2 \}$ means that either of $f \sim f_2$, $g \sim g_2$ or $f \sim g_2$, $g \sim f_2$ holds.
The following is suggested from the computational result:
\begin{description}
\item[Kaplansky conjecture (modified version):]
if two ternary positive-definite quadratic forms $f \not\sim g$ satisfy $q_\IntegerRing(f) = q_\IntegerRing(g)$,
one of the following holds:
\begin{enumerate}[(i)]
\item $\{ f, g \} \sim \left\{ c ( x_1^2 - x_1 x_2 + x_2^2 ) + d x_3^2,  c ( x_1^2 + 3 x_2^2) + d x_3^2 \right\}$ for some
$c, d \in \RealField$,

\item $\{ f, g \} \sim \left\{ c ( x_1^2 - x_1 x_2 + x_2^2 ) + d (x_1 + x_2 + 3 x_3)^2,  c ( x_1^2 + 3 x_2^2) + d (x_1 + 3 x_3)^2 \right\}$ for some
$c, d \in \RealField$,

\item $\{ f, g \} \sim \{ c f_2, c g_2 \}$ for some $c \in \RealField$ and $f_2, g_2$ contained in either of the No.1--53 in
Tables~\ref{Forty-nine groups of ternary positive definite quadratic forms representing the same numbers(1/3)}--\ref{Forty-nine groups of ternary positive definite quadratic forms representing the same numbers(3/3)}.
\end{enumerate}
\end{description}

Although the non-regular cases newly found in our search are also included in the above, 
it should be noted that 
they were just confirmed to have the identical set of representations up to 3,000,000, by computation.
With regard to regular quadratic forms,
it is not difficult to confirm that all of their integral representations are perfectly identical. 
%(the number is $\infty$ with regard to the quadratic forms are regular). 

Before proceeding to our theoretical results motivated by the Kaplansky conjecture, 
first we provide the following proposition; for any field $k$ of characteristic $\ne 2$, 
the set of all the $n$-ary quadratic forms over $k$
is denoted by ${\rm Sym}^2 (k^n)^*$,
and the set of all the $s$-tuples of such forms 
is denoted by ${\rm Sym}^2 (k^n)^* \otimes_{k} k^s$.
For any subring $R \subset k$ and $f_1, \ldots, f_s \in {\rm Sym}^2 (k^n)^*$,
the elements of $q_{R}(f_1, \ldots, f_s) := \{ (f_1(v), \ldots, f_s(v)) : 0 \neq v \in R^n \}$
are called \textit{simultaneous representations of $f_1, \ldots, f_s$ over $R$}.

\begin{prop}\label{thm:proposition 1}
With regard to positive-definite $f \in {\rm Sym}^2 (\RealField^3)^*$ not  
contained in $\RealField^\times {\rm Sym}^2 (\RationalField^3)^*$,
the Kaplansky conjecture is true 
if and only if the following (*) is true:
\begin{description}
\item[(*)]
If both of $(A_i, B_i) \in {\rm Sym}^2 (\RationalField^3)^* \otimes_{\RationalField} \RationalField^2$ ($i = 1, 2$) satisfy
\begin{enumerate}[(a)]
\item $A_i$ and $B_i$ are linearly independent over $\RationalField$,
\item $c A_i + d B_i$ is positive-definite for some $c, d \in \RationalField$ (\IE d-pencil),
\item $q_\IntegerRing(A_1, B_1) = q_\IntegerRing(A_2, B_2)$, then 
\end{enumerate}
$(A_1, B_1) = (w, 1) \cdot (A_2, B_2)$ holds for some $w \in GL_3(\IntegerRing)$,
or otherwise, $\{ (A_1, B_1), (A_2, B_2) \}$ equals either of the following as a set, for some $(w_i, v) \in GL_3(\IntegerRing) \times GL_2(\RationalField)$ ($i = 1, 2$):
\begin{enumerate}[(i)]
\item $\left\{  
	(w_1, v) \cdot ( x_1^2 - x_1 x_2 + x_2^2, x_3^2 ),
	(w_2, v) \cdot ( x_1^2 + 3 x_2^2, x_3^2 )
\right\}$,

\item $\left\{  
	(w_1, v) \cdot ( x_1^2 - x_1 x_2 + x_2^2, (x_1 + x_2 + 3 x_3)^2 ),
	(w_2, v) \cdot ( x_1^2 + 3 x_2^2, (x_1 + 3 x_3)^2 )
\right\}$,

\end{enumerate}
where $GL_3(\RationalField) \times GL_2(\RationalField)$ acts on ${\rm Sym}^2 (\RationalField^3)^* \otimes_{\RationalField} \RationalField^2$ by
\begin{eqnarray*}
\left(w, 
\begin{pmatrix}
	r & s \\
	t & u
\end{pmatrix}
\right) \cdot (A, B) = (r A({\mathbf x} w) + s B({\mathbf x} w), t A({\mathbf x} w) + u B({\mathbf x} w)).
\end{eqnarray*}

\end{description}

\end{prop}

We shall say that a pair $(A, B) \in {\rm Sym}^2 (k^n)^* \otimes_{k} k^2$ is \textit{singular},
if $\det(A x - B y) = 0$ as a polynomial in $k[x, y]$, and \textit{non-singular} otherwise.
A non-singular $(A, B)$ is said to be \textit{anisotropic} over $k$ if $A(x) = B(x) = 0$ does not hold for any $0 \neq x \in k^n$.
It should be noted that any pair $(A, B)  \in {\rm Sym}^2 (k^n)^* \otimes_{k} k^2$ with $n \geq 3$
is a d-pencil if and only if $(A, B)$ is non-singular and anisotropic over $\RealField$ \cite{Finsler36} (\CF \cite{Uhlig79}).

%For any polynomial $P(x, y) = \prod_{j=1}^n (\alpha_j x - \beta_j y)$,
%its discriminant ${\rm Disc}(P(x, y))$ is defined by ${\rm Disc}(P(x, y)) = \prod_{1 \le j < k \le n} (\alpha_j \beta_k - \alpha_k \beta_j)^2$.
It can be proved without difficulty that 
(*) holds true in the following case: %, if only the case of ${\rm Disc}(A_i, B_i) := {\rm Disc}(4 \det(A_i x - B_i y)) = 0$ is considered.
%\begin{thm}
%In the situation of Theorem \ref{thm:main result over RationalField},
%we also assume that 
%\begin{enumerate}[(a)]
%\item[(d)] $q_{\IntegerRing_p}(A_1, B_1) = q_{\IntegerRing_p}(A_2, B_2)$ for any finite primes $p$,
%and all the root of $\det(A_i x - B_i y) = 0$ ($i = 1, 2$) belong to ${\mathbb P}^1(\RealField)$.
%\end{enumerate}
%In this case, the above $r_1, r_2$ are squares in $\IntegerRing^\times$, hence $(A_1, B_1)$ is equivalent to $(A_2, B_2)$ by the action of $GL_3(\RationalField) \times \{ 1 \}$.
%\qed
%\end{thm}

%Note that the above (d) always holds, when $(A_1, B_1), (A_2, B_2) \in ({\rm Sym}^2 \RationalField^3)^* \otimes_{\RationalField} \RationalField^2$ satisfy the conditions (a)--(c) .

\begin{prop}\label{thm:theorem 1}
If $\det(A_i x - B_i y) = 0$ has a multiple root (for at least one of $i = 1, 2$), 
the above (*) holds true. 
\qed
\end{prop}

Motivated by the Kaplansky conjecture and Proposition \ref{thm:proposition 1}, the following is proved in this article.

\begin{thm}\label{thm:main result over RationalField}
We assume that $(A_1, B_1), (A_2, B_2) \in {\rm Sym}^2 (\RationalField^3)^* \otimes_{\RationalField} \RationalField^2$ satisfy
\begin{enumerate}[(a)]
\item \label{item: assumption (a)} $A_i$ and $B_i$ are linearly independent over  $\RationalField$.
\item[(b')] \label{item: assumption (b)} $(A_i, B_i)$ is non-singular and anisotropic over $\RationalField$.
\item[(c')] \label{item: assumption (c)} $q_\RationalField(A_1, B_1) = q_\RationalField(A_2, B_2)$.
\end{enumerate}
In this case, 
$(r_1 A_1, r_1 B_1)$ is equivalent to $(r_2 A_2, r_2 B_2)$ by the action of $GL_3(\RationalField) \times \{ 1 \}$
for any integers $r_1, r_2$ that satisfy $r_1^{-1} \det(A_1 x - B_1) = r_2^{-1} \det(A_2 x - B_2)$.
\qed
\end{thm}

Since (b') is obtained by replacing $\RealField$ in (b) with $\RationalField$, 
Theorem \ref{thm:main result over RationalField} handles a more general case than Proposition \ref{thm:proposition 1}.
In the proof, one-to-one correspondence between the set of pairs of quadratic forms and the set of quartic rings and its resolvent cubic rings\cite{Bhargava2004}, is used.

%\begin{thm}\label{thm:main result 2 over RationalField}
%In the situation of Theorem \ref{thm:main result over RationalField}, 
%suppose that $(A_1, B_1), (A_2, B_2)$ also 
%satisfy $q_\IntegerRing(A_1, B_1) = q_\IntegerRing(A_2, B_2)$.
%In this case, $(A_1, B_1)$ is equivalent to $(A_2, B_2)$ by the action of $GL_3(\RationalField) \times \{ 1 \}$.
%\qed
%\end{thm}

%The above theorem actually limits the existence of positive-definite ternary quadratic forms with perfectly identical representations:
%
%\begin{cor}\label{cor: corollary of main theorem}
%Let $f_1 \not\sim f_2$ be ternary positive-definite quadratic forms with real coefficients satisfying $q_\IntegerRing(f_1) = q_\IntegerRing(f_2)$.
%In this case, $f_1$ and $f_2$ are equivalent over $\RationalField$.
%\qed
%\end{cor}

We shall explain the outline of the proof of Theorem \ref{thm:main result over RationalField};
as proved in Proposition \ref{prop:same det(Ax+By)},
the situation of Theorem \ref{thm:main result over RationalField} 
leads to $\det(A_1 x - B_1 y) = c \det(A_2 x - B_2 y)$ for some $c \in \RationalField^\times$.
This and $q_\RationalField(A_1, B_1) = q_\RationalField(A_2, B_2)$
imply that 
both $(A_i, B_i)$ can be transformed to another pairs $(\tilde{A}_i, \tilde{B}_i)$ ($i = 1, 2$)
that correspond to quartic $\RationalField$-algebras $\RationalField[x]/(f_i(x))$, 
by the action of some $(W_i, V) \in GL_3(\RationalField) \times GL_2(\RationalField)$
(Lemma \ref{prop: monogenic ring}).
Proposition \ref{prop:same det(Ax+By)} also implies that 
if the resolvent cubic polynomial of $f_i(x)$ is denoted by $f_i^{res}(x)$,
$\RationalField[x]/(f_i^{res}(x))$ ($i = 1, 2$) are isomorphic as $k$-algebras.
Since the above $V$ is common, $(\tilde{A}_i, \tilde{B}_i)$ also satisfy (a), (b') and (c').
In particular, (b') implies that $f_i(x) = 0$ ($i = 1, 2)$ have a root in $\RationalField_p$ 
and $\RationalField[x]/(f_i^{res}(x))$ completely splits over $\RationalField_p$
with regard to the same set of primes $p$ (Corollaries \ref{cor: anisotropic iff anisotropic}, \ref{cor: (A, B) is isotropic over k}).
Thus, $\RationalField[x]/(f_i(x))$ ($i = 1, 2$) are isomorphic as $k$-algebras (Lemma \ref{lem:isomorphic k-algebras}).
Theorem \ref{thm:main result over RationalField} and some relation formula between $f_1(x)$ and $f_2(x)$ (Proposition \ref{prop:case of monogenic})
are obtained as a result.

In crystallography, in order to determine the crystal lattice (\IE the equivalence class over $\IntegerRing$ of a positive-definite ternary quadratic form $f$ with real coefficients) from information about $q_\IntegerRing(f)$ that is extracted from the experimental data,
it has been recognized that 
some $f \not\sim g$ have the perfectly identical representations over $\IntegerRing$ (\cite{Mighell75}, \CF \cite{Tomiyasu2016}).
A three-dimensional lattice is \textit{hexagonal} if and only if 
it has a basis $v_1, v_2, v_3$ satisfying 
\begin{eqnarray*} 
	( v_i \cdot v_j )_{1 \leq i, j \leq 3} = 
	\begin{pmatrix}
		c    & -c /2 & 0 \\
	  -c/2 &  c      & 0 \\
	       0    &     0     & d \\
	\end{pmatrix},
\end{eqnarray*} 
for some $c, d \in \RealField$. A three-dimensional lattice is \textit{rhombohedral} if and only if 
it has a basis satisfying 
\begin{eqnarray*} 
	( v_i \cdot v_j )_{1 \leq i, j \leq 3} &=& 
	\begin{pmatrix}
		c+d    & -c/2+d & -c/2+d \\
	  -c/2+d &  c+d     & -c/2+d \\
	  -c/2+d & -c/2+d & c+d
	\end{pmatrix} \nonumber \\
	&=&
	\begin{pmatrix}
		1 &  0 & 0 \\
	    0 &  1 & 0 \\
	   -1 &-1&  1
	\end{pmatrix}
	\begin{pmatrix}
		c+d    & -c/2+d & 3 d \\
	  -c/2+d &  c+d     & 3 d \\
	    3 d & 3 d & 9 d
	\end{pmatrix}
	\begin{pmatrix}
		1 &  0 & -1 \\
	    0 &  1 & -1 \\
	    0 &  0 &   1
	\end{pmatrix},
\end{eqnarray*} 
for some $c, d \in \RealField$. Hence, (i) $c (x_1^2 - x_1 x_2 + x_2^2) + d x_3^2$
and (ii) $c (x_1^2 - x_1 x_2 + x_2^2) + d (x_1 + x_2 + 3 x_3 )^2$  ($c, d \in \RealField$)
parameterize all hexagonal and rhombohedral lattices, respectively.

\section{Notation and symbols}
\label{Notation and symbols}
Throughout this paper, 
a quadratic form $\sum_{1 \leq i \leq j \leq n} s_{ij} x_i x_j$ is always identified with the symmetric matrix
$(s_{ij})$ with $s_{ii}$ in the $(i, i)$-entry and $s_{ij}/2$ in the $(i, j)$-entry.
For any quadratic forms $f(x_1, \ldots, x_m)$, $g(x_1, \ldots, x_n)$,
their \textit{direct sum} is the $(n+m)$-ary quadratic form $f(x_1, \ldots, x_m) + g(x_{m+1}, \ldots, x_{m+n})$,
and denoted by $f \perp g$.
A quadratic form $\sum_{i=1}^n c_i x_i^2$ is represented as a diagonal matrix or $[c_1, \ldots, c_n]$.
In particular, $[c]$ means the unary quadratic form $c x^2$.

In Tables~\ref{Forty-nine groups of ternary positive definite quadratic forms representing the same numbers(1/3)}--\ref{Forty-nine groups of ternary positive definite quadratic forms representing the same numbers(3/3)} presenting the search result, 
the ternary quadratic forms are \textit{reduced}, in the sense that they are Minkowski-reduced Eq.(\ref{eq:condition3})
and satisfy the boundary conditions $C$1--5 provided by Eisenstein~\cite{Eisenstein1851}:
\begin{enumerate}[($C$1)]
	\item $s_{12}, s_{13}, s_{23} > 0$ or $s_{12}, s_{13}, s_{23} \leq 0$,
	\item $s_{11} = s_{22} \Longrightarrow \abs{s_{23}} \leq \abs{s_{13}}$,
	\item $s_{22} = s_{33} \Longrightarrow \abs{s_{13}} \leq \abs{s_{12}}$.
	\item case of positive Eisenstein forms (\IE $s_{12}, s_{13}, s_{23}> 0$): 
		%under the assumption that $s_{ij} = s_{ji}$, 
		for any distinct $1 \leq i, j, k \leq 3$,
		$s_{ii} = 2 \abs{s_{ij}} \Rightarrow \abs{s_{i k} } \leq 2 \abs{ s_{j k} }$.
	\item case of non-positive Eisenstein forms (\IE $s_{12}, s_{13}, s_{23} \leq 0$):
		for any distinct $1 \leq i, j, k \leq 3$, $s_{ii} = 2 \abs{s_{ij}} \Rightarrow s_{i k} = 0$.
		In addition, 
\begin{eqnarray*}
	s_{11} + s_{22} = 2 \abs{ s_{12} + s_{13} + s_{23} } \Rightarrow s_{11} \leq \abs{ s_{12} + 2 s_{13} }. \label{eq:condition2}
\end{eqnarray*}
\end{enumerate}

With regard to the notation for pairs of quadratic forms and quartic rings, 
those of \cite{O'Dorney2016}, in addition to \cite{Bhargava2004} are adopted herein.  
For any unitary commutative ring $R$, 
$(R^n)^* := {\rm Hom}(R^n, R)$ has the structure of a finitely-generated free $R$-module. 
%$({\rm Sym}^2 R^n)^*$ denotes the set of all the $n$-ary quadratic forms that maps $R^n$ to $R$.
For any $R$-module, 
${\rm Sym}^i M$ is the $i$-th symmetric power of $M$, 
and $\Lambda^i M$ is the $i$-th exterior power of $M$.
If ${\rm char}\ R \ne 2$,
${\rm Sym}^2(R^n)^*$ and ${\rm Sym}^2(R^n)^* \otimes_R R^s$
can be naturally identified with the set of all the $n$-ary quadratic forms over $R$,
and the set of all the $s$-tuples of $n$-ary quadratic forms over $R$, respectively.

It is said that $f, g \in {\rm Sym}^2 (R^n)^*$ are \textit{equivalent over $R$}, and denoted by $f \sim_R g$,
if there exists $w \in GL_n(R)$ such that $f(\mathbf{x} w) = g(\mathbf{x})$.
For any subring $R_2 \subset R$ and $f \in {\rm Sym}^2 (R^n)^*$,
the elements of $q_{R_2}(f) := \{ f(v) : 0 \neq v \in R_2^n \}$
are called the \textit{representations of $f$ over $R_2$}.

Throughout this paper, $k$ is a global field with ${\rm char}\ k \ne 2$, 
although some statements hold true for any fields.
The algebraic closure of $k$ is always denoted by $\bar{k}$.
As in the theory of prehomogeneous vector spaces,
${\rm Sym}^2 (k^3)^* \otimes_{k} k^2$ and $GL_3(k) \times GL_2(k)$ 
are denoted by $V_k$ and $G_k$, respectively.

The discriminant ${\rm Disc}(P(x, y))$ of a polynomial $P(x, y) = \prod_{j=1}^n (\alpha_j x - \beta_j y)$ is defined by ${\rm Disc}(P(x, y)) = \prod_{1 \le j < k \le n} (\alpha_j \beta_k - \alpha_k \beta_j)^2$.
The resolvent of a quartic polynomial $f(x) := x^4 + a_{3} x^3 + a_{2} x^2 + a_{1} x + a_{0}$ is the cubic polynomial
$f^{res}(x) :=  x^3 - a_{2} x^2 + (a_{1} a_{3}  - 4 a_{0}) x  + (4 a_{0} a_{2} - a_{1}^2 - a_{0} a_{3}^2)$.

For any $(A, B) \in V_k$,
${\rm Disc}(4 \det (A x - B y))$ is denoted by ${\rm Disc}(A, B)$.
$(\QA{k}{A}{B}, \langle 1, \xi_{1}, \xi_{2}, \xi_{3} \rangle)$ 
and $(\RA{k}{A}{B}, \langle 1, \omega_{1}, \omega_{2} \rangle)$ denote the quartic $k$-algebra
and its resolvent cubic algebra with their bases, assigned to $(A, B)$ by Eq.(\ref{eq: multiplicative structure of Q}), (\ref{eq: normalization}) and (\ref{eq: multiplicative structure of R}).

\section{A table of quadratic forms with the same representations over $\IntegerRing$}
\label{A table of quadratic forms with the same representations over Z}

The algorithm used to exhaustively search for
sets of positive-definite ternary quadratic forms with the identical representations is presented in Table~\ref{Recursive procedure} of Section~\ref{Algorithm}, 
with all the discussions to prove that the algorithm actually works.

The following $P$1--3 describe the searched region; 
for each quadratic form $\sum_{1 \leq i \leq j \leq 3} s_{ij} x_i x_j$ that satisfies the following conditions,
all the representations less than a given threshold are computed as a sorted set $\Lambda = \langle q_1, \ldots, q_t \rangle$,
and passed to the algorithm as an input. 
During the execution of the algorithm, the computer program always checks if the threshold is large enough 
to output all the forms with the required properties.
\begin{enumerate}[($P$1)]
 	\item all of $s_{ij}$ are integral and their greatest common divisor is 1.
 	\item The form is reduced, \IE satisfies Eq.(\ref{eq:condition3}) and $C$1-- 5 in Section \ref{Notation and symbols}.
      \item $s_{33} \leq 115$.

\end{enumerate}
All quadratic forms over $\RationalField$ (or their scalar multiples)
can be contained in the searched region by increasing the number 115 in (P3).
%A rational quadratic form $S$ is said to be \textit{regular}, when $S$ represents $a \in \RationalField$ over $\IntegerRing$ if and only if this is true over $\IntegerRing_p$ for any prime $p$.
The algorithm 
was also applied to all (possibly) regular quadratic forms in the tables of \cite{Jagy97},
in order to check that all the pairs of regular forms are contained in the output.

The results are presented in Tables~\ref{Forty-nine groups of ternary positive definite quadratic forms representing the same numbers(1/3)}--\ref{Forty-nine groups of ternary positive
definite quadratic forms representing the same numbers(3/3)}.
%They are classified by the size and the ratio of the determinants.
By using a computer,
the quadratic forms in each set have been confirmed to have the identical representations over $\IntegerRing$ up to 3,000,000. 
The set contained in either of the hexagonal and rhombohedral families were removed from the output.

\begin{table}[htbp]
\caption{Ternary positive-definite quadratic forms with the same representations over $\IntegerRing$}
\label{Forty-nine groups of ternary positive definite quadratic forms representing the same numbers(1/3)}
\begin{minipage}{\textwidth}
\begin{scriptsize}
\begin{tabular}{p{1mm}p{19mm}p{4mm}p{4mm}p{4mm} p{4mm}p{4mm}p{4mm}
%p{22mm}
%		p{0mm} p{5mm}p{5mm}p{5mm} p{5mm}p{5mm}p{5mm}
		lr}
 \hline
   No.
	& Determinant (Ratio) & $s_{11}$ & $s_{22}$ & $s_{33}$ & $s_{12}$ & $s_{13}$ & $s_{23}$ 
%	& & $a^*$ & $b^*$ & $c^*$ & $\cos \alpha^*$ & $\cos \beta^*$ & $\cos \gamma^*$
%	& Anisotropic prime\ ($u \in \RationalField^\times/(\RationalField^\times)^2$, $\not\subset q_{\RationalField}(S)$)
	& Integers not represented by the genus
	& Bravais Type\\
 \hline
1 &
$^{**}2^9 3^3$ (8) & 11 & 32 & 44 & -8 & -4 & -16
% & & $2\sqrt{11}$ & $4\sqrt{2}$ & $\sqrt{11}$ & $-\frac{1}{\sqrt{22}}$ & $-\frac{1}{11}$ & $-\frac{1}{\sqrt{22}}$ 
% mod 3:  (2, 2*3, 2*3^2).
% & \multirow{2}{*}{ 3 ($3 u_{+}$) }
 & $4 n + 1$, $4 n + 2$, $8 n + 7$, $2^2 (4 n + 1)$,
$2^2 (4 n + 2)$,
 & Triclinic \\
   &
$^{**}2^6 3^3 11$ (11) & 11 & 32 & 59 & 8 & 10 & 8 
% & & $\sqrt{59}$ & $4\sqrt{2}$ & $\sqrt{11}$ & $\frac{1}{\sqrt{22}}$ & $\frac{5}{\sqrt{649}}$ & $\frac{1}{\sqrt{118}}$ 
 & $3n + 1$, $3^{2k+1}(3n + 1)$
 & Triclinic \\
\hline
2 &
$^{**}2^9 3^2$ (8) & 5 & 20 & 48 & -4 & 0 & 0 
% & & $2\sqrt{5}$ & $4\sqrt{3}$ & $\sqrt{5}$ & $0$ & $\frac{1}{5}$ & $0$ 
% & \multirow{2}{*}{ 3 ($u_{+}$) }
 & $4 n + 2$, $4 n + 3$, $8 n + 1$, $2^2 (4 n + 2)$, $2^2 (4 n + 3)$,
 & Monoclinic(P) \\
   &
$^{*}2^6 3^2 11$ (11) & 5 & 20 & 68 & -4 & -4 & -8 
% & & $2\sqrt{17}$ & $2\sqrt{5}$ & $\sqrt{5}$ & $-\frac{1}{5}$ & $-\frac{1}{\sqrt{85}}$ & $-\frac{1}{\sqrt{85}}$ 
 & $3^{2k}(3n + 1)$
 & Triclinic \\
\hline
3 &
$^{**}2^9 3^2$ (8) & 17 & 17 & 20 & -14 & -4 & -4
% & & $\sqrt{5}$ & $2\sqrt{3}$ & $2\sqrt{5}$ & $0$ & $\frac{1}{5}$ & $0$
% & \multirow{2}{*}{ 3 ($u_{+}$) }
 & $4 n + 2$, $4 n + 3$, $8 n + 5$, $2^2 (4 n + 2)$, $2^2 (4 n + 3)$, 
 & Monoclinic(C) \\
   &
$^{**}2^6 3^2 11$ (11) & 17 & 20 & 20 & -4 & -4 & -8
% & & $2\sqrt{2}$ & $2\sqrt{3}$ & $\sqrt{17}$ & $0$ & $\frac{1}{\sqrt{34}}$ & $0$
 & $3^{2k}(3n + 1)$
 & Monoclinic(C) \\
\hline
4 &
$^{**}2^9 3$ (8) & 7 & 15 & 16 & -6 & 0 & 0
% & & $\sqrt{15}$ & $4$ & $\sqrt{7}$ & $0$ & $\frac{3}{\sqrt{105}}$ & $0$
% & \multirow{2}{*}{ 3 ($3 u_{+}$) }
 & $4 n + 1$, $4 n + 2$, $8 n + 3$, $2^2 (4 n + 1)$, $2^2 (4 n + 2)$, 
 & Monoclinic(P) \\
   &
$^{*}2^6 3 \cdot 11$ (11) & 7 & 15 & 23 & -6 & -2 & -6
% & & $\sqrt{23}$ & $\sqrt{15}$ & $\sqrt{7}$ & $-\frac{3}{\sqrt{105}}$ & $-\frac{1}{\sqrt{161}}$ & $-\frac{3}{\sqrt{345}}$
 & $3^{2k + 1}(3n + 1)$
 & Triclinic \\
\hline
5 &
$^{**}2^9 3$ (8) & 11 & 11 & 16 & 6 & 8 & 8
% & & $\sqrt{7}$ & $2$ & $4$ & $0$ & $\frac{1}{\sqrt{7}}$ & $0$
% & \multirow{2}{*}{ 3 ($3 u_{+}$) }
 & $4 n + 1$, $4 n + 2$, $8 n + 7$, $2^2 (4 n + 1)$, $2^2 (4 n + 2)$, 
 & Monoclinic(C) \\
   &
$^{**}2^6 3 \cdot 11$ (11) & 11 & 11 & 19 & 6 & 2 & 2
% & & $\sqrt{7}$ & $2$ & $\sqrt{19}$ & $0$ & $\frac{1}{\sqrt{133}}$ & $0$
 & $3^{2k+1}(3n + 1)$
 & Monoclinic(C) \\
\hline
6 &
$^{**}2^5 3^3$ (8) & 8 & 11 & 11 & -4 & -4 & -2
% & & $\sqrt{5}$ & $\sqrt{6}$ & $2\sqrt{2}$ & $0$ & $\frac{1}{\sqrt{10}}$ & $0$
% & \multirow{2}{*}{ 3 ($3 u_{+}$) }
 & \multirow{2}{*}{ $4 n + 1$, $4 n + 2$, $3n + 1$, $3^{2k+1}(3n + 1)$ }
 & Monoclinic(C) \\
   &
$^{**}2^2 3^3 \cdot 11$ (11) & 8 & 11 & 15 & -4 & 0 & -6
% & & $\sqrt{15}$ & $\sqrt{11}$ & $2\sqrt{2}$ & $-\frac{1}{\sqrt{22}}$ & $0$ & $-\frac{3}{\sqrt{165}}$
 & 
 & Triclinic \\
\hline
7 &
$^{**}2^5 3^2$ (8) & 5 & 5 & 12 & -2 & 0 & 0
% & & $\sqrt{3}$ & $\sqrt{2}$ & $2\sqrt{3}$ & $0$ & $0$ & $0$
% & \multirow{2}{*}{ 3 ($u_{+}$) }
 & \multirow{2}{*}{ $4 n + 2$, $4 n + 3$, $3^{2k}(3n + 1)$ }
 & Orthorhombic(C) \\
   &
$^{**}2^2 3^2 11$ (11) & 5 & 5 & 17 & -2 & -2 & -2
% & & $\sqrt{2}$ & $\sqrt{3}$ & $\sqrt{17}$ & $0$ & $\frac{1}{\sqrt{34}}$ & $0$
 & 
 & Monoclinic(C) \\
\hline
8 &
$^{**}2^5 3$ (8) & 4 & 4 & 7 & 0 & -4 & 0
% & & $\sqrt{6}$ & $1$ & $2$ & $0$ & $0$ & $0$
% & \multirow{2}{*}{ 3 ($3 u_{+}$) }
 & \multirow{2}{*}{ $4 n + 1$, $4 n + 2$, $3^{2k+1}(3n + 1)$ }
 & Orthorhombic(C) \\
   &
$^{**}2^2 3\ 11$ (11) & 4 & 7 & 7 & -4 & 0 & -6
% & & $\sqrt{6}$ & $1$ & $\sqrt{7}$ & $0$ & $\frac{3}{\sqrt{42}}$ & $0$
 & 
 & Monoclinic(C) \\
\hline
9 &
$^{**}2^2 3^3 5$ (5) & 5 & 8 & 17 & 4 & 2 & 8
% & & $\sqrt{2}$ & $\sqrt{15}$ & $\sqrt{5}$ & $0$ & $\frac{1}{\sqrt{10}}$ & $0$
% & \multirow{2}{*}{ 3 ($3 u_{+}$) }
 & \multirow{2}{*}{ $4 n + 2$, $4 n + 3$, $3n + 1$, $3^{2k+1}(3n + 1)$ }
 & Monoclinic(C) \\
   &
$^{**}2^5 3^3$ (8) & 5 & 8 & 24 & -4 & 0 & 0
% & & $2\sqrt{2}$ & $2\sqrt{6}$ & $\sqrt{5}$ & $0$ & $\frac{1}{\sqrt{10}}$ & $0$
 & 
 & Monoclinic(P) \\
\hline
10 &
$^{**}2^2 3^2 5$ (5) & 3 & 8 & 8 & 0 & 0 & -4
% & & $\sqrt{5}$ & $\sqrt{3}$ & $\sqrt{3}$ & $0$ & $0$ & $0$
% & \multirow{2}{*}{ 3 ($u_{+}$) }
 & \multirow{2}{*}{ $4 n + 1$, $4 n + 2$, $3^{2n} (3 n + 1)$ }
 & Orthorhombic(C) \\
   &
$^{**}2^5 3^2$ (8) & 3 & 8 & 12 & 0 & 0 & 0
% & & $2\sqrt{3}$ & $2\sqrt{2}$ & $\sqrt{3}$ & $0$ & $0$ & $0$
 & 
 & Orthorhombic(P) \\
\hline
11 &
$^{**}2^2 3\ 5$ (5) & 1 & 4 & 16 & 0 & 0 & -4
% & & $\sqrt{15}$ & $1$ & $1$ & $0$ & $0$ & $0$
% & \multirow{2}{*}{ 3 ($3 u_{+}$) }
 & \multirow{2}{*}{ $4 n + 3$, $4 n + 2$, $3^{2k + 1}(3n + 1)$ }
 & Orthorhombic(C) \\
   &
$^{**}2^5 3$ (8) & 1 & 4 & 24 & 0 & 0 & 0
% & & $2\sqrt{6}$ & $2$ & $1$ & $0$ & $0$ & $0$
 & 
 & Orthorhombic(P) \\
\hline
12 &
$2^{-2} 3^3 5^2$ (25) & 5 & 5 & 8 & -2 & -4 & -1
% & & $2\sqrt{2}$ & $\sqrt{5}$ & $\sqrt{5}$ & $-\frac{1}{5}$ & $-\frac{1}{\sqrt{10}}$ & $-\frac{1}{4\sqrt{10}}$
% & \multirow{2}{*}{ 3 ($3 u_{-}$) }
 & \multirow{2}{*}{ $3 n + 1$, $3^{2k+1}(3 n - 1)$ }
 & Triclinic
%  Genus and spinor representatives:
%    Standard Lattice of rank 3 and degree 3
%    Determinant: 1350
%    Factored Determinant: 2 * 3^3 * 5^2
%    Minimum: 10
%    Kissing Number: 4
%    Inner Product Matrix:
%    [10 -2 -4]
%    [-2 10 -1]
%    [-4 -1 16],
%
%    Standard Lattice of rank 3 and degree 3
%    Determinant: 1350
%    Factored Determinant: 2 * 3^3 * 5^2
%    Minimum: 4
%    Inner Product Matrix:
%    [ 4  2  1]
%    [ 2 10  5]
%    [ 1  5 40],
%
%    Standard Lattice of rank 3 and degree 3
%    Determinant: 1350
%    Factored Determinant: 2 * 3^3 * 5^2
%    Minimum: 6
%    Inner Product Matrix:
%    [ 6 -3 -3]
%    [-3 10  1]
%    [-3  1 28],
%
%    Standard Lattice of rank 3 and degree 3
%    Determinant: 1350
%    Factored Determinant: 2 * 3^3 * 5^2
%    Minimum: 4
%    Inner Product Matrix:
%    [ 4  0 -1]
%    [ 0  6  3]
%    [-1  3 58]
\\
   &
$2^{-2} 3^3 37$ (37) & 5 & 8 & 8 & -4 & -1 & -5
% & & $2\sqrt{2}$ & $2\sqrt{2}$ & $\sqrt{5}$ & $-\frac{1}{\sqrt{10}}$ & $-\frac{1}{4\sqrt{10}}$ & $-\frac{5}{16}$
 & 
 & Triclinic
%  Genus and spinor representatives:
%    Standard Lattice of rank 3 and degree 3
%    Determinant: 1998
%    Factored Determinant: 2 * 3^3 * 37
%    Minimum: 10
%    Kissing Number: 2
%    Inner Product Matrix:
%    [10 -4 -1]
%    [-4 16 -5]
%    [-1 -5 16],
%
%    Standard Lattice of rank 3 and degree 3
%    Determinant: 1998
%    Factored Determinant: 2 * 3^3 * 37
%    Minimum: 4
%    Inner Product Matrix:
%    [ 4 -1 -2]
%    [-1 16 -4]
%    [-2 -4 34],
%
%    Standard Lattice of rank 3 and degree 3
%    Determinant: 1998
%    Factored Determinant: 2 * 3^3 * 37
%    Minimum: 4
%    Inner Product Matrix:
%    [ 4 -2  1]
%    [-2 10  4]
%    [ 1  4 58],
%
%    Standard Lattice of rank 3 and degree 3
%    Determinant: 1998
%    Factored Determinant: 2 * 3^3 * 37
%    Minimum: 6
%    Inner Product Matrix:
%    [ 6  0 -3]
%    [ 0 18  6]
%    [-3  6 22],
%
%    Standard Lattice of rank 3 and degree 3
%    Determinant: 1998
%    Factored Determinant: 2 * 3^3 * 37
%    Minimum: 6
%    Inner Product Matrix:
%    [ 6  3 -3]
%    [ 3  6 -3]
%    [-3 -3 76]
\\
\hline
13 &
$2^{-2} 3^2 5^2$ (25) & 3 & 4 & 7 & 3 & 3 & 4
% & & $\frac{\sqrt{13}}{2}$ & $\frac{\sqrt{3}}{2}$ & $\sqrt{7}$ & $0$ & $\frac{4}{\sqrt{91}}$ & $0$
% & \multirow{2}{*}{ 3 ($u_{-}$) }
 & \multirow{2}{*}{ $3^{2k}(3 n - 1)$ }
 & Monoclinic(C)
%  Genus and spinor representatives:
%    Standard Lattice of rank 3 and degree 3
%    Determinant: 450
%    Factored Determinant: 2 * 3^2 * 5^2
%    Minimum: 6
%    Kissing Number: 2
%    Inner Product Matrix:
%    [ 6 -3 -3]
%    [-3  8  4]
%    [-3  4 14],
%
%    Standard Lattice of rank 3 and degree 3
%    Determinant: 450
%    Factored Determinant: 2 * 3^2 * 5^2
%    Minimum: 2
%    Inner Product Matrix:
%    [ 2  0 -1]
%    [ 0 12  3]
%    [-1  3 20],
%
%    Standard Lattice of rank 3 and degree 3
%    Determinant: 450
%    Factored Determinant: 2 * 3^2 * 5^2
%    Minimum: 6
%    Inner Product Matrix:
%    [ 6  0  0]
%    [ 0  6 -3]
%    [ 0 -3 14]
\\
   &
$2^{-2} 3^2 37$ (37) & 3 & 4 & 7 & 0 & 0 & -1
% & & $\sqrt{7}$ & $\sqrt{3}$ & $2$ & $0$ & $\frac{1}{4\sqrt{7}}$ & $0$
 & 
 & Monoclinic(P)
%  Genus and spinor representatives:
%    Standard Lattice of rank 3 and degree 3
%    Determinant: 666
%    Factored Determinant: 2 * 3^2 * 37
%    Minimum: 6
%    Kissing Number: 2
%    Inner Product Matrix:
%    [ 6  0  0]
%    [ 0  8 -1]
%    [ 0 -1 14],
%
%    Standard Lattice of rank 3 and degree 3
%    Determinant: 666
%    Factored Determinant: 2 * 3^2 * 37
%    Minimum: 6
%    Inner Product Matrix:
%    [ 6 -3 -3]
%    [-3 12 -3]
%    [-3 -3 14],
%
%    Standard Lattice of rank 3 and degree 3
%    Determinant: 666
%    Factored Determinant: 2 * 3^2 * 37
%    Minimum: 6
%    Inner Product Matrix:
%    [ 6  0 -3]
%    [ 0  6  0]
%    [-3  0 20],
%
%    Standard Lattice of rank 3 and degree 3
%    Determinant: 666
%    Factored Determinant: 2 * 3^2 * 37
%    Minimum: 2
%    Inner Product Matrix:
%    [ 2  0 -1]
%    [ 0  6  0]
%    [-1  0 56],
%
%    Standard Lattice of rank 3 and degree 3
%    Determinant: 666
%    Factored Determinant: 2 * 3^2 * 37
%    Minimum: 2
%    Inner Product Matrix:
%    [  2  -1   0]
%    [ -1   2   0]
%    [  0   0 222]
\\
\hline
14 &
$2^{-2} 3\ 5^2$ (25) & 1 & 4 & 5 & 0 & -1 & -1
% & & $\frac{\sqrt{19}}{2}$ & $\frac{1}{2}$ & $2$ & $0$ & $\frac{1}{2\sqrt{19}}$ & $0$
% & \multirow{2}{*}{ 3 ($3 u_{-}$) }
 & \multirow{2}{*}{ $3^{2k+1}(3 n - 1)$ }
 & Monoclinic(C)
%  Genus and spinor representatives:
%    Standard Lattice of rank 3 and degree 3
%    Determinant: 150
%    Factored Determinant: 2 * 3 * 5^2
%    Minimum: 2
%    Kissing Number: 2
%    Inner Product Matrix:
%    [ 2  0 -1]
%    [ 0  8 -1]
%    [-1 -1 10],
%
%    Standard Lattice of rank 3 and degree 3
%    Determinant: 150
%    Factored Determinant: 2 * 3 * 5^2
%    Minimum: 2
%    Inner Product Matrix:
%    [ 2  0  0]
%    [ 0  2 -1]
%    [ 0 -1 38],
%
%    Standard Lattice of rank 3 and degree 3
%    Determinant: 150
%    Factored Determinant: 2 * 3 * 5^2
%    Minimum: 4
%    Inner Product Matrix:
%    [ 4  0 -1]
%    [ 0  6 -3]
%    [-1 -3  8]
\\
   &
$2^{-2} 3\ 37$ (37) & 1 & 4 & 7 & 0 & 0 & -1
% & & $\sqrt{7}$ & $1$ & $2$ & $0$ & $\frac{1}{4\sqrt{7}}$ & $0$
 & 
 & Monoclinic(P)
%  Genus and spinor representatives:
%    Standard Lattice of rank 3 and degree 3
%    Determinant: 222
%    Factored Determinant: 2 * 3 * 37
%    Minimum: 2
%    Kissing Number: 2
%    Inner Product Matrix:
%    [ 2  0  0]
%    [ 0  8 -1]
%    [ 0 -1 14],
%
%    Standard Lattice of rank 3 and degree 3
%    Determinant: 222
%    Factored Determinant: 2 * 3 * 37
%    Minimum: 2
%    Inner Product Matrix:
%    [ 2  0  0]
%    [ 0  6 -3]
%    [ 0 -3 20],
%
%    Standard Lattice of rank 3 and degree 3
%    Determinant: 222
%    Factored Determinant: 2 * 3 * 37
%    Minimum: 2
%    Inner Product Matrix:
%    [ 2  0 -1]
%    [ 0  2  0]
%    [-1  0 56],
%
%    Standard Lattice of rank 3 and degree 3
%    Determinant: 222
%    Factored Determinant: 2 * 3 * 37
%    Minimum: 2
%    Inner Product Matrix:
%    [ 2 -1  0]
%    [-1  4 -1]
%    [ 0 -1 32],
%
%    Standard Lattice of rank 3 and degree 3
%    Determinant: 222
%    Factored Determinant: 2 * 3 * 37
%    Minimum: 2
%    Inner Product Matrix:
%    [ 2 -1  0]
%    [-1  2  0]
%    [ 0  0 74]
 \\
\hline
15 &
$2 \cdot 3^4$ (1) & 4 & 7 & 7 & 2 & 2 & 5
% & & $\frac{\sqrt{19}}{2}$ & $\frac{3}{2}$ & $2$ & $0$ & $\frac{1}{\sqrt{19}}$ & $0$
% & \multirow{2}{*}{ 2 ($2 u_{7}$) }
 & \multirow{2}{*}{ $4 n + 2$, $3 n - 1$, $3 (3 n \pm 1)$, $2^{2k+1}(8n + 7)$ }
 & Monoclinic(C)
%  Genus and spinor representatives:
%    Standard Lattice of rank 3 and degree 3
%    Determinant: 1296
%    Factored Determinant: 2^4 * 3^4
%    Minimum: 8
%    Kissing Number: 2
%    Inner Product Matrix:
%    [ 8  2  2]
%    [ 2 14  5]
%    [ 2  5 14],
%
%    Standard Lattice of rank 3 and degree 3
%    Determinant: 1296
%    Factored Determinant: 2^4 * 3^4
%    Minimum: 2
%    Inner Product Matrix:
%    [ 2  1 -1]
%    [ 1 14  4]
%    [-1  4 50]
 \\
   &
$2^3 3^4$ (4) & 4 & 7 & 25 & -2 & -2 & -4
% & & $5$ & $\sqrt{7}$ & $2$ & $-\frac{1}{2\sqrt{7}}$ & $-\frac{1}{10}$ & $-\frac{2}{5\sqrt{7}}$
 & 
 & Triclinic
%  Genus and spinor representatives:
%    Standard Lattice of rank 3 and degree 3
%    Determinant: 5184
%    Factored Determinant: 2^6 * 3^4
%    Inner Product Matrix:
%    [ 8 -2 -2]
%    [-2 14 -4]
%    [-2 -4 50],
%
%    Standard Lattice of rank 3 and degree 3
%    Determinant: 5184
%    Factored Determinant: 2^6 * 3^4
%    Inner Product Matrix:
%    [  2   0   0]
%    [  0  54 -18]
%    [  0 -18  54],
%
%    Standard Lattice of rank 3 and degree 3
%    Determinant: 5184
%    Factored Determinant: 2^6 * 3^4
%    Inner Product Matrix:
%    [ 8  0 -4]
%    [ 0 18  0]
%    [-4  0 38]
\\
\hline
16 &
$2^{-1} 3^5$ (1) & 2 & 2 & 41 & 2 & 2 & 1
% & & $\frac{9\sqrt{2}}{2}$ & $\frac{\sqrt{6}}{2}$ & $\frac{\sqrt{2}}{2}$ & $0$ & $0$ & $0$
% & \multirow{2}{*}{ 3 ($3 u_{+}$) }
 & \multirow{2}{*}{ $3 n + 1$, $3^2(3 n + 1)$ $3^{2k+1}(3n + 1)$ }
 & Orthorhombic(F)
%  Genus and spinor representatives:
%    Standard Lattice of rank 3 and degree 3
%    Determinant: 972
%    Factored Determinant: 2^2 * 3^5
%    Minimum: 4
%    Kissing Number: 6
%    Inner Product Matrix:
%    [ 4 -2 -1]
%    [-2  4 -1]
%    [-1 -1 82],
%
%    Standard Lattice of rank 3 and degree 3
%    Determinant: 972
%    Factored Determinant: 2^2 * 3^5
%    Minimum: 4
%    Inner Product Matrix:
%    [ 4 -1  2]
%    [-1 10 -5]
%    [ 2 -5 28],
%
%    Standard Lattice of rank 3 and degree 3
%    Determinant: 972
%    Factored Determinant: 2^2 * 3^5
%    Minimum: 10
%    Inner Product Matrix:
%    [10  3  0]
%    [ 3 12  6]
%    [ 0  6 12],
%
%    Standard Lattice of rank 3 and degree 3
%    Determinant: 972
%    Factored Determinant: 2^2 * 3^5
%    Minimum: 4
%    Inner Product Matrix:
%    [ 4  0 -2]
%    [ 0 12 -3]
%    [-2 -3 22]
\\
   &
$2 \cdot 3^5$ (4) & 2 & 6 & 41 & 0 & -1 & -3
% & & $\sqrt{41}$ & $\sqrt{6}$ & $\sqrt{2}$ & $0$ & $-\frac{1}{2\sqrt{82}}$ & $-\frac{3}{2\sqrt{246}}$
 &
 & Triclinic
\\
\hline
17 &
$2^{-1} 3^4$ (1) & 2 & 2 & 14 & 2 & 2 & 1
% & & $\frac{3\sqrt{6}}{2}$ & $\frac{\sqrt{6}}{2}$ & $\frac{\sqrt{2}}{2}$ & $0$ & $0$ & $0$
% & \multirow{2}{*}{ 3 ($u_{+}$) }
 & \multirow{2}{*}{ $3 (3 n + 1)$, $3^{2k}(3n + 1)$ }
 & Orthorhombic(F)
%  Genus and spinor representatives:
%    Standard Lattice of rank 3 and degree 3
%    Determinant: 324
%    Factored Determinant: 2^2 * 3^4
%    Minimum: 4
%    Kissing Number: 6
%    Inner Product Matrix:
%    [ 4 -2 -1]
%    [-2  4 -1]
%    [-1 -1 28],
%
%    Standard Lattice of rank 3 and degree 3
%    Determinant: 324
%    Factored Determinant: 2^2 * 3^4
%    Minimum: 4
%    Inner Product Matrix:
%    [ 4  1  1]
%    [ 1 10  4]
%    [ 1  4 10]
\\
   &
$2 \cdot 3^4$ (4) & 2 & 6 & 14 & 0 & -1 & -3
% & & $\sqrt{14}$ & $\sqrt{6}$ & $\sqrt{2}$ & $0$ & $-\frac{1}{4\sqrt{7}}$ & $-\frac{3}{4\sqrt{21}}$
 & 
 & Triclinic
%  Genus and spinor representatives:
%    Standard Lattice of rank 3 and degree 3
%    Determinant: 1296
%    Factored Determinant: 2^4 * 3^4
%    Minimum: 4
%    Kissing Number: 2
%    Inner Product Matrix:
%    [ 4  0 -1]
%    [ 0 12 -3]
%    [-1 -3 28],
%
%    Standard Lattice of rank 3 and degree 3
%    Determinant: 1296
%    Factored Determinant: 2^4 * 3^4
%    Minimum: 10
%    Inner Product Matrix:
%    [10  3 -5]
%    [ 3 12 -6]
%    [-5 -6 16],
%
%    Standard Lattice of rank 3 and degree 3
%    Determinant: 1296
%    Factored Determinant: 2^4 * 3^4
%    Minimum: 4
%    Inner Product Matrix:
%    [ 4  1 -1]
%    [ 1 10  2]
%    [-1  2 34]
 \\
\hline
18 &
$2^{-1} 3^3$ (1) & 2 & 2 & 4 & 0 & -2 & -1
% & & $\frac{\sqrt{14}}{2}$ & $\frac{\sqrt{2}}{2}$ & $\sqrt{2}$ & $0$ & $\frac{1}{2\sqrt{7}}$ & $0$
% & \multirow{2}{*}{ 3 ($3 u_{+}$) }
 & \multirow{2}{*}{ $3 (3 n - 1)$, $3^{2k+1}(3n + 1)$ }
 & Monoclinic(C)
%  Genus and spinor representatives:
%    Standard Lattice of rank 3 and degree 3
%    Determinant: 108
%    Factored Determinant: 2^2 * 3^3
%    Minimum: 4
%    Kissing Number: 4
%    Inner Product Matrix:
%    [ 4  0 -2]
%    [ 0  4 -1]
%    [-2 -1  8],
%
%    Standard Lattice of rank 3 and degree 3
%    Determinant: 108
%    Factored Determinant: 2^2 * 3^3
%    Minimum: 2
%    Inner Product Matrix:
%    [ 2 -1  1]
%    [-1  4 -1]
%    [ 1 -1 16],
%
%    Standard Lattice of rank 3 and degree 3
%    Determinant: 108
%    Factored Determinant: 2^2 * 3^3
%    Minimum: 2
%    Inner Product Matrix:
%    [ 2  0 -1]
%    [ 0  2  1]
%    [-1  1 28]
\\
   &
$2 \cdot 3^3$ (4) & 2 & 4 & 8 & 1 & 1 & 4
% & & $2\sqrt{2}$ & $2$ & $\sqrt{2}$ & $\frac{1}{4\sqrt{2}}$ & $\frac{1}{8}$ & $\frac{1}{2\sqrt{2}}$
 & 
 & Triclinic
%  Genus and spinor representatives:
%    Standard Lattice of rank 3 and degree 3
%    Determinant: 432
%    Factored Determinant: 2^4 * 3^3
%    Minimum: 4
%    Kissing Number: 2
%    Inner Product Matrix:
%    [ 4 -1 -1]
%    [-1  8  4]
%    [-1  4 16],
%
%    Standard Lattice of rank 3 and degree 3
%    Determinant: 432
%    Factored Determinant: 2^4 * 3^3
%    Minimum: 2
%    Inner Product Matrix:
%    [ 2  0 -1]
%    [ 0 14  1]
%    [-1  1 16],
%
%    Standard Lattice of rank 3 and degree 3
%    Determinant: 432
%    Factored Determinant: 2^4 * 3^3
%    Minimum: 8
%    Inner Product Matrix:
%    [8 0 1]
%    [0 8 3]
%    [1 3 8],
%
%    Standard Lattice of rank 3 and degree 3
%    Determinant: 432
%    Factored Determinant: 2^4 * 3^3
%    Minimum: 4
%    Inner Product Matrix:
%    [ 4 -2  1]
%    [-2  8  1]
%    [ 1  1 16],
%
%    Standard Lattice of rank 3 and degree 3
%    Determinant: 432
%    Factored Determinant: 2^4 * 3^3
%    Minimum: 2
%    Inner Product Matrix:
%    [ 2  0  1]
%    [ 0  8  2]
%    [ 1  2 28],
%
%    Standard Lattice of rank 3 and degree 3
%    Determinant: 432
%    Factored Determinant: 2^4 * 3^3
%    Minimum: 2
%    Inner Product Matrix:
%    [ 2 -1  0]
%    [-1  4  1]
%    [ 0  1 62]
\\
\hline
19 &
$2 \cdot 3^4$ (1) & 5 & 7 & 7 & 5 & 1 & 6
% & & $\sqrt{7}$ & $\sqrt{7}$ & $\sqrt{5}$ & $-\frac{5}{2\sqrt{35}}$ & $\frac{1}{2\sqrt{35}}$ & $-\frac{3}{7}$
% & \multirow{2}{*}{ 2 ($2 u_{7}$) }
 & \multirow{2}{*}{ $4 n + 2$, $3 (3 n \pm 1)$, $3^3 (3 n \pm 1)$, $2^{2k+1}(8n + 7)$ }
 & Triclinic
%  Genus and spinor representatives:
%    Standard Lattice of rank 3 and degree 3
%    Determinant: 1296
%    Factored Determinant: 2^4 * 3^4
%    Minimum: 10
%    Kissing Number: 2
%    Inner Product Matrix:
%    [10 -5 -1]
%    [-5 14 -5]
%    [-1 -5 14],
%
%    Standard Lattice of rank 3 and degree 3
%    Determinant: 1296
%    Factored Determinant: 2^4 * 3^4
%    Minimum: 2
%    Inner Product Matrix:
%    [ 2  0  1]
%    [ 0  8  2]
%    [ 1  2 82],
%
%    Standard Lattice of rank 3 and degree 3
%    Determinant: 1296
%    Factored Determinant: 2^4 * 3^4
%    Minimum: 2
%    Inner Product Matrix:
%    [ 2 -1 -1]
%    [-1 10  4]
%    [-1  4 70],
%
%    Standard Lattice of rank 3 and degree 3
%    Determinant: 1296
%    Factored Determinant: 2^4 * 3^4
%    Minimum: 2
%    Inner Product Matrix:
%    [ 2 -1 -1]
%    [-1 22  9]
%    [-1  9 34],
%
%    Standard Lattice of rank 3 and degree 3
%    Determinant: 1296
%    Factored Determinant: 2^4 * 3^4
%    Minimum: 10
%    Inner Product Matrix:
%    [10  3  0]
%    [ 3 10  4]
%    [ 0  4 16],
%
%    Standard Lattice of rank 3 and degree 3
%    Determinant: 1296
%    Factored Determinant: 2^4 * 3^4
%    Minimum: 8
%    Inner Product Matrix:
%    [ 8 -2  0]
%    [-2 10  3]
%    [ 0  3 18]
\\
   &
$2 \cdot 3^4$ (1) & 5 & 5 & 8 & -3 & -4 & 0
% & & $2\sqrt{2}$ & $\sqrt{5}$ & $\sqrt{5}$ & $-\frac{3}{10}$ & $-\frac{1}{\sqrt{10}}$ & $0$
 &
 & Triclinic
%  Genus and spinor representatives:
%    Standard Lattice of rank 3 and degree 3
%    Determinant: 1296
%    Factored Determinant: 2^4 * 3^4
%    Minimum: 10
%    Kissing Number: 4
%    Inner Product Matrix:
%    [10 -3 -4]
%    [-3 10  0]
%    [-4  0 16],
%
%    Standard Lattice of rank 3 and degree 3
%    Determinant: 1296
%    Factored Determinant: 2^4 * 3^4
%    Minimum: 8
%    Inner Product Matrix:
%    [ 8 -2  0]
%    [-2 10 -3]
%    [ 0 -3 18],
%
%    Standard Lattice of rank 3 and degree 3
%    Determinant: 1296
%    Factored Determinant: 2^4 * 3^4
%    Minimum: 10
%    Inner Product Matrix:
%    [10  5  1]
%    [ 5 14 -5]
%    [ 1 -5 14],
%
%    Standard Lattice of rank 3 and degree 3
%    Determinant: 1296
%    Factored Determinant: 2^4 * 3^4
%    Minimum: 2
%    Inner Product Matrix:
%    [ 2  0  1]
%    [ 0  8 -2]
%    [ 1 -2 82],
%
%    Standard Lattice of rank 3 and degree 3
%    Determinant: 1296
%    Factored Determinant: 2^4 * 3^4
%    Minimum: 2
%    Inner Product Matrix:
%    [ 2 -1 -1]
%    [-1 22 -8]
%    [-1 -8 34],
%
%    Standard Lattice of rank 3 and degree 3
%    Determinant: 1296
%    Factored Determinant: 2^4 * 3^4
%    Minimum: 2
%    Inner Product Matrix:
%    [ 2 -1 -1]
%    [-1 10 -3]
%    [-1 -3 70]
\\
\hline
20 &
$2 \cdot 29$ (1) & 3 & 5 & 5 & 3 & 1 & 3
% & & $\sqrt{5}$ & $\sqrt{5}$ & $\sqrt{3}$ & $\frac{3}{2\sqrt{15}}$ & $\frac{1}{2\sqrt{15}}$ & $\frac{3}{10}$
% & \multirow{2}{*}{ 2 ($2 u_{3}$) }
 & \multirow{2}{*}{ $4 n + 2$, $2^{2k+1}(8n + 3)$ }
 & Triclinic
%  Genus and spinor representatives:
%    Standard Lattice of rank 3 and degree 3
%    Determinant: 464
%    Factored Determinant: 2^4 * 29
%    Minimum: 6
%    Kissing Number: 2
%    Inner Product Matrix:
%    [ 6 -3 -1]
%    [-3 10 -2]
%    [-1 -2 10],
%
%    Standard Lattice of rank 3 and degree 3
%    Determinant: 464
%    Factored Determinant: 2^4 * 29
%    Minimum: 2
%    Inner Product Matrix:
%    [ 2  0 -1]
%    [ 0  8  2]
%    [-1  2 30],
%
%    Standard Lattice of rank 3 and degree 3
%    Determinant: 464
%    Factored Determinant: 2^4 * 29
%    Minimum: 6
%    Inner Product Matrix:
%    [ 6  1  1]
%    [ 1  6  2]
%    [ 1  2 14]
\\
   &
$2 \cdot 29$ (1) & 3 & 3 & 7 & 1 & 2 & 1
% & & $\sqrt{7}$ & $\sqrt{3}$ & $\sqrt{3}$ & $\frac{1}{6}$ & $\frac{1}{\sqrt{21}}$ & $\frac{1}{2\sqrt{21}}$
 &
 & Triclinic
%  Genus and spinor representatives:
%    Standard Lattice of rank 3 and degree 3
%    Determinant: 464
%    Factored Determinant: 2^4 * 29
%    Minimum: 6
%    Kissing Number: 4
%    Inner Product Matrix:
%    [ 6  1  2]
%    [ 1  6  1]
%    [ 2  1 14],
%
%    Standard Lattice of rank 3 and degree 3
%    Determinant: 464
%    Factored Determinant: 2^4 * 29
%    Minimum: 6
%    Inner Product Matrix:
%    [ 6 -3  1]
%    [-3 10 -3]
%    [ 1 -3 10],
%
%    Standard Lattice of rank 3 and degree 3
%    Determinant: 464
%    Factored Determinant: 2^4 * 29
%    Minimum: 2
%    Inner Product Matrix:
%    [ 2  0 -1]
%    [ 0  8 -2]
%    [-1 -2 30]
\\
\hline
21 &
$^{*}3^3$ (1) & 1 & 4 & 7 & 0 & -1 & 0
% & & $\frac{3\sqrt{3}}{2}$ & $\frac{1}{2}$ & $2$ & $0$ & $0$ & $0$
% & \multirow{2}{*}{ 3 ($3 u_{-}$) }
 & \multirow{2}{*}{ $4 n + 2$, $3 (3 n + 1)$, $3^{2k+1}(3 n - 1)$ }
 & Orthorhombic(C)
\\
   &
$^{*}3^3$ (1) & 1 & 5 & 7 & 1 & 1 & 5
% & & $\frac{\sqrt{19}}{2}$ & $\frac{1}{2}$ & $\sqrt{7}$ & $0$ & $\frac{5}{\sqrt{133}}$ & $0$
 &
 & Monoclinic(C)
\\
\hline
22 &
$^{*}2^{-2} 3^3$ (1) & 1 & 1 & 7 & 0 & -1 & 0
% & & $\frac{3\sqrt{3}}{2}$ & $\frac{1}{2}$ & $1$ & $0$ & $0$ & $0$
% & \multirow{2}{*}{ 3 ($3 u_{-}$) }
 & \multirow{2}{*}{ $3 (3 n + 1)$, $3^{2k+1}(3 n - 1)$ }
 & Orthorhombic(C) \\
   &
$^{*}2^{-2} 3^3$ (1) & 1 & 2 & 4 & -1 & 0 & -1
% & & $\frac{\sqrt{7}}{2}$ & $\frac{1}{2}$ & $2$ & $0$ & $\frac{1}{2\sqrt{7}}$ & $0$
 &
 & Monoclinic(C) \\
\hline
23 &
$2^{-2} 3^2 59$ (59) & 5 & 5 & 6 & -2 & -3 & 0
% & & $\sqrt{6}$ & $\sqrt{5}$ & $\sqrt{5}$ & $-\frac{1}{5}$ & $-\frac{3}{2\sqrt{30}}$ & $0$
% & \multirow{2}{*}{ 3 ($u_{+}$) }
 & \multirow{2}{*}{ $3^{2k}(3n + 1)$ }
 & Triclinic
%  Genus and spinor representatives:
%    Standard Lattice of rank 3 and degree 3
%    Determinant: 1062
%    Factored Determinant: 2 * 3^2 * 59
%    Minimum: 10
%    Kissing Number: 4
%    Inner Product Matrix:
%    [10 -2 -3]
%    [-2 10  0]
%    [-3  0 12],
%
%    Standard Lattice of rank 3 and degree 3
%    Determinant: 1062
%    Factored Determinant: 2 * 3^2 * 59
%    Minimum: 4
%    Inner Product Matrix:
%    [ 4  0 -1]
%    [ 0  6  3]
%    [-1  3 46],
%
%    Standard Lattice of rank 3 and degree 3
%    Determinant: 1062
%    Factored Determinant: 2 * 3^2 * 59
%    Minimum: 4
%    Inner Product Matrix:
%    [ 4  1 -1]
%    [ 1 10  2]
%    [-1  2 28]
\\
   &
$2^{-2} 3^2 71$ (71) & 5 & 5 & 8 & -4 & -2 & -1
% & & $2\sqrt{2}$ & $\sqrt{5}$ & $\sqrt{5}$ & $-\frac{2}{5}$ & $-\frac{1}{2\sqrt{10}}$ & $-\frac{1}{4\sqrt{10}}$
 &
 & Triclinic
%  Genus and spinor representatives:
%    Standard Lattice of rank 3 and degree 3
%    Determinant: 1278
%    Factored Determinant: 2 * 3^2 * 71
%    Minimum: 10
%    Kissing Number: 4
%    Inner Product Matrix:
%    [10 -4 -2]
%    [-4 10 -1]
%    [-2 -1 16],
%
%    Standard Lattice of rank 3 and degree 3
%    Determinant: 1278
%    Factored Determinant: 2 * 3^2 * 71
%    Minimum: 4
%    Inner Product Matrix:
%    [ 4 -1  2]
%    [-1 10 -2]
%    [ 2 -2 34],
%
%    Standard Lattice of rank 3 and degree 3
%    Determinant: 1278
%    Factored Determinant: 2 * 3^2 * 71
%    Minimum: 6
%    Inner Product Matrix:
%    [ 6  3 -3]
%    [ 3 10 -5]
%    [-3 -5 28],
%
%    Standard Lattice of rank 3 and degree 3
%    Determinant: 1278
%    Factored Determinant: 2 * 3^2 * 71
%    Minimum: 6
%    Inner Product Matrix:
%    [ 6 -3 -3]
%    [-3 12  3]
%    [-3  3 22]
\\
\hline
24 &
$2^{-2} 3 \cdot 59$ (59) & 2 & 4 & 7 & 1 & 2 & 4 
%& & $\sqrt{7}$ & $2$ & $\sqrt{2}$ & $\frac{1}{4\sqrt{2}}$ & $\frac{1}{\sqrt{14}}$ & $\frac{1}{\sqrt{7}}$
% & \multirow{2}{*}{ 3 ($3 u_{+}$) }
 & \multirow{2}{*}{ $3^{2k+1}(3n + 1)$ }
 & Triclinic
%  Genus and spinor representatives:
%    Standard Lattice of rank 3 and degree 3
%    Determinant: 354
%    Factored Determinant: 2 * 3 * 59
%    Minimum: 4
%    Kissing Number: 2
%    Inner Product Matrix:
%    [ 4 -1 -2]
%    [-1  8  4]
%    [-2  4 14],
%
%    Standard Lattice of rank 3 and degree 3
%    Determinant: 354
%    Factored Determinant: 2 * 3 * 59
%    Minimum: 2
%    Inner Product Matrix:
%    [ 2  0 -1]
%    [ 0 12  3]
%    [-1  3 16],
%
%    Standard Lattice of rank 3 and degree 3
%    Determinant: 354
%    Factored Determinant: 2 * 3 * 59
%    Minimum: 4
%    Inner Product Matrix:
%    [ 4 -1  0]
%    [-1 10  3]
%    [ 0  3 10]
\\
   &
$2^{-2} 3 \cdot 71$ (71) & 2 & 4 & 7 & -1 & -1 & 0
% & & $\sqrt{7}$ & $2$ & $\sqrt{2}$ & $-\frac{1}{4\sqrt{2}}$ & $-\frac{1}{2\sqrt{14}}$ & $0$
 &
 & Triclinic
%  Genus and spinor representatives:
%    Standard Lattice of rank 3 and degree 3
%    Determinant: 426
%    Factored Determinant: 2 * 3 * 71
%    Minimum: 4
%    Kissing Number: 2
%    Inner Product Matrix:
%    [ 4 -1 -1]
%    [-1  8  0]
%    [-1  0 14],
%
%    Standard Lattice of rank 3 and degree 3
%    Determinant: 426
%    Factored Determinant: 2 * 3 * 71
%    Minimum: 2
%    Inner Product Matrix:
%    [ 2  0 -1]
%    [ 0 14 -2]
%    [-1 -2 16],
%
%    Standard Lattice of rank 3 and degree 3
%    Determinant: 426
%    Factored Determinant: 2 * 3 * 71
%    Minimum: 2
%    Inner Product Matrix:
%    [ 2 -1 -1]
%    [-1  4 -1]
%    [-1 -1 62],
%
%    Standard Lattice of rank 3 and degree 3
%    Determinant: 426
%    Factored Determinant: 2 * 3 * 71
%    Minimum: 4
%    Inner Product Matrix:
%    [ 4 -1 -2]
%    [-1 10  5]
%    [-2  5 14]
 \\
\hline
25 &
$3^4$ (4) & 4 & 4 & 6 & -2 & -3 & 0
% & & $\sqrt{6}$ & $2$ & $2$ & $-\frac{1}{4}$ & $-\frac{3}{4\sqrt{6}}$ & $0$
% & \multirow{2}{*}{ 3 ($u_{-}$) }
 & \multirow{2}{*}{ $3(3 n + 1)$, $3^{2k}(3 n - 1)$ }
 & Triclinic
%  Genus and spinor representatives:
%    Standard Lattice of rank 3 and degree 3
%    Determinant: 648
%    Factored Determinant: 2^3 * 3^4
%    Minimum: 8
%    Kissing Number: 4
%    Inner Product Matrix:
%    [ 8 -2 -3]
%    [-2  8  0]
%    [-3  0 12],
%
%    Standard Lattice of rank 3 and degree 3
%    Determinant: 648
%    Factored Determinant: 2^3 * 3^4
%    Minimum: 2
%    Inner Product Matrix:
%    [ 2 -1 -1]
%    [-1  8  2]
%    [-1  2 44]
\\
   &
$2^{-2} 3^4 7$ (7) & 4 & 6 & 7 & 3 & 2 & 3
% & & $\sqrt{7}$ & $\sqrt{6}$ & $2$ & $\frac{3}{4\sqrt{6}}$ & $\frac{1}{2\sqrt{7}}$ & $\frac{3}{2\sqrt{42}}$
 & 
 & Triclinic
%  Genus and spinor representatives:
%    Standard Lattice of rank 3 and degree 3
%    Determinant: 1134
%    Factored Determinant: 2 * 3^4 * 7
%    Minimum: 8
%    Kissing Number: 2
%    Inner Product Matrix:
%    [ 8  3  2]
%    [ 3 12  3]
%    [ 2  3 14],
%
%    Standard Lattice of rank 3 and degree 3
%    Determinant: 1134
%    Factored Determinant: 2 * 3^4 * 7
%    Minimum: 2
%    Inner Product Matrix:
%    [ 2 -1 -1]
%    [-1 14  5]
%    [-1  5 44],
%
%    Standard Lattice of rank 3 and degree 3
%    Determinant: 1134
%    Factored Determinant: 2 * 3^4 * 7
%    Minimum: 2
%    Inner Product Matrix:
%    [ 2  0  0]
%    [ 0 12  3]
%    [ 0  3 48]
\end{tabular}
\end{scriptsize}
\footnotetext[1]{
The mark $**$ indicates that the form provides the only class in its genus. Therefore, it is regular.
%\IE it represents a number $a$ iff it does locally at every prime.
}
\footnotetext[2]{
The mark $*$ indicates that the form is regular, although its genus consists of more than one class.
}
\footnotetext[3]{
The notation $*!$ indicates that the form  is one of the 14 quadratic forms that may be regular.  The regularity has been proved only under the generalized Riemann Hypothesis  \cite{Oliver2014}.
}
\end{minipage}
\end{table}

\begin{table}[htbp]
\caption{Ternary positive-definite quadratic forms with the same representations over $\IntegerRing$}
\label{Forty-nine groups of ternary positive definite quadratic forms representing the same numbers(2/3)}
\begin{scriptsize}
\begin{tabular}{p{1mm}p{20mm}p{4mm}p{4mm}p{4mm} p{4mm}p{4mm}p{4mm}%p{22mm}
	% p{0mm} p{5mm}p{5mm}p{5mm} p{5mm}p{5mm}p{5mm}
	lr}
 \hline
   No.
	& Determinant (Ratio) & $s_{11}$ & $s_{22}$ & $s_{33}$ & $s_{12}$ & $s_{13}$ & $s_{23}$
%	& & $a^*$ & $b^*$ & $c^*$ & $\cos \alpha^*$ & $\cos \beta^*$ & $\cos \gamma^*$
%	& Anisotropic prime\ ($u \in \RationalField^\times/(\RationalField^\times)^2$, $\not\subset q_{\RationalField}(S)$)
	& Integers not represented by the genus
	& Bravais Type\\
\hline
 26  &
${}^{**}3^3 5$ (5) & 5 & 5 & 8 & 5 & 4 & 2
% & & $\frac{\sqrt{5}}{2}$ & $\frac{\sqrt{15}}{2}$ & $2\sqrt{2}$ & $0$ & $\frac{1}{\sqrt{10}}$ & $0$
% & \multirow{3}{*}{ 3 ($3 u_{+}$) }
 & \multirow{3}{*}{ $4 n + 2$, $3 n + 1$, $3^{2k+1}(3n + 1)$ }
 & Monoclinic(C) 
\\
   &
${}^{**}2^3 3^3$ (8) & 5 & 8 & 8 & 4 & 2 & 8
% & & $\sqrt{2}$ & $\sqrt{6}$ & $\sqrt{5}$ & $0$ & $\frac{1}{\sqrt{10}}$ & $0$
 &
 & Monoclinic(C)
 \\
   &
${}^{**}3^3 11$ (11) & 5 & 8 & 9 & 2 & 3 & 6
% & & $3$ & $2\sqrt{2}$ & $\sqrt{5}$ & $\frac{1}{2\sqrt{10}}$ & $\frac{1}{2\sqrt{5}}$ & $\frac{1}{2\sqrt{2}}$
 & 
 & Triclinic \\
\hline
 27 &
${}^{**}3^2 5$ (5) & 3 & 5 & 5 & 3 & 3 & 5
% & & $\sqrt{3}$ & $\frac{\sqrt{5}}{2}$ & $\frac{\sqrt{3}}{2}$ & $0$ & $0$ & $0$
% & \multirow{3}{*}{3 ($u_{+}$) }
 & \multirow{3}{*}{ $4 n + 2$, $3^{2k}(3n + 1)$ }
 & Orthorhombic(I) \\
   &
${}^{**}2^3 3^2$ (8) & 3 & 5 & 5 & 0 & 0 & -2
% & & $\sqrt{3}$ & $\sqrt{2}$ & $\sqrt{3}$ & $0$ & $0$ & $0$
 &
 & Orthorhombic(C) \\
   &
${}^{**}3^2 11$ (11) & 3 & 5 & 8 & -3 & 0 & -2
% & & $\frac{\sqrt{17}}{2}$ & $\frac{\sqrt{3}}{2}$ & $2\sqrt{2}$ & $0$ & $\frac{1}{\sqrt{34}}$ & $0$
 &
 & Monoclinic(C) \\
\hline
 28  &
${}^{**}3 \cdot 5$ (5) & 1 & 4 & 5 & 0 & -1 & -4
% & & $\frac{\sqrt{15}}{2}$ & $1$ & $\frac{1}{2}$ & $0$ & $0$ & $0$
% & \multirow{3}{*}{ 3 ($3 u_{+}$) }
 & \multirow{3}{*}{ $4 n + 2$, $3^{2k+1}(3n + 1)$ }
 & Orthorhombic(I) \\
   &
${}^{**}2^3 3$ (8) & 1 & 4 & 7 & 0 & 0 & -4
% & & $\sqrt{6}$ & $1$ & $1$ & $0$ & $0$ & $0$
 &
 & Orthorhombic(C) \\
   &
${}^{**}3 \cdot 11$ (11) & 1 & 5 & 7 & -1 & 0 & -1
% & & $\frac{\sqrt{19}}{2}$ & $\frac{1}{2}$ & $\sqrt{7}$ & $0$ & $\frac{1}{\sqrt{133}}$ & $0$
 &
 & Monoclinic(C) \\
\hline
 29 &
${}^{**}2^4 3$ (3) & 1 & 8 & 8 & 0 & 0 & -8
% & & $2\sqrt{2}$ & $2\sqrt{2}$ & $1$ & $0$ & $0$ & $\frac{1}{2}$
% & \multirow{3}{*}{ 2 ($u_{5}$) }
 & \multirow{3}{*}{ $4 n + 2$, $4 n + 3$, $2^{2k} (8 n + 5)$ }
 & Hexagonal \\
   &
${}^{**}2^4 11$ (11) & 1 & 8 & 24 & 0 & 0 & -8
% & & $\sqrt{22}$ & $\sqrt{2}$ & $1$ & $0$ & $0$ & $0$
 &
 & Orthorhombic(C) \\
   &
${}^{*}2^6 3$ (12) & 1 & 8 & 24 & 0 & 0 & 0
% & & $2\sqrt{6}$ & $2\sqrt{2}$ & $1$ & $0$ & $0$ & $0$
 &
 & Orthorhombic(P) \\
\hline
 30 &
${}^{**}3$ (3) & 1 & 2 & 2 & 0 & 0 & -2
% & & $\sqrt{2}$ & $\sqrt{2}$ & $1$ & $0$ & $0$ & $\frac{1}{2}$
% & \multirow{3}{*}{ 2 ($u_{5}$) }
 & \multirow{3}{*}{ $2^{2k} (8 n + 5)$ }
 & Hexagonal \\
   &
${}^{**}11$ (11) & 1 & 2 & 6 & 0 & 0 & -2
% & & $\frac{\sqrt{22}}{2}$ & $\frac{\sqrt{2}}{2}$ & $1$ & $0$ & $0$ & $0$
 &
 & Orthorhombic(C) \\
   &
${}^{**}2^2 3$ (12) & 1 & 2 & 6 & 0 & 0 & 0
% & & $\sqrt{6}$ & $\sqrt{2}$ & $1$ & $0$ & $0$ & $0$
 &
 & Orthorhombic(P) \\
\hline
 31 &
${}^{**}2^{-1} 3$ (3) & 1 & 1 & 2 & -1 & 0 & 0
% & & $1$ & $1$ & $\sqrt{2}$ & $0$ & $0$ & $\frac{1}{2}$
% & \multirow{3}{*}{ 2 ($2 u_{5}$) }
 & \multirow{3}{*}{ $2^{2k + 1} (8 n + 5)$ }
 & Hexagonal \\
   &
${}^{**}2^{-1} 11$ (11) & 1 & 2 & 3 & 0 & -1 & 0
% & & $\frac{\sqrt{11}}{2}$ & $\frac{1}{2}$ & $\sqrt{2}$ & $0$ & $0$ & $0$
 &
 & Orthorhombic(C) \\
   &
${}^{**}2 \cdot 3$ (12) & 1 & 2 & 3 & 0 & 0 & 0
% & & $\sqrt{3}$ & $\sqrt{2}$ & $1$ & $0$ & $0$ & $0$
 &
 & Orthorhombic(P) \\
\hline
 32 &
${}^{**}2^4$ (1) & 3 & 3 & 3 & -2 & -2 & -2
% & & $1$ & $1$ & $1$ & $0$ & $0$ & $0$
% & \multirow{3}{*}{ 2 ($u_{7}$) }
 & \multirow{3}{*}{ $4 n + 1$, $4 n + 2$, $2^{2k}(8n + 7)$ }
 & Cubic(I) \\
   &
${}^{**}2^6$ (4) & 3 & 3 & 8 & -2 & 0 & 0
% & & $\sqrt{2}$ & $1$ & $2\sqrt{2}$ & $0$ & $0$ & $0$
 &
 & Orthorhombic(C) \\
   &
${}^{**}2^4 3^2$ (9) & 3 & 3 & 19 & -2 & -2 & -2
% & & $3 \sqrt{2}$ & $\sqrt{2}$ & $1$ & $0$ & $0$ & $0$
 &
 & Orthorhombic(F) \\
\hline
 33 &
${}^{**}2$ (1) & 1 & 1 & 3 & 1 & 1 & 1
% & & $\sqrt{3}$ & $\sqrt{3}$ & $\sqrt{3}$ & $\frac{5}{6}$ & $\frac{5}{6}$ & $\frac{5}{6}$
% & \multirow{3}{*}{ 2 ($2 u_{7}$) }
 & \multirow{3}{*}{ $4 n + 2$, $2^{2k+1}(8n + 7)$ }
 & Rhombohedral \\
   &
${}^{**}2^3$ (4) & 1 & 3 & 3 & 0 & 0 & -2
% & & $\sqrt{2}$ & $1$ & $1$ & $0$ & $0$ & $0$
 &
 & Orthorhombic(C) \\
   &
${}^{*}2 \cdot 3^2$ (9) & 1 & 3 & 7 & 1 & 1 & 2
% & & $\frac{\sqrt{11}}{2}$ & $\frac{1}{2}$ & $2\sqrt{2}$ & $0$ & $\frac{2}{\sqrt{22}}$ & $0$
 &
 & Monoclinic(C) \\
\hline
 34 &
${}^{**}1$ (1) & 1 & 1 & 1 & 0 & 0 & 0
% & & $1$ & $1$ & $1$ & $0$ & $0$ & $0$
% & \multirow{3}{*}{ 2 ($u_{7}$) }
 & \multirow{3}{*}{ $2^{2 k}(8 n + 7)$ }
 & Cubic(P) \\
   &
${}^{**}2^2$ (4) & 1 & 2 & 2 & 0 & 0 & 0
% & & $\sqrt{2}$ & $\sqrt{2}$ & $1$ & $0$ & $0$ & $0$
 &
 & Tetragonal(P) \\
   &
${}^{**}3^2$ (9) & 1 & 2 & 5 & 0 & 0 & -2
% & & $\frac{3\sqrt{2}}{2}$ & $\frac{\sqrt{2}}{2}$ & $1$ & $0$ & $0$ & $0$
 &
 & Orthorhombic(C) \\
\hline
 35 &
${}^{**}2^{-1} 3^3$ (2) & 1 & 4 & 4 & 1 & 1 & 2
% & & $\frac{3}{2}$ & $\frac{\sqrt{6}}{2}$ & $\frac{1}{2}$ & $0$ & $0$ & $0$
% & \multirow{3}{*}{ 3 ($3 u_{+}$) }
 & \multirow{3}{*}{ $3 n - 1$, $3^{2k+1}(3 n + 1)$ }
 & Orthorhombic(I) \\
   &
${}^{**}2 \cdot 3^3$ (8) & 1 & 6 & 9 & 0 & 0 & 0
% & & $3$ & $\sqrt{6}$ & $1$ & $0$ & $0$ & $0$
 &
 & Orthorhombic(P) \\
   &
${}^{*!}2^{-2} 3^3 11$ (11) & 1 & 6 & 13 & 0 & -1 & -3
% & & $\frac{\sqrt{51}}{2}$ & $\frac{1}{2}$ & $\sqrt{6}$ & $0$ & $\frac{1}{\sqrt{34}}$ & $0$
 &
 & Monoclinic(C) \\
\hline
 36 &
${}^{**}2^{-1} 5^2$ (2) & 1 & 4 & 4 & 1 & 1 & 3
% & & $\frac{\sqrt{10}}{2}$ & $\frac{\sqrt{5}}{2}$ & $\frac{1}{2}$ & $0$ & $0$ & $0$
% & \multirow{3}{*}{ 5 ($u_{-}$) }
 & \multirow{3}{*}{ $5^{2 k} (5 n \pm 2)$ }
 & Orthorhombic(I) \\
   &
${}^{*}2^{-2} 3 \cdot 5^2$ (3) & 1 & 4 & 5 & -1 & 0 & 0
% & & $\frac{\sqrt{15}}{2}$ & $\frac{1}{2}$ & $\sqrt{5}$ & $0$ & $0$ & $0$
 &
 & Orthorhombic(C) \\
   &
${}^{**}2 \cdot 5^2$ (8) & 1 & 5 & 10 & 0 & 0 & 0
% & & $\sqrt{10}$ & $\sqrt{5}$ & $1$ & $0$ & $0$ & $0$
 &
 & Orthorhombic(P) \\
\hline
 37 &
${}^{**}2^{-1} 5$ (2) & 1 & 2 & 2 & 1 & 1 & 2
% & & $\frac{\sqrt{5}}{2}$ & $\frac{\sqrt{2}}{2}$ & $\frac{1}{2}$ & $0$ & $0$ & $0$
% & \multirow{3}{*}{ 5 ($5 u_{-}$) }
 & \multirow{3}{*}{ $5^{2k+1}(5 n \pm 2)$ }
 & Orthorhombic(I) \\
   &
${}^{*}2^{-2} 3 \cdot 5$ (3) & 1 & 2 & 2 & 0 & 0 & -1
% & & $\frac{\sqrt{5}}{2}$ & $\frac{\sqrt{3}}{2}$ & $1$ & $0$ & $0$ & $0$
 &
 & Orthorhombic(C) \\
   &
${}^{**}2 \cdot 5$ (8) & 1 & 2 & 5 & 0 & 0 & 0
% & & $\sqrt{5}$ & $\sqrt{2}$ & $1$ & $0$ & $0$ & $0$
 &
 & Orthorhombic(P) \\
\hline
 38 &
${}^{**}2^{-2} 5^2$ (1) & 2 & 2 & 2 & -1 & -1 & -1
% & & $\sqrt{2}$ & $\sqrt{2}$ & $\sqrt{2}$ & $-\frac{1}{4}$ & $-\frac{1}{4}$ & $-\frac{1}{4}$
% & \multirow{3}{*}{ 5 ($u_{+}$) }
 & \multirow{3}{*}{ $5^{2k}(5n \pm 1)$ }
 & Rhombohedral \\
   &
${}^{**}5^2$ (4) & 2 & 3 & 5 & -2 & 0 & 0
% & & $\frac{\sqrt{10}}{2}$ & $\frac{\sqrt{2}}{2}$ & $\sqrt{5}$ & $0$ & $0$ & $0$
 &
 & Orthorhombic(C) \\
   &
${}^{**}5^2$ (4) & 2 & 2 & 7 & -1 & -1 & -1
% & & $\frac{\sqrt{3}}{2}$ & $\frac{\sqrt{5}}{2}$ & $\sqrt{7}$ & $0$ & $\frac{1}{\sqrt{21}}$ & $0$
 &
 & Monoclinic(C) \\
\hline
 39 &
${}^{**}2^{-2} 5$ (1) & 1 & 1 & 2 & 1 & 1 & 1
% & & $\sqrt{2}$ & $\sqrt{2}$ & $\sqrt{2}$ & $\frac{3}{4}$ & $\frac{3}{4}$ & $\frac{3}{4}$
% & \multirow{3}{*}{ 5 ($5 u_{+}$) }
 & \multirow{3}{*}{ $5^{2k+1}(5n \pm 1)$ }
 & Rhombohedral \\
   &
${}^{**}5$ (4) & 1 & 2 & 3 & 0 & 0 & -2
% & & $\frac{\sqrt{10}}{2}$ & $\frac{\sqrt{2}}{2}$ & $1$ & $0$ & $0$ & $0$
 &
 & Orthorhombic(C) \\
   &
${}^{**}5$ (4) & 1 & 2 & 3 & -1 & 0 & -1
% & & $\frac{\sqrt{7}}{2}$ & $\frac{1}{2}$ & $\sqrt{3}$ & $0$ & $\frac{1}{\sqrt{21}}$ & $0$
 &
 & Monoclinic(C) \\
\hline
 40 &
$2^4 3 \cdot 13$ (39) & 5 & 12 & 12 & -4 & -4 & 0
% & & $\sqrt{6}$ & $\sqrt{6}$ & $\sqrt{5}$ & $0$ & $\frac{2}{\sqrt{30}}$ & $0$
% & \multirow{4}{*}{ 2 ($u_{1}$) }
 & \multirow{4}{*}{ $4 n + 2$, $4 n + 3$, $2^{2k}(8n + 1)$ }
 & Monoclinic(C)
%  Genus and spinor representatives:
%    Standard Lattice of rank 3 and degree 3
%    Determinant: 4992
%    Factored Determinant: 2^7 * 3 * 13
%    Inner Product Matrix:
%    [10 -4 -4]
%    [-4 24  0]
%    [-4  0 24],
%
%    Standard Lattice of rank 3 and degree 3
%    Determinant: 4992
%    Factored Determinant: 2^7 * 3 * 13
%    Inner Product Matrix:
%    [16 -8  0]
%    [-8 16  0]
%    [ 0  0 26]
 \\
   &
$2^4 71$ (71) & 5 & 12 & 21 & -4 & -2 & -4
% & & $\sqrt{21}$ & $2\sqrt{3}$ & $\sqrt{5}$ & $-\frac{1}{\sqrt{15}}$ & $-\frac{1}{\sqrt{105}}$ & $-\frac{1}{3\sqrt{7}}$
 &
 & Triclinic
%  Genus and spinor representatives:
%    Standard Lattice of rank 3 and degree 3
%    Determinant: 9088
%    Factored Determinant: 2^7 * 71
%    Inner Product Matrix:
%    [10 -4 -2]
%    [-4 24 -4]
%    [-2 -4 42],
%
%    Standard Lattice of rank 3 and degree 3
%    Determinant: 9088
%    Factored Determinant: 2^7 * 71
%    Inner Product Matrix:
%    [16 -8  0]
%    [-8 26  2]
%    [ 0  2 26]
 \\
   &
$2^4 79$ (79) & 5 & 12 & 24 & -4 & 0 & -8
% & & $2\sqrt{6}$ & $2\sqrt{3}$ & $\sqrt{5}$ & $-\frac{1}{\sqrt{15}}$ & $0$ & $-\frac{1}{3\sqrt{2}}$
 &
 & Triclinic
%  Genus and spinor representatives:
%    Standard Lattice of rank 3 and degree 3
%    Determinant: 10112
%    Factored Determinant: 2^7 * 79
%    Inner Product Matrix:
%    [10 -4  0]
%    [-4 24 -8]
%    [ 0 -8 48],
%
%    Standard Lattice of rank 3 and degree 3
%    Determinant: 10112
%    Factored Determinant: 2^7 * 79
%    Inner Product Matrix:
%    [ 10   2  -2]
%    [  2  10  -2]
%    [ -2  -2 106],
%
%    Standard Lattice of rank 3 and degree 3
%    Determinant: 10112
%    Factored Determinant: 2^7 * 79
%    Inner Product Matrix:
%    [16 -8 -8]
%    [-8 16  0]
%    [-8  0 58]
 \\
   &
$2^4 5 \cdot 19$ (95) & 5 & 12 & 28 & -4 & -4 & 0
% & & $2\sqrt{7}$ & $2\sqrt{3}$ & $\sqrt{5}$ & $-\frac{1}{\sqrt{15}}$ & $-\frac{1}{\sqrt{35}}$ & $0$
 &
 & Triclinic
%  Genus and spinor representatives:
%    Standard Lattice of rank 3 and degree 3
%    Determinant: 12160
%    Factored Determinant: 2^7 * 5 * 19
%    Inner Product Matrix:
%    [10 -4 -4]
%    [-4 24  0]
%    [-4  0 56],
%
%    Standard Lattice of rank 3 and degree 3
%    Determinant: 12160
%    Factored Determinant: 2^7 * 5 * 19
%    Inner Product Matrix:
%    [10  0  0]
%    [ 0 16 -8]
%    [ 0 -8 80]
 \\
\hline
 41 &
$3 \cdot 13$ (39) & 3 & 3 & 5 & 0 & -2 & -2
% & & $\frac{\sqrt{6}}{2}$ & $\frac{\sqrt{6}}{2}$ & $\sqrt{5}$ & $0$ & $\frac{2}{\sqrt{30}}$ & $0$
% & \multirow{4}{*}{ 2 ($u_{1}$) }
 & \multirow{4}{*}{ $2^{2k}(8n + 1)$ }
 & Monoclinic(C) 
%  Genus and spinor representatives:
%    Standard Lattice of rank 3 and degree 3
%    Determinant: 312
%    Factored Determinant: 2^3 * 3 * 13
%    Inner Product Matrix:
%    [ 6  0 -2]
%    [ 0  6 -2]
%    [-2 -2 10],
%
%    Standard Lattice of rank 3 and degree 3
%    Determinant: 312
%    Factored Determinant: 2^3 * 3 * 13
%    Inner Product Matrix:
%    [ 4  2  0]
%    [ 2  4  0]
%    [ 0  0 26]
\\
   &
$71$ (71) & 3 & 5 & 6 & 2 & 2 & 4
% & & $\sqrt{6}$ & $\sqrt{5}$ & $\sqrt{3}$ & $\frac{1}{\sqrt{15}}$ & $\frac{1}{3\sqrt{2}}$ & $\frac{2}{\sqrt{30}}$
 &
 & Triclinic
%  Genus and spinor representatives:
%    Standard Lattice of rank 3 and degree 3
%    Determinant: 568
%    Factored Determinant: 2^3 * 71
%    Inner Product Matrix:
%    [ 6  2  2]
%    [ 2 10  4]
%    [ 2  4 12],
%
%    Standard Lattice of rank 3 and degree 3
%    Determinant: 568
%    Factored Determinant: 2^3 * 71
%    Inner Product Matrix:
%    [ 4  2 -2]
%    [ 2 12  0]
%    [-2  0 14]
 \\
   &
$79$ (79) & 3 & 5 & 6 & -2 & -2 & 0
% & & $\sqrt{6}$ & $\sqrt{5}$ & $\sqrt{3}$ & $-\frac{1}{\sqrt{15}}$ & $-\frac{1}{3\sqrt{2}}$ & $0$
 &
 & Triclinic
%  Genus and spinor representatives:
%    Standard Lattice of rank 3 and degree 3
%    Determinant: 632
%    Factored Determinant: 2^3 * 79
%    Inner Product Matrix:
%    [ 6 -2 -2]
%    [-2 10  0]
%    [-2  0 12],
%
%    Standard Lattice of rank 3 and degree 3
%    Determinant: 632
%    Factored Determinant: 2^3 * 79
%    Inner Product Matrix:
%    [ 4  0 -2]
%    [ 0  6 -2]
%    [-2 -2 28],
%
%    Standard Lattice of rank 3 and degree 3
%    Determinant: 632
%    Factored Determinant: 2^3 * 79
%    Inner Product Matrix:
%    [ 4 -2  2]
%    [-2  4 -2]
%    [ 2 -2 54]
 \\
   &
$5 \cdot 19$ (95) & 3 & 5 & 7 & -2 & 0 & -2
% & & $\sqrt{7}$ & $\sqrt{5}$ & $\sqrt{3}$ & $-\frac{1}{\sqrt{15}}$ & $0$ & $-\frac{1}{\sqrt{35}}$
 &
 & Triclinic
%  Genus and spinor representatives:
%    Standard Lattice of rank 3 and degree 3
%    Determinant: 760
%    Factored Determinant: 2^3 * 5 * 19
%    Inner Product Matrix:
%    [ 6 -2  0]
%    [-2 10 -2]
%    [ 0 -2 14],
%
%    Standard Lattice of rank 3 and degree 3
%    Determinant: 760
%    Factored Determinant: 2^3 * 5 * 19
%    Inner Product Matrix:
%    [ 4  0 -2]
%    [ 0 10  0]
%    [-2  0 20]
 \\
\hline
 42 &
$2^{-1} 3 \cdot 13$ (39) & 3 & 3 & 3 & 3 & 1 & 1
% & & $\frac{3}{2}$ & $\frac{\sqrt{3}}{2}$ & $\sqrt{3}$ & $0$ & $\frac{1}{3\sqrt{3}}$ & $0$
% & \multirow{4}{*}{ 2 ($2 u_{1}$) }
 & \multirow{4}{*}{ $2^{2k+1}(8n + 1)$ }
 & Monoclinic(C)
%  Genus and spinor representatives:
%    Standard Lattice of rank 3 and degree 3
%    Determinant: 156
%    Factored Determinant: 2^2 * 3 * 13
%    Minimum: 6
%    Kissing Number: 8
%    Inner Product Matrix:
%    [ 6 -3 -1]
%    [-3  6  0]
%    [-1  0  6],
%
%    Standard Lattice of rank 3 and degree 3
%    Determinant: 156
%    Factored Determinant: 2^2 * 3 * 13
%    Minimum: 2
%    Inner Product Matrix:
%    [ 2  1  0]
%    [ 1  2  0]
%    [ 0  0 52]
 \\
   &
$2^{-1} 71$ (71) & 3 & 3 & 5 & 1 & 3 & 2
% & & $\sqrt{5}$ & $\sqrt{3}$ & $\sqrt{3}$ & $\frac{1}{6}$ & $\frac{3}{2\sqrt{15}}$ & $\frac{1}{\sqrt{15}}$
 &
 & Triclinic
%  Genus and spinor representatives:
%    Standard Lattice of rank 3 and degree 3
%    Determinant: 284
%    Factored Determinant: 2^2 * 71
%    Minimum: 6
%    Kissing Number: 4
%    Inner Product Matrix:
%    [ 6 -1 -3]
%    [-1  6 -1]
%    [-3 -1 10],
%
%    Standard Lattice of rank 3 and degree 3
%    Determinant: 284
%    Factored Determinant: 2^2 * 71
%    Minimum: 2
%    Inner Product Matrix:
%    [ 2  1  0]
%    [ 1  6  1]
%    [ 0  1 26]
 \\
   &
$2^{-1} 79$ (79) & 3 & 3 & 5 & -1 & -2 & -1
% & & $\sqrt{5}$ & $\sqrt{3}$ & $\sqrt{3}$ & $-\frac{1}{6}$ & $-\frac{1}{\sqrt{15}}$ & $-\frac{1}{2\sqrt{15}}$
 &
 & Triclinic
%  Genus and spinor representatives:
%    Standard Lattice of rank 3 and degree 3
%    Determinant: 316
%    Factored Determinant: 2^2 * 79
%    Minimum: 6
%    Kissing Number: 4
%    Inner Product Matrix:
%    [ 6 -1 -2]
%    [-1  6 -1]
%    [-2 -1 10],
%
%    Standard Lattice of rank 3 and degree 3
%    Determinant: 316
%    Factored Determinant: 2^2 * 79
%    Minimum: 2
%    Inner Product Matrix:
%    [ 2  0 -1]
%    [ 0 12 -2]
%    [-1 -2 14],
%
%    Standard Lattice of rank 3 and degree 3
%    Determinant: 316
%    Factored Determinant: 2^2 * 79
%    Minimum: 2
%    Inner Product Matrix:
%    [  2  -1   1]
%    [ -1   2   0]
%    [  1   0 106]
 \\
   &
$2^{-1} 5 \cdot 19$ (95) & 3 & 5 & 5 & 3 & 2 & 5
% & & $\sqrt{5}$ & $\sqrt{5}$ & $\sqrt{3}$ & $\frac{3}{2\sqrt{15}}$ & $\frac{1}{\sqrt{15}}$ & $\frac{1}{2}$
 &
 & Triclinic
%  Genus and spinor representatives:
%    Standard Lattice of rank 3 and degree 3
%    Determinant: 380
%    Factored Determinant: 2^2 * 5 * 19
%    Minimum: 6
%    Kissing Number: 2
%    Inner Product Matrix:
%    [ 6 -3  2]
%    [-3 10  3]
%    [ 2  3 10],
%
%    Standard Lattice of rank 3 and degree 3
%    Determinant: 380
%    Factored Determinant: 2^2 * 5 * 19
%    Minimum: 2
%    Inner Product Matrix:
%    [ 2 -1  0]
%    [-1 10  0]
%    [ 0  0 20]
 \\
\hline
\end{tabular}
\end{scriptsize}
\end{table}

\begin{table}[htbp]
\caption{Ternary positive-definite quadratic forms with the same representations over $\IntegerRing$}
\label{Forty-nine groups of ternary positive definite quadratic forms representing the same numbers(3/3)}
\begin{scriptsize}
\begin{tabular}{p{1mm}p{20mm}p{4mm}p{4mm}p{4mm} p{4mm}p{4mm}p{4mm}%p{22mm}
	% p{0mm} p{5mm}p{5mm}p{5mm} p{5mm}p{5mm}p{5mm}
	lr}
 \hline
   No.
	& Determinant (Ratio) & $s_{11}$ & $s_{22}$ & $s_{33}$ & $s_{12}$ & $s_{13}$ & $s_{23}$
%	& & $a^*$ & $b^*$ & $c^*$ & $\cos \alpha^*$ & $\cos \beta^*$ & $\cos \gamma^*$
%	& Anisotropic prime\ ($u \in \RationalField^\times/(\RationalField^\times)^2$, $\not\subset q_{\RationalField}(S)$)
	& Integers not represented by the genus
	& Bravais Type\\
\hline
 43  &
${}^{**}2^4 7$ (7) & 5 & 5 & 5 & 2 & 2 & 2
% & & $\sqrt{5}$ & $\sqrt{5}$ & $\sqrt{5}$ & $\frac{1}{5}$ & $\frac{1}{5}$ & $\frac{1}{5}$
% & \multirow{4}{*}{ 2 ($u_{1}$) }
 & \multirow{4}{*}{ $4 n + 2$, $4 n + 3$, $2^{2k}(8n + 1)$ }
 & Rhombohedral \\
   &
${}^{**}2^4 3 \cdot 5$ (15) & 5 & 5 & 12 & -2 & -4 & -4
% & & $\sqrt{10}$ & $\sqrt{3}$ & $\sqrt{2}$ & $0$ & $0$ & $0$
 &
 & Orthorhombic(F) \\
   &
${}^{**}2^4 23$ (23) & 5 & 8 & 12 & 0 & -4 & -8
% & & $\sqrt{10}$ & $\sqrt{2}$ & $\sqrt{5}$ & $0$ & $\frac{2}{5\sqrt{2}}$ & $0$
 &
 & Monoclinic(C) \\
   &
${}^{*}2^6 7$ (28) & 5 & 8 & 12 & 0 & -4 & 0
% & & $2\sqrt{3}$ & $2\sqrt{2}$ & $\sqrt{5}$ & $0$ & $\frac{1}{\sqrt{15}}$ & $0$
 &
 & Monoclinic(P) \\
\hline
 44  &
${}^{**}7$ (7) & 2 & 2 & 3 & 2 & 2 & 2
% & & $\sqrt{3}$ & $\sqrt{3}$ & $\sqrt{3}$ & $\frac{2}{3}$ & $\frac{2}{3}$ & $\frac{2}{3}$
% & \multirow{4}{*}{ 2 ($u_{1}$) }
 & \multirow{4}{*}{ $2^{2k}(8n + 1)$ }
 & Rhombohedral \\
   &
${}^{**}3 \cdot 5$ (15) & 2 & 3 & 3 & -2 & 0 & 0
% & & $\frac{\sqrt{10}}{2}$ & $\frac{\sqrt{2}}{2}$ & $\sqrt{3}$ & $0$ & $0$ & $0$
 &
 & Orthorhombic(C) \\
   &
${}^{**}23 $ (23) & 2 & 3 & 5 & -2 & 0 & -2
% & & $\frac{\sqrt{10}}{2}$ & $\frac{\sqrt{2}}{2}$ & $\sqrt{5}$ & $0$ & $\frac{2}{5\sqrt{2}}$ & $0$
 &
 & Monoclinic(C) \\
   &
${}^{**}2^2 7$ (28) & 2 & 3 & 5 & 0 & 0 & -2
% & & $\sqrt{5}$ & $\sqrt{2}$ & $\sqrt{3}$ & $0$ & $\frac{1}{\sqrt{15}}$ & $0$
 &
 & Monoclinic(P) \\
\hline
 45  &
${}^{**}2^{-1} 7$ (7) & 1 & 1 & 5 & 1 & 1 & 1
% & & $\sqrt{5}$ & $\sqrt{5}$ & $\sqrt{5}$ & $\frac{9}{10}$ & $\frac{9}{10}$ & $\frac{9}{10}$
% & \multirow{4}{*}{ 2 ($2 u_{1}$) }
 & \multirow{4}{*}{ $2^{2k+1}(8n + 1)$ }
 & Rhombohedral \\
   &
${}^{**}2^{-1} 3 \cdot 5$ (15) & 1 & 3 & 3 & 1 & 1 & 1
% & & $\frac{\sqrt{6}}{2}$ & $\frac{\sqrt{5}}{2}$ & $\frac{1}{2}$ & $0$ & $0$ & $0$
 &
 & Orthorhombic(I) \\
   &
${}^{**}2^{-1} 23$ (23) & 1 & 3 & 5 & 1 & 1 & 3
% & & $\frac{\sqrt{11}}{2}$ & $\frac{1}{2}$ & $\sqrt{5}$ & $0$ & $\frac{3}{\sqrt{55}}$ & $0$
 &
 & Monoclinic(C) \\
   &
${}^{**}2 \cdot 7$ (28) & 1 & 3 & 5 & 0 & 0 & -2
% & & $\sqrt{5}$ & $1$ & $\sqrt{3}$ & $0$ & $\frac{1}{\sqrt{15}}$ & $0$
 &
 & Monoclinic(P) \\
\hline
 46  &
${}^{**}2^{-2} 3^3$ (1) & 2 & 2 & 2 & 1 & 1 & 1
% & & $\sqrt{2}$ & $\sqrt{2}$ & $\sqrt{2}$ & $\frac{1}{4}$ & $\frac{1}{4}$ & $\frac{1}{4}$
% & \multirow{4}{*}{ 3 ($3 u_{-}$) }
 & \multirow{4}{*}{ $3n + 1$, $3^{2k+1}(3 n - 1)$ }
 & Rhombohedral \\
   &
${}^{**}3^3$ (4) & 2 & 3 & 5 & 0 & -2 & 0
% & & $\frac{3\sqrt{2}}{2}$ & $\frac{\sqrt{2}}{2}$ & $\sqrt{3}$ & $0$ & $0$ & $0$
 &
 & Orthorhombic(C) \\
   &
${}^{**}3^3$ (4) & 2 & 2 & 8 & 1 & 2 & 2
% & & $\frac{\sqrt{5}}{2}$ & $\frac{\sqrt{3}}{2}$ & $2\sqrt{2}$ & $0$ & $\frac{1}{\sqrt{10}}$ & $0$
 &
 & Monoclinic(C) \\
   &
${}^{*}2^{-2} 3^3 7$ (7) & 2 & 3 & 8 & 0 & -1 & 0
% & & $2\sqrt{2}$ & $\sqrt{3}$ & $\sqrt{2}$ & $0$ & $\frac{1}{8}$ & $0$
 &
 & Monoclinic(P) \\
\hline
 47  &
${}^{**}2^{-2} 3^2$ (1) & 1 & 1 & 3 & -1 & 0 & 0
% & & $1$ & $1$ & $\sqrt{3}$ & $0$ & $0$ & $\frac{1}{2}$
% & \multirow{4}{*}{ 3 ($u_{-}$) }
 & \multirow{4}{*}{ $3^{2k} (3 n - 1)$ }
 & Hexagonal \\
   &
${}^{**}3^2$ (4) & 1 & 3 & 3 & 0 & 0 & 0
% & & $\sqrt{3}$ & $\sqrt{3}$ & $1$ & $0$ & $0$ & $0$
 &
 & Tetragonal(P) \\
   &
${}^{**}3^2$ (4) & 1 & 3 & 4 & 0 & -1 & -3
% & & $\sqrt{3}$ & $\frac{\sqrt{3}}{2}$ & $\frac{1}{2}$ & $0$ & $0$ & $0$
 &
 & Orthorhombic(I) \\
   &
${}^{*}2^{-2} 3^2 7$ (7) & 1 & 3 & 6 & 0 & 0 & -3
% & & $\frac{\sqrt{21}}{2}$ & $\frac{\sqrt{3}}{2}$ & $1$ & $0$ & $0$ & $0$
 &
 & Orthorhombic(C) \\
\hline
 48  &
${}^{**}2^{-2} 3$ (1) & 1 & 1 & 1 & -1 & 0 & 0
% & & $1$ & $1$ & $1$ & $0$ & $0$ & $\frac{1}{2}$
% & \multirow{4}{*}{ 3 ($3 u_{-}$) }
 & \multirow{4}{*}{ $3^{2 k + 1}(3 n - 1)$ }
 & Hexagonal \\
   &
${}^{**}3$ (4) & 1 & 1 & 3 & 0 & 0 & 0
% & & $1$ & $1$ & $\sqrt{3}$ & $0$ & $0$ & $0$
 &
 & Tetragonal(P) \\
   &
${}^{**}3$ (4) & 1 & 2 & 2 & 1 & 1 & 1
% & & $1$ & $\frac{\sqrt{3}}{2}$ & $\frac{1}{2}$ & $0$ & $0$ & $0$
 &
 & Orthorhombic(I) \\
   &
${}^{*}2^{-2} 3 \cdot 7$ (7) & 1 & 2 & 3 & -1 & 0 & 0
% & & $\frac{\sqrt{7}}{2}$ & $\frac{1}{2}$ & $\sqrt{3}$ & $0$ & $0$ & $0$
 &
 & Orthorhombic(C) \\
\hline
 49  &
${}^{*}2^{-2} 3^3$ (1) & 1 & 1 & 9 & -1 & 0 & 0
% & & $1$ & $1$ & $3$ & $0$ & $0$ & $\frac{1}{2}$
% & \multirow{4}{*}{ 3 ($3 u_{-}$) }
 & \multirow{4}{*}{ $3 n - 1$, $3^{2k + 1}(3 n - 1)$ }
 & Hexagonal \\
   &
${}^{*}2^{-2} 3^3$ (1) & 1 & 3 & 3 & 0 & 0 & -3
% & & $\sqrt{3}$ & $\sqrt{3}$ & $1$ & $0$ & $0$ & $\frac{1}{2}$
 &
 & Hexagonal \\
   &
${}^{*}3^3$ (4) & 1 & 3 & 10 & 0 & -1 & -3
% & & $3$ & $\frac{\sqrt{3}}{2}$ & $\frac{1}{2}$ & $0$ & $0$ & $0$
 &
 & Orthorhombic(I) \\
   &
${}^{**}3^3$ (4) & 1 & 3 & 9 & 0 & 0 & 0
% & & $3$ & $\sqrt{3}$ & $1$ & $0$ & $0$ & $0$
 &
 & Orthorhombic(P) \\
\hline
 50  &
${}^{**}2^{-1}$ (1) & 1 & 1 & 1 & 1 & 1 & 1
% & & $\frac{\sqrt{2}}{2}$ & $\frac{\sqrt{2}}{2}$ & $\frac{\sqrt{2}}{2}$ & $0$ & $0$ & $0$
% & \multirow{4}{*}{ 2 ($2 u_{7}$) }
 & \multirow{4}{*}{ $2^{2k + 1} (8 n + 7)$ }
 & Cubic(F) \\
   &
${}^{**}2$ (4) & 1 & 1 & 2 & 0 & 0 & 0
% & & $1$ & $1$ & $\sqrt{2}$ & $0$ & $0$ & $0$
 &
 & Tetragonal(P) \\
   &
${}^{**}2^{-1} 3^2$ (9) & 1 & 2 & 3 & 0 & -1 & -2
% & & $\frac{3}{2}$ & $\frac{\sqrt{2}}{2}$ & $\frac{1}{2}$ & $0$ & $0$ & $0$
 &
 & Orthorhombic(I) \\
   &
${}^{**}2^3$ (16) & 1 & 2 & 4 & 0 & 0 & 0
% & & $2$ & $\sqrt{2}$ & $1$ & $0$ & $0$ & $0$
 &
 & Orthorhombic(P) \\
\hline
 51  &
${}^{**}2^{-1} 3^3$ (2) & 2 & 2 & 5 & 2 & 2 & 1
% & & $\frac{3\sqrt{2}}{2}$ & $\frac{\sqrt{6}}{2}$ & $\frac{\sqrt{2}}{2}$ & $0$ & $0$ & $0$
% & \multirow{5}{*}{ 3 ($3 u_{+}$) }
 & \multirow{5}{*}{ $3 n + 1$, $3^{2k+1}(3n + 1)$ }
 & Orthorhombic(F) \\
   &
${}^{**}2^{-2} 3^3 5$ (5) & 2 & 5 & 5 & 2 & 1 & 5
% & & $\frac{\sqrt{5}}{2}$ & $\frac{\sqrt{15}}{2}$ & $\sqrt{2}$ & $0$ & $\frac{1}{\sqrt{10}}$ & $0$
 &
 & Monoclinic(C) \\
   &
${}^{**}2 \cdot 3^3$ (8) & 2 & 5 & 6 & -2 & 0 & 0
% & & $\frac{3\sqrt{2}}{2}$ & $\frac{\sqrt{2}}{2}$ & $\sqrt{6}$ & $0$ & $0$ & $0$
 &
 & Orthorhombic(C) \\
   &
${}^{**}2 \cdot 3^3$ (8) & 2 & 5 & 6 & -1 & 0 & -3
% & & $\sqrt{6}$ & $\sqrt{5}$ & $\sqrt{2}$ & $-\frac{1}{2\sqrt{10}}$ & $0$ & $-\frac{3}{2\sqrt{30}}$
 &
 & Triclinic \\
   &
${}^{**}2^{-2} 3^3 11$ (11) & 2 & 5 & 8 & -1 & -1 & -2
% & & $2\sqrt{2}$ & $\sqrt{5}$ & $\sqrt{2}$ & $-\frac{1}{2\sqrt{10}}$ & $-\frac{1}{8}$ & $-\frac{1}{2\sqrt{10}}$
 &
 & Triclinic \\
\hline
 52  &
${}^{**}2^{-1} 3^2$ (2) & 2 & 2 & 2 & 2 & 2 & 1
% & & $\frac{\sqrt{6}}{2}$ & $\frac{\sqrt{6}}{2}$ & $\frac{\sqrt{2}}{2}$ & $0$ & $0$ & $0$
% & \multirow{5}{*}{ 3 ($u_{+}$) }
 & \multirow{5}{*}{ $3^{2k} (3 n + 1)$ }
 & Tetragonal(I) \\
   &
${}^{**}2^{-2} 3^2 5$ (5) & 2 & 2 & 3 & -1 & 0 & 0
% & & $\frac{\sqrt{5}}{2}$ & $\frac{\sqrt{3}}{2}$ & $\sqrt{3}$ & $0$ & $0$ & $0$
 &
 & Orthorhombic(C) \\
   &
${}^{**}2 \cdot 3^2$ (8) & 2 & 3 & 3 & 0 & 0 & 0
% & & $\sqrt{3}$ & $\sqrt{3}$ & $\sqrt{2}$ & $0$ & $0$ & $0$
 &
 & Tetragonal(P) \\
   &
${}^{**}2 \cdot 3^2$ (8) & 2 & 2 & 5 & 1 & 1 & 1
% & & $\frac{\sqrt{5}}{2}$ & $\frac{\sqrt{3}}{2}$ & $\sqrt{5}$ & $0$ & $\frac{1}{5}$ & $0$
 &
 & Monoclinic(C) \\
   &
${}^{**}2^{-2} 3^2 11$ (11) & 2 & 3 & 5 & 0 & -1 & -3
% & & $\frac{\sqrt{17}}{2}$ & $\frac{\sqrt{3}}{2}$ & $\sqrt{2}$ & $0$ & $\frac{1}{\sqrt{34}}$ & $0$
 &
 & Monoclinic(C) \\
\hline
 53  &
${}^{**}2^{-1} 3$ (2) & 1 & 1 & 2 & 0 & -1 & -1
% & & $\frac{\sqrt{2}}{2}$ & $\frac{\sqrt{2}}{2}$ & $\frac{\sqrt{6}}{2}$ & $0$ & $0$ & $0$
% & \multirow{5}{*}{ 3 ($3 u_{+}$) }
 & \multirow{5}{*}{ $3^{2k + 1} (3 n + 1)$ }
 & Tetragonal(I) \\
   &
${}^{**}2^{-2} 3 \cdot 5$ (5) & 1 & 1 & 4 & 0 & -1 & 0
% & & $\frac{\sqrt{15}}{2}$ & $\frac{1}{2}$ & $1$ & $0$ & $0$ & $0$
 &
 & Orthorhombic(C) \\
   &
${}^{**}2 \cdot 3$ (8) & 1 & 1 & 6 & 0 & 0 & 0
% & & $1$ & $1$ & $\sqrt{6}$ & $0$ & $0$ & $0$
 &
 & Tetragonal(P) \\
   &
${}^{**}2 \cdot 3$ (8) & 1 & 2 & 4 & 1 & 1 & 2
% & & $\frac{\sqrt{7}}{2}$ & $\frac{1}{2}$ & $2$ & $0$ & $\frac{1}{\sqrt{7}}$ & $0$
 &
 & Monoclinic(C) \\
   &
${}^{**}2^{-2} 3 \cdot 11$ (11) & 1 & 2 & 5 & 1 & 1 & 1
% & & $\frac{\sqrt{7}}{2}$ & $\frac{1}{2}$ & $\sqrt{6}$ & $0$ & $\frac{3}{\sqrt{42}}$ & $0$
 &
 & Monoclinic(C) \\
\hline
\end{tabular}
\end{scriptsize}
\end{table}

Overall, 53 sets consisting of 151 quadratic forms were obtained as the candidates that may have the completely identical representations as a non-equivalent form.
It should be noted that all of them 
can be obtained if the algorithm in Section~\ref{Algorithm}
is carried out under the constraint $s_{33} \leq 48$.
Since the number is small compared with the upper bound $115$ of the searched region,
we expect that the 53 cases, in addition to those in the hexagonal and rhombohedral families, provide all the searched quadratic forms, up to the action of $GL_3(\IntegerRing)$ and constant multiple.

The tables also have information about the regularity of each quadratic form. This is based on the tables of Jagy \textit{et\ al.\ }\cite{Jagy97} and \cite{Oh2011}.
If a form in Tables~\ref{Forty-nine groups of ternary positive definite quadratic forms representing the same numbers(1/3)}--\ref{Forty-nine groups of ternary positive definite quadratic forms representing the same numbers(3/3)}
is regular, it is marked with $**$ or $*$.
In the tables, 38 out of the 53 cases consist of regular (or possibly regular) quadratic forms.
The others are neither regular nor spinor-regular from the result in \cite{Benham90}.
%We confirmed by direct computation that they have the same representations up to a million.
%The confirmation can be carried out by computation.
% is because the condition for $q \in \Lambda(S)$ is not simply represented if $S$ is not regular.

In general, it is difficult to exactly determine the set of integral representations for a ternary quadratic form (\CF \cite{Ono97}).
%For example, it has not been proved that all of the quadratic forms in No. 35 are really regular, 
%which also implies that they have the completely identical representations.
% (see the footnote in Table \ref{Forty-nine groups of ternary positive definite quadratic forms representing the same numbers(1/3)}).
Although it is occasionally possible to prove that the forms have the same representations, without providing the exact set of their representations, as done for the hexagonal and the rhombohedral families,
the author could not do this for the non-regular forms in the tables.
When two positive-definite quadratic forms $f$, $g \in {\rm Sym}^2 (\RationalField^3)^*$ satisfy $q_\IntegerRing(f) = q_\IntegerRing(g)$,
it can be proved that they are equivalent over $\RationalField$ if and only if $\det f / \det g$ is a square in $\RationalField^\times$.
% (the proof is found in the supplementary materials of \cite{Tomiyasu2016}).
%For example, in each set of the No.\ 15--20,
%the quadratic forms are equivalent over $\RationalField$.
%It should be also emphasized that some sets in the tables (\EG No. 12--14, No. 40--42) have the perfectly identical determinant ratios.

%If a quadratic form is not marked with either of  $**$, $*$, $*!$.
%it is not regular.

%\section{Theorems}
%
%As an immediate result of Lemma \ref{lem:decomposition of positive definite symmety matrices},
%any two irrational quadratic forms $f_1$, $f_2$ 
%can be simultaneously decomposed into finite sums $f_i = \sum_{j=1}^s \lambda_j T_{ij}$ ($i = 1, 2$),
%where $\lambda_1, \ldots, \lambda_s \in \RealField$ are linearly independent over $\RationalField$,
%and $T_{ij}$ are rational and positive-definite.
%In this case, $q_\IntegerRing(f_1) = q_\IntegerRing(f_2)$ 
%if and only if $q_\IntegerRing(T_{11}, \ldots, T_{1s}) = q_\IntegerRing(T_{21}, \ldots, T_{2s})$.
%Therefore, investigations of pairs of rational quadratic forms and their simultaneous representations
%are the essential part to prove the Kaplansky conjecture for quadratic forms with irrational coefficients.
%In the following sections, the main theorems are proved.

\section{Preliminaries for the main theorem}
\label{Preliminaries for the main theorem}

Some basic properties on simultaneous representations that will be used in the proof of Theorem \ref{thm:main result over RationalField} are presented herein.

Herein, $k$ is a field with ${\rm char}\ k \neq 2$ as defined in Section \ref{Notation and symbols}.
The following lemmas and Corollary \ref{cor: anisotropic iff anisotropic} are repeatedly used to prove the theorem.
Lemma \ref{lem: diagonalization over F} can be seen as a generalization of the well-known fact that any pairs of positive-definite quadratic forms are simultaneously diagonalized over $\RealField$:
\begin{lem}\label{lem: diagonalization over F}
For any $(A, B) \in {\rm Sym}^2 (k^n)^* \otimes_{k} k^2$ with $\det A \neq 0$,
let $K \subset \bar{k}$ (\textit{resp.}, $K_2 \subset \bar{k}$) be the field generated by all the roots (\textit{resp.}, all the roots of multiplicity $> 1$) of $\det(Ax - B) = 0$ in $\bar{k}$ over $k$. 
If $(A, B)$ is anisotropic over $K_2$, then
$A$ and $B$ are simultaneously diagonalized by the action of $GL_3(K)$.
Consequently, for any $\alpha \in K$, 
the rank of $A \alpha - B$ equals $n$ minus the multiplicity of $\alpha$ as a root of $\det(A x - B) = 0$.
\end{lem}

\begin{proof}
Let $w \in GL_n(K)$ be the matrix that provides a Jordan decomposition of $S := A^{-1} B$.
Since $B$ is symmetric, we have $\tr{w} B w = (\tr{w} A w) (w^{-1} S w) = \tr{(w^{-1} S w)} (\tr{w} A w)$.
Hence, if $w^{-1} S w$ is diagonal, 
$(\tr{w} A w$, $\tr{w} B w)$ is a simultaneous block-diagonalization of $(A, B)$
and each block corresponds to an eigenspace of $S$.
If $[A_1, \cdots, A_m]$ and $[B_1, \cdots, B_m]$ are the blocks of $A$ and $B$, then
at least one of $A_i$, $B_i$ is a constant multiple of the other.
Hence, $(A, B)$ can be simultaneously diagonalized.

We next assume that $S$ has a Jordan block of size $m > 1$
that corresponds to a multiple root $\alpha \in K_2$ of $\det(A x - B) = \det A \det(I x - S) = 0$.
Let $S_2$, $A_2$, $B_2$ be the $m \times m$ blocks of $w^{-1} S w$, $\tr{w} A w$, $\tr{w} B w$, respectively, that correspond to the eigenspace of $\alpha$.
Since we have $B_2 = A_2 S_2 = \tr{S}_2 A_2$,
the $(1, 1)$-entries of $A_2$ and $B_2$ equal $0$.
This is impossible since $(A, B)$ is anisotropic over $K_2$.
Hence, all the Jordan blocks in $S$ must have size $1$. The lemma is proved.
\end{proof}

In particular, $K_2 = k$ always holds for $n=3$.
Therefore, any $(A, B) \in V_k$ anisotropic over $k$, is simultaneously diagonalized by the action of $GL_3(K)$.
\begin{lem}\label{lem: equivalence conditions on isotropy}
We assume that $(A, B) \in V_k$ is linearly independent,
non-singular and anisotropic over $k$, and satisfies $\det A \neq 0$, ${\rm Disc}(A, B) \neq 0$.
Let $\alpha_i \in \bar{k}$ ($i = 1, 2, 3$) be the roots of $\det(A x - B) = 0$
and $K_i$ be the Galois closure of $k_i = k(\alpha_i)$ over $k$.
We choose $C_i \in {\rm Sym}^2 (k_i^2)^*$ satisfying $A \alpha_i - B \sim_{k_i} C_i \perp [ 0 ]$.
Then, 
$C_i$ is anisotropic over $K_i$ for at least two of $i = 1, 2, 3$ and all of $i$ with $k_i \neq k$.
\end{lem}
\begin{proof}
Since ${\rm Disc}(A, B) \neq 0$, 
the vector $0 \neq v_i \in k_i^3$ with $(A \alpha_i - B) v_i = 0$ is uniquely determined, up to constant multiple of $k_i^\times$.
From $(A \alpha_i - B) v_i = 0$ and $(A \alpha_j - B) v_j = 0$,
we have $\tr{v}_i A v_j =  \tr{v}_i B v_j = 0$ for any distinct $i, j = 1, 2, 3$.
Hence, if $w$ equals the matrix $\tr{(v_1\ v_2\ v_3)}$, then
$w A \tr{w}$ and $w B \tr{w}$ are diagonal.
It may be assumed that $[d_1, d_2, d_3]$ and $[d_1 \alpha_1, d_2 \alpha_2, d_3 \alpha_3]$ are diagonal entries of $A$ and $B$.
For any $i, j = 1, 2, 3$, 
if $k_i, k_j \neq k$ are conjugate over $k$,
it may be assumed that 
the isomorphism $\sigma_{ij} : k_i \rightarrow k_j$ over $k$ induced by $\alpha_i \mapsto  \alpha_j$ maps $v_i$ to $v_j$.
Then, $\sigma_{ij}$ maps $d_i$ to $d_j$.
We now have the following:
\begin{small}
\begin{eqnarray}\label{eq: transform of (A,B)}
	\left( w, 
		\begin{pmatrix}	
			\alpha_1 & -1 \\
			\alpha_2 & -1 
		\end{pmatrix}	
	\right) 
	\cdot
	(A, B) 
	=
	\left( 
		\begin{pmatrix}	
			0 & 0 & 0 \\
			0 & d_2 (\alpha_1 - \alpha_2) & 0 \\
			0 & 0 & d_3 (\alpha_1 - \alpha_3)
		\end{pmatrix},
		\begin{pmatrix}	
			d_1 (\alpha_2 - \alpha_1) & 0 & 0 \\
			0 & 0 & 0 \\
			0 & 0 & d_3 (\alpha_2 - \alpha_3)
		\end{pmatrix}	
	\right). \nonumber \\
\end{eqnarray}
\end{small}
In what follows, the pair in the right-hand side is denoted by $(\tilde{A}, \tilde{B})$.
For the proof, we may assume one of the following:
\begin{enumerate}[(i)]
	\item $k_1 = k_2 = k_3 = k$,
	\item $k_1, k_2$ are quadratic over $k$,
	\item $k_1, k_2, k_3$ are cubic over $k$.
\end{enumerate}
We shall show that the assumption that the diagonal $D_1 = [d_2 (\alpha_1 - \alpha_2), d_3 (\alpha_1 - \alpha_3)]$ 
is isotropic over $K_1$ and $D_2 = [ d_2 (\alpha_1 - \alpha_2), d_3 (\alpha_1 - \alpha_3) ]$ is isotropic over $K_2$, 
leads to a contradiction.
In fact, the assumption holds if and only if 
some $\beta_1 \in K_1$, $\beta_2 \in K_2$ satisfy 
$d_2 d_3 (\alpha_1 - \alpha_2)(\alpha_1 - \alpha_3) = -\beta_1^2$ and 
$d_1 d_3 (\alpha_2 - \alpha_1)(\alpha_2 - \alpha_3) = -\beta_2^2$.
Then, $(\tilde{A}, \tilde{B})({\mathbf x}) = 0$ holds,
if we put ${\mathbf x} := (-\beta_2 / d_1 (\alpha_1 - \alpha_2), \beta_1 / d_2(\alpha_1 - \alpha_2), \pm 1)$.
Hence, 
$(A, B)$ is isotropic over $k$ in case (i).
The same is true in case (ii), 
since 
$(A, B)({\mathbf x} w) = (\tilde{A}, \tilde{B})({\mathbf x}) = 0$ and ${\mathbf x} w \in k^3$, 
if we choose $\beta_2$ such that $\beta_2 = \sigma_{12}(\beta_1)$.
In case (iii),
$A \alpha_i - B$ is isotropic over $K_1 = K_2 = K_3$ for all $i = 1, 2, 3$
due to conjugacy.
There exist $\beta_i \in k_i$ ($i = 1, 2, 3$) such that either of the following holds for every distinct $1 \leq h, i, j \leq 3$,
considering that $d_i d_j (\alpha_h - \alpha_i)(\alpha_h - \alpha_j)$ belongs to $k_h$:
\begin{itemize}
	\item $d_i d_j (\alpha_h - \alpha_i)(\alpha_h - \alpha_j) = -\beta_h^2$,
	\item $d_i d_j (\alpha_h - \alpha_i)(\alpha_h - \alpha_j) = -\beta_h^2 (\alpha_1 - \alpha_2)^2 (\alpha_2 - \alpha_3)^2 (\alpha_3 - \alpha_1)^2$.
\end{itemize}
It is possible to choose $\beta_i$ so that $\beta_1, \beta_2, \beta_3$ are conjugate.
If we put ${\mathbf x}_2 := (1/\beta_1, 1/\beta_2, 1/\beta_3)$,
we have $(\tilde{A}, \tilde{B})({\mathbf x}_2) = (A, B)({\mathbf x}_2 \tr{w}) = 0$ and ${\mathbf x}_2 w \in k^3$.
Thus, the lemma is proved.
%As a result, $A \alpha_i - B$ can be isotropic over $k_i$ for at most one $i$.
%Suppose that the same thing does not hold if $k_i$ is replaced by its Galois closure $K$.
%This can happen only when $k_i$ is cubic and not Galois over $k$,
%thus $-\det C_i \in k_i$ does not belong to $(k_i^\times)^2$ for every $i$.
%Therefore, in order to have $-\det C_i \in (K^\times)^2$, 
%there must exist $\beta_i \in k_i$ ($i = 1, 2, 3$) conjugate over $k$ such that  
%$d_i d_j (\alpha_k - \alpha_i)(\alpha_k - \alpha_j) = -\beta_k^2 \prod_{1 \leq i < j \leq 3} (\alpha_i - \alpha_j)^2$ holds.
%Thus, even in this case, 
%${\mathbf x}_2:= (1/\beta_1, 1/\beta_2, 1/\beta_3)$
%satisfies $(\tilde{A}, \tilde{B})({\mathbf x}_2) = (A, B)({\mathbf x}_2 \tr{w}) = 0$.
\end{proof}

%As a fundamental fact, $q_\IntegerRing(A_1, B_1) = q_\IntegerRing(A_2, B_2)$ leads to $q_\RationalField(A_1, B_1) = q_\RationalField(A_2, B_2)$,
%and $\{ (0, 0) \} \cup q_{\IntegerRing_p}(A_1, B_1) = \{ (0, 0) \} \cup q_{\IntegerRing_p}(A_2, B_2)$, 
%by taking the topological closure of $q_\IntegerRing(A_i, B_i)$ in $\RationalField_p^2$.
In general, for any global field $k$, prime $v$ of $k$, and $(A_i, B_i) \in V_k$ ($i = 1, 2$), if we start from $q_k(A_1, B_1) = q_k(A_2, B_2)$, 
the following are immediately obtained:
\begin{enumerate}[($C$1)]
	\item If both $(A_1, B_1)$ and $(A_2, B_2)$ are anisotropic over $k_v$,
		$q_{k_v} (A_1, B_1) = q_{k_v}(A_2, B_2)$, since the topological closure of $q_k(A_i, B_i)$ in $k_v^2$ equals 
		$\{ (0, 0) \} \cup q_{k_v}(A_i, B_i)$.

	\item If both $(A_1, B_1)$ and $(A_2, B_2)$ are isotropic over $k_v$,
		$q_{k_v} (A_1 \pi - B_1) = q_{k_v}(A_2 \pi - B_2) = k_v$ for any $\pi \in k_v$.

	\item If $(A_1, B_1)$ is anisotropic and $(A_2, B_2)$ is isotropic over $k_v$,
		$A_1 \pi - B_1$ is isotropic over $k_v$ for any $\pi \in k_v$.

\end{enumerate}

As a corollary of Lemma \ref{lem: equivalence conditions on isotropy},
($C$3) is excluded in the following case:
\begin{cor}\label{cor: anisotropic iff anisotropic}
Let $v$ be a prime of $k$.
We assume that $(A_1, B_1), (A_2, B_2) \in V_k$ 
satisfy $q_k(A_1, B_1) = q_k(A_2, B_2)$, and $\det (A_1 x - B_1 y)$ completely splits in $k_v$.
In this case, either of the above ($C$1) or ($C$2) is true,
hence $q_{k_v} (A_1 \pi - B_1) = q_{k_v}(A_2 \pi - B_2)$ for any $\pi \in k_v$.
In particular, $(A_1, B_1)$ is isotropic over $k_v$ if and only if so is $(A_2, B_2)$.
\end{cor}

For understanding of the following sections, we need to recall that 
a one-to-one correspondence between the elements of $V_\IntegerRing := {\rm Sym}^2 (\IntegerRing^3)^* \otimes \IntegerRing^2$ 
and all pairs of a quartic ring and its resolvent ring was recently proved in \cite{Bhargava2004}.
The result was generalized to the case of any base scheme $S$ \cite{Wood2011}.
Herein, 
the generalization for Dedekind domains of \cite{O'Dorney2016} is adopted.

In the following of this section, as in \cite{O'Dorney2016}, $R$ is always a \textit{Dedekind domain}, \IE a Noetherian, integrally closed integral domain
that has the property that every nonzero prime ideal is maximal. Therefore, any fields are also a Dedekind domain.
A finitely generated, torsion-free $R$-module
is called a \textit{lattice over $R$.}
If $M$ is a lattice over $R$ and $k$ is the field of fractions of $R$,
the {\textit rank} of $M$ is defined as the dimension of $M \otimes_R k$ over $k$.
A unitary commutative associative $R$-algebra is called a \textit{quartic ring} (\textit{resp. cubic ring}), if it has rank 4 (\textit{resp. 3})
as a lattice over $R$.
When we put $L := Q / R$, $M := C/R$ and $L^* := {\rm Hom}(L, R)$,
a \textit{quadratic map} means an element of ${\rm Sym}^2 L^* \otimes_R M$.

\begin{defin}
Let $R$ be a Dedekind domain.
For any quartic ring $Q$, 
its \textit{cubic resolvent ring} $C$ (also called numerical resolvent in \cite{O'Dorney2016}) is defined as the $R$-algebra with the following properties:
\begin{itemize}
	\item It is equipped with an $R$-module isomorphism $\theta: \Lambda^2 (C/R) \rightarrow \Lambda^3 (Q/R)$ and a quadratic map $\phi: Q/R \rightarrow C/R$ such that
		\begin{eqnarray}\label{eq: eq for Q}
		 	x \wedge y \wedge x y = \theta(\phi(x) \wedge \phi(y))  \text{ for any } x, y \in Q/R.
		\end{eqnarray}

	\item The multiplicative structure of $C$ is determined by 
		\begin{eqnarray}\label{eq: eq for C}
			x \wedge x^2 \wedge x^3 = \theta(\phi(x) \wedge \phi(x)^2) \text{ for any } x \in Q/R,
		\end{eqnarray}
		where $\phi(x)^2$ is the square $\tilde{\phi}(x)^2$ of any lift $\tilde{\phi} : Q \rightarrow C$ of $\phi$.

\end{itemize}

The quadratic map $\phi$ is called the \textit{resolvent map of $(Q, C)$}.
\end{defin}

If $R$ is a field, and $Q = R[\alpha]$ \IE the quartic extension by $\alpha$, the \textit{classical resolvent map} from $Q$ to $R$ is defined by $\alpha \mapsto \alpha \alpha^\prime + \alpha^{\prime\prime} \alpha^{\prime\prime\prime}$, 
by using the conjugates $\alpha^\prime$, $\alpha^{\prime\prime}$, $\alpha^{\prime\prime\prime}$ of $\alpha$ over $R$.
The following explains a generalization of the Bhargava correspondence for quartic rings 
to the case over a Dedekind domain:

\begin{thm}[Theorem 1.4 \cite{O'Dorney2016}]
	Let $R$ be a Dedekind domain. There is a canonical bijection between 
	\begin{enumerate}[(i)]
	\item Isomorphism classes of pairs $(Q, C)$, where $Q$ is a quartic ring
	and $C$ is a cubic resolvent ring of $Q$.

	\item Quadruples $(L, M, \theta, \phi)$, where $L$ and $M$ are lattices of ranks $3$ and $2$ over $R$, respectively,
	$\theta: \Lambda^2 M \rightarrow \Lambda^3 L$ is an isomorphism and $\phi: L \rightarrow M$ is a quadratic map.
	\end{enumerate}
	Under this bijection, the identifications $Q/R \cong L$ and $C/R \cong M$ are obtained.
Any quartic ring $Q$ has a cubic resolvent, and if $Q$ is Dedekind, the resolvent is unique.
\end{thm}

In what follows, we explain how the pair of $Q$ and $R$ in (i), is associated with  $(L, M, \theta, \phi)$ in (ii). 

Since $R$ is a Dedekind domain,
the lattice $Q/R$ is isomorphic to ${\mathfrak a}_1 \bar{\xi}_1 \oplus {\mathfrak a}_2 \bar{\xi}_2
\oplus {\mathfrak a}_3 \bar{\xi}_3$
for some $\bar{\xi}_1, \bar{\xi}_2, \bar{\xi}_3 \in Q/R$ and ideals ${\mathfrak a}_1, {\mathfrak a}_2, {\mathfrak a}_3$ of $R$.
Similarly, $C/R$ is isomorphic to ${\mathfrak b}_1 \bar{\omega}_1 \oplus {\mathfrak b}_2 \bar{\omega}_2$
for some $\bar{\omega}_1, \bar{\omega}_2 \in C/R$ and ideals ${\mathfrak b}_1, {\mathfrak b}_2$ of $R$.
If these ${\mathfrak a}_i$, $\bar{\xi}_i$, ${\mathfrak b}_j$, $\bar{\omega}_j$ are fixed, 
the quadratic map $\phi : Q/R \rightarrow C/R$ 
is uniquely associated with a pair of ternary quadratic forms $(A, B)$ over the fraction field $k$,
by 
\begin{eqnarray}\label{eq: relation between phi and (A, B)}
\phi(x_1 \bar{\xi}_1 + x_2 \bar{\xi}_2 + x_3 \bar{\xi}_3) = B(x_1, x_2, x_3) \bar{\omega}_1 + A(x_1, x_2, x_3) \bar{\omega}_2.
\end{eqnarray}

If we put $L := Q/R$, $Q$ is isomorphic to $R \oplus L$ as an $R$-module.
Let $\xi_i \in Q$ ($1 \leq i \leq 3$) be the element corresponding to $\bar{\xi}_i$ in $Q / R$.
We first note that Eq.(\ref{eq: eq for Q}) implies
		\begin{eqnarray*}
			\left( \sum_{i=1}^3 x_i \bar{\xi}_i \right) \wedge \left( \sum_{i=1}^3 y_i \bar{\xi}_i \right)
			 \wedge \left( \sum_{i=1}^3 \sum_{j=1}^3 x_i y_j \bar{\xi}_i \bar{\xi}_j \right)
 &=& \theta\left( \phi \left( \sum_{i=1}^3 x_i \bar{\xi}_i \right) \wedge \phi \left( \sum_{i=1}^3 y_i \bar{\xi}_i \right) \right).
		\end{eqnarray*}
The multiplicative structure of $Q$, \IE all $c_{ij}^k \in R$ of the equalities $\xi_i \xi_j = c_{ij}^0 + \sum_{i=1}^3 c_{ij}^k \xi_k$, is determined from Eq.(\ref{eq: eq for Q}) as follows;
we denote the coefficients of $A$ and $B$ by $A(x_1, x_2, x_3) = \sum_{1 \leq i \leq j \leq 3} a_{ij} x_i x_j$,
$B(x_1, x_2, x_3) = \sum_{1 \leq i \leq j \leq 3} b_{ij} x_i x_j$.
We then have
		\begin{eqnarray*}
			(\xi_1 \wedge \xi_2 \wedge \xi_3)
			\begin{vmatrix}			
				x_1 & y_1 & \sum_{i, j=1}^3 c_{ij}^1 x_i y_j \\
				x_2 & y_2 & \sum_{i, j=1}^3 c_{ij}^2 x_i y_j \\
				x_3 & y_3 & \sum_{i, j=1}^3 c_{ij}^3 x_i y_j \\
			\end{vmatrix}
			=
-
\theta\left( \bar{\omega}_1 \wedge \bar{\omega}_2 \right)
\sum_{1 \leq i \leq j \leq 3} 
\sum_{1 \leq k \leq l \leq 3} x_i x_j y_k y_l 
			\begin{vmatrix}
				a_{ij} & b_{ij} \\
				a_{kl} & b_{kl} \\
			\end{vmatrix}.
		\end{eqnarray*}

Replacing each $\xi_i$ by $-\xi_i$ if necessary, 
we can fix the sign as follows:
		\begin{eqnarray*}
			\begin{vmatrix}			
				x_1 & y_1 & \sum_{i, j=1}^3 c_{ij}^1 x_i y_i \\
				x_2 & y_2 & \sum_{i, j=1}^3 c_{ij}^2 x_i y_i \\
				x_3 & y_3 & \sum_{i, j=1}^3 c_{ij}^3 x_i y_i \\
			\end{vmatrix}
			= -
\sum_{1 \leq i \leq j \leq 3} 
\sum_{1 \leq k \leq l \leq 3} x_i x_j y_k y_l 
			\begin{vmatrix}
				a_{ij} & b_{ij} \\
				a_{kl} & b_{kl} \\
			\end{vmatrix}.
		\end{eqnarray*}

Comparing the coefficients of each term, the following equations are obtained: 
\begin{eqnarray}\label{eq: multiplicative structure of Q}
c^j_{ii} &=& \epsilon \lambda^{ii}_{ik}, \nonumber \\
c^k_{ij} &=& \epsilon \lambda^{jj}_{ii}, \nonumber \\
c^j_{ij} - c^k_{ik} &=& \epsilon \lambda^{jk}_{ii}, \\
c^i_{ii} - c^j_{ij} - c^k_{ik} &=& \epsilon \lambda^{ij}_{ik},\ \nonumber 
\lambda^{ij}_{kl} := \begin{vmatrix} a_{ij} & b_{ij} \\ a_{kl} & b_{kl} \end{vmatrix},
\end{eqnarray}
where $(i, j, k)$ denotes any permutation of $(1, 2, 3)$ and $\epsilon = \pm 1$ is its sign.
The $c^0_{ij}$ are determined from the associative law of $Q$.
Consequently, all  $c_{ij}^k \in R$, are uniquely determined, up to the transformations given by 
$c^j_{ij} \mapsto c^j_{ij} + a$, $c^k_{ij} \mapsto c^k_{ij} + a$, $c^i_{ii} \mapsto c^i_{ii} + 2a$ ($a \in Q$)
which corresponds to the replacement of $\xi_i$ by $\xi_i + a$.
In order to uniquely determine $c_{ij}^k$, the following constraints were applied in \cite{Bhargava2004}:
\begin{eqnarray}\label{eq: normalization}
	c^1_{12} = c^2_{12} = c^1_{13} = 0.
\end{eqnarray}

If we put $f_{det}(x, y) := 4 \det(A x - B y)$, 
the following equation holds for any $x = \sum_{i=1}^3 x_i \bar{\xi}_i \in Q/R$.
\begin{eqnarray}\label{eq: formula for 4 det(B(x) A - A(x) B}
	x \wedge x^2 \wedge x^3 = f_{det}(B(x_1, x_2, x_3), A(x_1, x_2, x_3)) \bar{\xi}_1 \wedge \bar{\xi}_2 \wedge \bar{\xi}_3.
\end{eqnarray}

Hence, the ring structure of $C$ determined by Eq.(\ref{eq: eq for C}), is same as that provided by $4 \det(A x - B y)$ under the Delone-Faddeev-Gan-Gross-Savin correspondence (\cite{Delone64}, \cite{Gan2002}).
If the coefficients are denoted by $4 \det(A x - B y) = a x^3 + b x^2 y + c x y^2 + d y^3$,
the basis $\langle \bar{\omega}_1, \bar{\omega}_2 \rangle$ of $C/R$ as an $R$-algebra is lifted to a basis $\langle 1, \omega_1, \omega_2 \rangle$ of $C$ that satisfies:
\begin{eqnarray}\label{eq: multiplicative structure of R}
	\omega_1^2 &=& -ac + b \omega_1 - a \omega_2, \nonumber \\ 
	\omega_1 \omega_2 &=& -ad, \\
	\omega_2^2 &=& -bd + d \omega_1 - c \omega_2. \nonumber 
\end{eqnarray}

%Only the quartic rings over a field $k$ are handled in the proof of Theorem~\ref{thm:main result over RationalField}.
%since it discusses the simultanoues representations over $\RationalField$.

\section{A canonical form for elements of $V_k$}
\label{A canonical form for elements of $V_k$}

Herein, the method to obtain a \textit{canonical form} of $(A, B) \in V_k$ is discussed, assuming that ${\rm char}\ k \neq 2, 3$.
\begin{lem}\label{lem:generator of Q otimes RationalField}
For any $h_0 \in k$ and $h := (h_{1}, h_{2}, h_{3}) \in k^3$,
if we put $\alpha := h_0 + \sum_{j=1}^3 h_j \xi_{j} \in \QA{k}{A}{B}$, the following holds:
\begin{eqnarray*}
1, \alpha, \alpha^2, \alpha^3 \textit{ is a basis of } \QA{k}{A}{B} \text{ over } k
\Longleftrightarrow
\det(B(h) A - A(h) B) \neq 0.
\end{eqnarray*}
When $(A, B)$ is non-singular and anisotropic over $k$,
$\QA{k}{A}{B}$ contains such an $\alpha$
if and only if $A, B$ are linearly independent over $k$. 
\end{lem}
\begin{proof}
For any element $\alpha \in \QA{k}{A}{B}$, 
we denote the image of $\alpha$ by the natural epimorphism $\QA{k}{A}{B} \twoheadrightarrow \QA{k}{A}{B} / k \cdot 1$ by $\overline{\alpha}$.
By Eq.(\ref{eq: formula for 4 det(B(x) A - A(x) B}), the following matrix $M$ satisfies $\det M = 4 \det(B(h) A - A(h) B)$:
\begin{eqnarray}
	M
	\begin{pmatrix}
		\overline{\xi_1} \\ \overline{\xi_2} \\ \overline{\xi_3} 
	\end{pmatrix}
 =
	\begin{pmatrix}
		\overline{\alpha} \\ \overline{\alpha^2} \\ \overline{\alpha^3} 
	\end{pmatrix}.
\end{eqnarray}
The first statement is obtained from this.
With regard to the second one, 
the only-if part immediately follows from the fact that $\QA{k}{A}{B}$
is isomorphic to $k[x, y, z] / (x^2, y^2, z^2, xy, yz, xz)$, if $A, B$ are linearly dependent.
We now prove the if part.
If $\det(A x - B y) = 0$ does not have roots $[x : y] \in {\mathbb P}^1(k)$, then 
$\det(B(h) A - A(h) B) = 0 \Leftrightarrow A(h) = B(h) = 0$.
However, none of $0 \neq h \in k^3$ satisfies this, owing to the anisotropy of $(A, B)$.
If $\emptyset \ne \Pi \subset {\mathbb P}^1(k)$ is the set of all the roots of $\det(A x - B y) = 0$,
then, 
\begin{eqnarray*}\label{eq:condition det(B(h) A - A(h) B) = 0}
	\det(B(h) A - A(h) B) = 0 \Longleftrightarrow u A (h) = v B (h) \text{ for some } [u : v] \in \Pi.
\end{eqnarray*}
Since $u A - v B$ is a non-zero quadratic form, 
if the cardinality of $\Pi$ is 1,
some $h \in k^3$ does not satisfy $u A (h) = v B (h)$.
Otherwise, all roots of $\det(A x - B y) = 0$ belong to ${\mathbb P}^1(k)$,
hence, $(A, B)$ is simultaneously diagonalized over $k$ by Lemma \ref{lem: diagonalization over F}.
Thus, it is easily seen that this case is eliminated as well.
\end{proof}

In the following, for fixed $(A, B) \in V_k$, we will take $\alpha$ as in Lemma \ref{lem:generator of Q otimes RationalField}
and put $a x^3 + b x^2 y + c x y^2 + d y^3 := 4 \det (A x - B y)$.
Let ${\rm ch}_\alpha(x) := x^4 + a_{3} x^3 + a_{2} x^2 + a_{1} x + a_{0}$ be the characteristic polynomial of $\alpha \in \QA{k}{A}{B}$. 
From Eq.(\ref{eq: multiplicative structure of Q}), these $a_i$ ($0 \leq i \leq 3$) are represented as an integral polynomial of $h_i$ ($0 \leq i \leq 3$) and the coefficients of $A$, $B$.

It was pointed out in \cite{Bhargava2004} that the image of $\alpha \in Q_\IntegerRing(A, B)$ by the resolvent map equals $z + B(h) \omega_1 + A(h) \omega_2$ for some $z \in \IntegerRing$.
%, and has $f^{res}(x)$ as its characteristic polynomial in $\RA{k}{A}{B}$.
In general, it can be verified by direct calculation that the following equality holds as a polynomial of $h_i$ and the coefficients of $A$, $B$.
\begin{eqnarray}\label{eq: ch_beta^{res}(x)}
{\rm ch}_\alpha^{res}(x) 
&=& (x-z)^3 + ( c A(h) - b B(h) ) (x-z)^2 \nonumber \\
& & + \{ b d A(h)^2 + (3 a d - b c) A(h) B(h)  + a c B(h)^2 \} (x-z) \nonumber \\
& & + a d^2 A(h)^3 - (b^2 d - 2 a c d) A(h)^2 B(h) 
+ (a c^2 - 2 a b d) A(h) B(h)^2 - a^2 d B(h)^3 \nonumber \\
&=& \prod_{i=1}^3 \left( x - z + a B(h) \frac{u^{(i)}}{v^{(i)}} - d A(h) \frac{v^{(i)}}{u^{(i)}} \right), \nonumber \\
\end{eqnarray}
where $[u^{(i)} : v^{(i)}] \in {\mathbb P}^1(\bar{k})$ ($1 \leq i \leq 3$)
are the roots of $4 \det (A x - B y) = 0$. %since $\omega_1$,  $\omega_2$ have the same characteristic polynomial as $a u^{(i)} / v^{(i)}$, $d v^{(i)} / u^{(i)}$, 
%we have
The $z$ equals $(a_2 + c A(h) - b B(h)) / 3$, which is obtained by comparing the coefficients of $x^2$.
% with ${\rm ch}_\alpha^{res}(x) =  x^3 - a_{2} x^2 + (a_{1} a_{3}  - 4 a_{0}) x  + (4 a_{0} a_{2} - a_{1}^2 - a_{0} a_{3}^2)$. 

\begin{lem}\label{prop: monogenic ring}
We assume that $(A, B) \in V_k$ is linearly independent, non-singular and anisotropic over $k$.  
As in Lemma \ref{lem:generator of Q otimes RationalField}, $\alpha \in \QA{k}{A}{B}$ and $h \in k^3$ that satisfies $\det( B(h) A - A(h) B) \neq 0$ are fixed.
When ${\rm ch}_\alpha(x) = x^4 + a_{3} x^3 + a_{2} x^2 + a_{1} x + a_{0}$ is the characteristic polynomial of $\alpha$,
let $\Lambda := \langle 1, \alpha, \alpha^2 + a_{3} \alpha + a_{2}, \alpha^3 + a_{3} \alpha^2 + a_{2} \alpha + a_{1} \rangle$
be a basis of $k[\alpha]$.
Let $(R, \Phi)$ be the cubic $k$-algebra corresponding to the resolvent polynomial ${\rm ch}_\alpha^{res}(x+a_{2}/3)$.
%$${\rm ch}_\alpha^{res}(x+a_{2}/3) 
%= x^3 + (a_{1} a_{3}  - 4 a_{0} - a_{2}^2/3) x  + (8 a_{0} a_{2}/3 - a_{1}^2 - a_{0} a_{3}^2 + a_{0} a_{1} a_{2}/3 - 2 a_{2}^3/27).$$
Then, $(k[\alpha], \Lambda)$ and $(R, \Phi)$ coincide
with the quartic ring $\QA{k}{\tilde{A}}{\tilde{B}}$ and the resolvent cubic $\RA{k}{\tilde{A}}{\tilde{B}}$ of the following $(\tilde{A}, \tilde{B})$.
\begin{eqnarray}\label{eq: definition of tilde{A} and tilde{B}}
\tilde{A} = 
	\begin{pmatrix}
		   0 &   0 & 1/2 \\
		   0 & -1 &     0 \\
		1/2 &  0 &     0 
	\end{pmatrix},\
\tilde{B} = -
	\begin{pmatrix}
		       1 & a_{3}/2  & a_{2}/6 \\
      a_{3}/2 & 2a_{2}/3 & a_{1}/2 \\
	 a_{2}/6 & a_{1}/2  & a_{0}
	\end{pmatrix}.
% 4 \det (\tilde{A} x - \tilde{B} y) = x^3 + (4 a_{2} + 3 z)
\end{eqnarray}
\end{lem}

\begin{proof}
It is straightforward to verify by direct computation that $\Lambda$ satisfies Eq.(\ref{eq: normalization}), and $(\tilde{A}, \tilde{B})$ corresponds to the pair of $(k[\alpha], \Lambda)$ and $(R, \Phi)$.
\end{proof}

In particular, we have $4 \det \left( \tilde{A} x  - \tilde{B} \right) = {\rm ch}_\alpha^{res}(x + a_2/3)$.
We shall construct $(W, V) \in G_k$
such that 
$(\tilde{A}, \tilde{B}) = (W, V) \cdot (A, B)$ holds;
when $\langle 1, \xi_1, \xi_2, \xi_3 \rangle$ is the basis of $\QA{k}{A}{B}$, let $W$ be the matrix uniquely determined by:
\begin{eqnarray*}\label{eq: definition of W}
	W
	\begin{pmatrix}
		\overline{\xi_1} \\ \overline{\xi_2} \\ \overline{\xi_3} 
	\end{pmatrix}
 =
	\begin{pmatrix}
		1  & 0 & 0 \\
		a_{3} & 1 & 0\\
		a_{2} & a_{3}  & 1
	\end{pmatrix}
	\begin{pmatrix}
		\overline{\alpha} \\ \overline{\alpha^2} \\ \overline{\alpha^3} 
	\end{pmatrix}.
\end{eqnarray*}
The determinant is given by $\det W = 4 \det(B(h) A - A(h) B) \neq 0$.
We define $V \in GL_2(k)$ by
\begin{eqnarray*}\label{eq: definition of V}
V
 &:=& \det W^{-1}
\begin{pmatrix}
 1 & 0  \\
 (c A(h) - b B(h))/3 & 1
\end{pmatrix}
\begin{pmatrix}
 B(h) & -A(h)  \\
 -c A(h) B(h) - d A(h)^2 & -a B(h)^2 - b A(h) B(h)
\end{pmatrix},
% 4 \det( (x, - y) \tr{(A, B)} ) = a \prod (x - \frac{u^{(i)}}{v^{(i)}} y)
% 4 \det( (x, - y) \tr{(\tilde{A}, \tilde{B})} )
% = 4 \det M^{-1} \det( ( q_B x + (c q_A q_B + d q_A^2) y, -q_A x + (a q_B^2 + b q_A q_B) y ) (A, B) )
% = \det M^{-1} a \prod ( q_B x + (c q_A q_B + d q_A^2) y - \frac{u^{(i)}}{v^{(i)}} ( q_A x - (a q_B^2 + b q_A q_B) y )  )
% = \det M^{-1} a \prod ( (q_B - u^{(i)}/v^{(i)}*q_A) x + ( c q_A q_B + d q_A^2 + u^{(i)}/v^{(i)}*(a q_B^2 + b q_A q_B) ) y ) 
% = \det M^{-1} 4 \det(q_B A - q_A B) \prod ( x + ( q_B a \frac{u^{(i)}}{v^{(i)}} - q_A d \frac{v^{(i)}}{u^{(i)}} ) y )
\end{eqnarray*}
which is derived from the following equiation obtained by using Eq.(\ref{eq: ch_beta^{res}(x)}) and $4 \det( Ax - B y ) = a \prod_{i=1}^3 (x~-~(u^{(i)} / v^{(i)}) y)$:
%$$a B(h)\frac{u^{(i)}}{v^{(i)}} - d A(h) \frac{v^{(i)}}{u^{(i)}} = \frac{ ((u^{(i)} / v^{(i)}) (a B(h)^2 + b A(h) B(h)) + c A(h) B(h) + d A(h)^2 }{ -(u^{(i)}/v^{(i)}) A(h) + B(h) }.$$
%$(\tilde{A}, \tilde{B}) = (W, V) \cdot (A, B)$ is proved by direct computation.
\begin{eqnarray*}
& & \hspace{-10mm}
{\rm ch}_\alpha^{res} \left( x + \frac{ a_2 + c A(h) - b B(h) }{3} \right) \\
&=& \frac{ 4 \det(B(h) A - A(h) B) }{ \det W } {\rm ch}_\alpha^{res} \left( x + \frac{ a_2 + c A(h) - b B(h) }{3} \right) \\
 &=& \frac{ a }{ \det W } \prod_{i=1}^3 \left( B(h) - \frac{u^{(i)}}{v^{(i)}} A(h) \right)\left( x + a B(h) \frac{u^{(i)}}{v^{(i)}} - d A(h) \frac{v^{(i)}}{u^{(i)}} \right) \\
% &=& \frac{a}{\det W} \prod \left( \left( B(h) -  \frac{u^{(i)}}{v^{(i)}} A(h) \right) x + ( c q_A q_B + d q_A^2 +  \frac{u^{(i)}}{v^{(i)}}  (a q_B^2 + b q_A q_B) ) \right) \\
 &=& \frac{a}{\det W} \prod_{i=1}^3 \left\{ B(h) x + c A(h) B(h) + d A(h)^2 - \frac{u^{(i)}}{v^{(i)}} (A(h) x - a B(h)^2 - b A(h) B(h)) \right\} \\
 &=& (\det W^{-1})\ 4 \det \left( (B(h) x + c A(h) B(h) + d A(h)^2) A + (-A(h) x + a B(h)^2 + b A(h) B(h)) B \right).
\end{eqnarray*}

As a result of Lemma \ref{prop: monogenic ring}, the following corollary is immediately obtained:
\begin{cor}\label{cor: (A, B) is isotropic over k}
We assume that $(A, B)$ and $\alpha$ are taken as in Lemma \ref{prop: monogenic ring}.
For any field $k \subset K \subset \bar{k}$,
$(A, B)$ is isotropic over $K$
if and only if  ${\rm ch}_\alpha(x) = 0$ has a root in $K$.
\end{cor}
\begin{proof}
Using the action of $G_k$,
we may replace $(A, B)$ with $(\tilde{A}, \tilde{B})$ in Eq.(\ref{eq: definition of tilde{A} and tilde{B}}).
This is proved as follows:
\begin{eqnarray*}
(\tilde{A}, \tilde{B})({\mathbf x}) = 0 \text{ for some } 0 \neq {\mathbf x} \in K^3
&\Leftrightarrow& (\tilde{A}, \tilde{B})(s^2, st, t^2) = 0 \text{ for some } [s : t] \in {\mathbb P}^1(K) \\
&\Leftrightarrow& {\rm ch}_\alpha(u) = 0 \text{ for some } u \in K.
\end{eqnarray*}
\end{proof}

As proved in Section~\ref{Proof of Theorem ref{thm:main result over RationalField}},
any $(A_1, B_1)$, $(A_2, B_2) \in V_k$ can be simultaneously transformed to pairs of the form of Eq.(\ref{eq: definition of tilde{A} and tilde{B}}),
if $\det (A_1 x - B_1 y) = c \det (A_2 x - B_2 y)$ for some $c \in k^\times$, and some $(q_A, q_B) \in q_k(A_1, B_1) \cap q_k(A_2, B_2)$ satisfies $\det (q_B A_i - q_A B_i) \neq 0$.

%\section{Automorphism group of $(A, B) \in V_k$}
%
%In what follows, let $k$ be a field and $\bar{k}$ be the algebraic closure of $k$.
%For any $(A, B) \in V_k$, we 
%define 
%\begin{eqnarray}
%{\rm Aut}_k(A, B) &:=& \{ (g_1, g_2) \in G_k: (g_1, g_2) \cdot (A, B) = (A, B) \},\\
%{\rm Aut}_k^1(A, B) &:=& \{ g \in GL_n(k): \det g = 1, (g, 1) \in {\rm Aut}_k(A, B) \},
%\end{eqnarray}
%where the condition $\det g = 1$ is used so that either of $\pm g$ belongings to ${\rm Aut}_k^1(A, B)$ when $n$ is odd.
%
\section{Automorphisms of cubic polynomials}
\label{Automorphisms of cubic polynomial}

The purpose of this section is to prove Lemma \ref{lem:automorphism of (A, B) with V} 
which is used in the proof of Theorem \ref{thm:main result over RationalField}.
The assumption ${\rm char}\ k \neq 2, 3$ is also used herein.

\begin{lem}\label{lem:automorphism of cubic polynomial}
If a cubic polynomial $f (x, y) = a x^3 + b x^2 y + c x y^2 + d y^3 \in k[x, y]$ has 
no multiple roots and satisfies $(\det V)^{-1} f ( (x, y) V ) = u f(x, y)$ for some $V \in GL_2(k)$ and $u \in k$, then 
$V^n = u^n I$ for at least one of $n = 1, 2, 3$.
\end{lem}

\begin{proof}
Let $[p_i : q_i] \in {\mathbb P}^1(\bar{k})$ ($i = 1, 2, 3$) be the roots of $f(x, y) = 0$.
$V \in GL_2(k)$ exchanges $[p_i : q_i]$. Hence, $V^n = v I$ for some $n = 1, 2, 3$ and $v \in k$.
From  
$(\det V^n)^{-1} f((x, y)V^n) = v f(x, y) = u^n f(x, y)$,
$v=u^n$ is obtained.
\end{proof}

In what follows, $p_i, q_i$ that satisfy $f(x, y) = \prod_{i=1}^3 (q_i x - p_i y)$ are fixed.
The following examples list all the cases in which $V \neq I$.

\begin{exa}[Case $V^2 = u^2 I$]
It may be assumed that $V$ swaps $[p_1 : q_1]$ and $[p_2 : q_2]$, and fixes $[p_3 : q_3]$.
%If we put $r = \frac{ p_3 q_1 - p_1 q_3 }{ p_2 q_3 - p_3 q_2 }$,
For some $0 \neq r_1, r_2 \in \bar{k}$,
\begin{eqnarray}
V
&=& u
\begin{pmatrix}
	p_1 & q_1 \\
	p_2 & q_2
\end{pmatrix}^{-1}
\begin{pmatrix}
	0 & r_2 \\
	r_1 & 0
\end{pmatrix}
\begin{pmatrix}
	p_1 & q_1 \\
	p_2 & q_2
\end{pmatrix}.
\end{eqnarray}
From $V^2 = u^2 I$, $r_1 r_2 = 1$ is obtained.
$V \begin{pmatrix} q_3 \\ -p_3 \end{pmatrix} = -u \begin{pmatrix} q_3 \\ -p_3 \end{pmatrix}$ is also obtained 
from $(\det V)^{-1} f((x, y)V) = - u^{-2} f((x, y)V) = u f(x, y)$.
%  [p_1 r_1 - p_2 : -q_1 r_1 + q_2 ] = [ p_3 : - q_3 ]
% -q_3 (p_1 r_1 - p_2) = p_3 (-q_1 r_1 + q_2)
Hence, $r_1 = r_2^{-1} = -(p_2 q_3 - q_2 p_3) / (p_1 q_3 - q_1 p_3)$.
Thus,
\begin{eqnarray*}
& & \hspace{-10mm}
-(p_2 q_3 - q_2 p_3) (p_1 q_3 - q_1 p_3) V \nonumber \\
&=& \frac{u}{ p_1 q_2 - p_2 q_1 }
\begin{pmatrix}
	q_2 &  -q_1 \\
	-p_2 & p_1 
\end{pmatrix} 
\begin{pmatrix}
	0 & (q_1 p_3 - p_1 q_3)^2 \\
	(q_2 p_3 - p_2 q_3)^2 & 0
\end{pmatrix}
\begin{pmatrix}
	p_1 & q_1 \\
	p_2 & q_2
\end{pmatrix}
\nonumber \\
&=& u
\begin{pmatrix}
p_1 p_2 q_3^2 - p_3^2 q_1 q_2
 & (p_1 q_2  + p_2 q_1) q_3^2 - 2 p_3 q_1 q_2 q_3 \\
(p_1 q_2 + p_2 q_1) p_3^2 - 2 p_1 p_2 p_3 q_3
 & -p_1 p_2 q_3^2 + p_3^2 q_1 q_2
\end{pmatrix}.
\end{eqnarray*}
Therefore, if we put 
$C := -1/(p_2 q_3 - q_2 p_3) (p_1 q_3 - q_1 p_3)
	p_3 / \frac{\partial f}{ \partial y }(p_3, q_3)$,
\begin{eqnarray}\label{eq:definition of V}
%&=& c
%\begin{pmatrix}
%- (p_1 q_2 q_3 + p_2 q_1 q_3  + p_3 q_1 q_2) p_3 + (p_1 p_2 q_3 + p_1 p_3 q_2 + p_2 p_3 q_1) q_3
% &  (p_1 p_2 q_3 + p_1 p_3 q_2 + p_2 p_3 q_1) p_3 - 3 p_1 p_2 p_3 q_3 \\
% - 3 p_3 q_1 q_2 q_3 + (p_1 q_2 q_3 + p_2 q_1 q_3 + p_3 q_1 q_2) q_3
% & (p_1 q_2 q_3 + p_2 q_1 q_3 + p_3 q_1 q_2) p_3 - (p_1 p_2 q_3 + p_1 p_3 q_2 + p_2 p_3 q_1) q_3
%\end{pmatrix} \nonumber \\V 
V
&=&  
u C
\begin{pmatrix}
 b p_3 + c q_3 
 & -3 a p_3 - b q_3 \\
  c p_3 + 3 d q_3 
 & -b p_3 - c q_3
\end{pmatrix}.
\end{eqnarray}
If $[p_3 : q_3] \notin {\mathbb P}^1(k)$,
$V \in GL_2(k)$ implies that 
$3 a c = b^2$ and $3 b d = c^2$, then $f(x, y) = (3b c)^{-1} (b x + c y)^3$. 
Since $f$ is assumed to have no multiple roots, we obtain $[p_3 : q_3] \in {\mathbb P}^1(k)$.
%It is straightforward to prove the remaining part.
\end{exa}

\begin{exa}[Case $V^3 = u^3 I$]
It may be assumed that 
$V$ maps $[p_i : q_i]$ to $[p_j : q_j]$ for any $(i, j) = (1, 2), (2, 3), (3, 1)$.
Since $V^3 = u^3 I$, for some $c_i \in \bar{k}$ with $c_1 c_2 c_3 = u^3$, we have 
\begin{eqnarray}
\begin{pmatrix}
	p_1 & q_1 \\
	p_2 & q_2 \\
	p_3 & q_3 \\
\end{pmatrix}
V
=
\begin{pmatrix}
0 & c_2 & 0 \\
0 & 0 & c_3 \\
c_1 & 0 & 0 
\end{pmatrix}
\begin{pmatrix}
	p_1 & q_1 \\
	p_2 & q_2 \\
	p_3 & q_3 \\
\end{pmatrix}.
\end{eqnarray}

We have $\det V = c_{i_2} c_{i_3} (p_{i_2} q_{i_3} - p_{i_3} q_{i_2}) / (p_{i_1} q_{i_2} - p_{i_2} q_{i_1})$
for any $(i_1, i_2, i_3) = (1, 2, 3), (2, 3, 1), (3, 1, 2)$. Hence,

\begin{eqnarray}
\begin{pmatrix}
c_1 \\
c_2 
\end{pmatrix}
= 
c_3 \frac{p_3 q_1 - p_1 q_3}{p_1 q_2 - p_2 q_1}
\begin{pmatrix}
\frac{p_2 q_3 - p_3 q_2}{p_1 q_2 - p_2 q_1} \\
\frac{p_3 q_1 - p_1 q_3}{p_2 q_3 - p_3 q_2}
\end{pmatrix}.
\end{eqnarray}

From $c_1 c_2 c_3 = u^3$, 
some $\zeta \in \bar{k}$ with $\zeta^3 = 1$
satisfies $c_3 = u \zeta (p_1 q_2 - p_2 q_1)/(p_3 q_1 - p_1 q_3)$.
Thus, 
\begin{eqnarray}
V
&=& u \zeta
\begin{pmatrix}
	p_1 & q_1 \\
	p_2 & q_2
\end{pmatrix}^{-1}
\begin{pmatrix}
\frac{p_3 q_1 - p_1 q_3}{p_2 q_3 - p_3 q_2} & 0 \\
0 & \frac{p_1 q_2 - p_2 q_1}{p_3 q_1 - p_1 q_3}
\end{pmatrix}
\begin{pmatrix}
	p_2 & q_2 \\
	p_3 & q_3
\end{pmatrix}.
\end{eqnarray}

From $\det V = u^2 \zeta^2 \in k$,% and $(\det V)^{-1} f((x, y)V) = u f(x, y)$
$\zeta \in k$ is obtained. 
Hence, $u$ may be replaced with $u \zeta^{-1}$.
If we put $\Delta = (p_1 q_2 - p_2 q_1)(p_2 q_3 - p_3 q_2)(p_3 q_1 - p_1 q_3)$,
we have ${\rm Disc}(f(x, 1)) = \Delta^2$ and 
\begin{eqnarray*}
\Delta &=& -2 (p_1^2 p_2 q_2 q_3^2 + p_2^2 p_3 q_1^2 q_3 + p_1 p_3^2 q_1 q_2^2) - b c + 3 a d \\
	   &=& 2 (p_1 p_2^2 q_3^2 q_1 + p_3 p_1^2 q_2^2 q_3 + p_2 p_3^2 q_1^2 q_2) + b c - 3 a d, \\
V
%\frac{k}{\Delta} 
%\begin{pmatrix}
%p_1^2 p_2 q_2 q_3^2 + p_3^2 p_1 q_1 q_2^2  + p_2^2 p_3 q_3 q_1^2 - 3 p_1 p_2 p_3 q_1 q_2 q_3
%&
%(p_1 p_2 q_3 + p_1 p_3 q_2 + p_2 p_3 q_1)^2 - 3 p_1 p_2 p_3 (p_1 q_2 q_3 + p_2 q_1 q_3 + p_3 q_1 q_2) \\
%-(p_1 q_2 q_3 + p_2 q_1 q_3 + p_3 q_1 q_2)^2 + 3 (p_1 p_2 q_3 + p_1 p_3 q_2 + p_2 p_3 q_1) q_1 q_2 q_3
%&
% - p_1 p_2^2 q_3^2 q_1 - p_3 p_1^2 q_2^2 q_3 - p_2 p_3^2 q_1^2 q_2 + 3 p_1 p_2 p_3 q_1 q_2 q_3
%\end{pmatrix} \\
&=&
\frac{u}{\Delta}
\begin{pmatrix}
(9 a d - b c - \Delta)/2
&
b^2 - 3 a c \\
-c^2 + 3 b d
&
-(9 a d - b c + \Delta)/2
\end{pmatrix}.
\end{eqnarray*}
Since $V \in GL_2(k)$, we have  
$\Delta \in k$.
Therefore, $f(x, y) = 0$ completely splits over $k$, or 
$k[x]/(f(x, 1))$ is a Galois cubic field over $k$.
\end{exa}

In the following lemma, $4 \det (A x - B y)$ is denoted by $f_{det}(x, y)$.

\begin{lem}\label{lem:automorphism of (A, B) with V}
We assume that $(A, B) \in V_k$ with ${\rm Disc}(A, B) \neq 0$ (hence, linearly independent), is non-singular and anisotropic over $k$. 
We further assume that there are $u \in k$ and a matrix $u I \ne V \in GL_2(k)$
such that $q_k(A, B) = q_k((I, V) \cdot (A, B))$ and 
$$
	({\det V})^{-1} f_{det}((x, y) \tilde{V}) = u f_{det}(x, y), \quad 
\tilde{V} :=
\begin{pmatrix}
	1 & 0 \\
	0 & -1 \\
\end{pmatrix}
V
\begin{pmatrix}
	1 & 0 \\
	0 & -1 \\
\end{pmatrix}.
$$
Then, 
there exists 
$W \in GL_3(k)$ such that $(W, -u^{-1} V) \cdot (A, B) = (A, B)$. 
\end{lem}

\begin{proof}
Using the action of $GL_2(k)$,
we may assume $\det A \neq 0$.
Let $\alpha_1, \alpha_2, \alpha_3 \in \bar{k}$ be the roots of $\det(A x - B) = 0$, and $k_i$ be the field $k(\alpha_i)$.
We fix $0 \neq v_i \in k_i^3$, $d_i \in k_i$ ($1 \leq i \leq 3$)
and $w = \tr{(v_1\ v_2\ v_3)}$
as in the first paragraph in the proof of Lemma \ref{lem: equivalence conditions on isotropy}.
$(A, B)$ is transformed as in Eq.(\ref{eq: transform of (A,B)}).

From $\tilde{V} \ne u I$, 
there are the following two cases:
\begin{itemize}
\item ($\tilde{V}^2 = V^2 = u^2 I$)
It may be assumed that $V$ exchanges $[\alpha_1 : -1]$ and $[\alpha_2 : -1]$.
We then have $k_1 = k_2$ and $\alpha_3 \in k$. Furthermore, 
\begin{eqnarray}
V
&=& u
\begin{pmatrix}
	\alpha_1 & -1 \\
	\alpha_2 & -1
\end{pmatrix}^{-1}
\begin{pmatrix}
	0 & -\frac{\alpha_1 - \alpha_3}{\alpha_2 - \alpha_3} \\
	-\frac{\alpha_2 - \alpha_3}{\alpha_1 - \alpha_3} & 0
\end{pmatrix}
\begin{pmatrix}
	\alpha_1 & -1 \\
	\alpha_2 & -1
\end{pmatrix}.
\end{eqnarray}
Hence,
\begin{footnotesize}
\begin{eqnarray*}
	& & \hspace{-10mm}
	\left( w, 
		\begin{pmatrix}	
			\alpha_1 & -1 \\
			\alpha_2 & -1 
		\end{pmatrix} V
	\right) 
	\cdot
	(A, B) \nonumber \\
	&=&
	\left( I, u
\begin{pmatrix}
	0 & -\frac{\alpha_1 - \alpha_3}{\alpha_2 - \alpha_3} \\
	-\frac{\alpha_2 - \alpha_3}{\alpha_1 - \alpha_3} & 0
\end{pmatrix}
	\right) 
	\cdot
	\left( 
		\begin{pmatrix}	
			0 & 0 & 0 \\
			0 & d_2 (\alpha_1 - \alpha_2) & 0 \\
			0 & 0 & d_3 (\alpha_1 - \alpha_3)
		\end{pmatrix},
		\begin{pmatrix}	
			d_1 (\alpha_2 - \alpha_1) & 0 & 0 \\
			0 & 0 & 0 \\
			0 & 0 & d_3 (\alpha_2 - \alpha_3)
		\end{pmatrix}	
	\right) \\
	&=&
	\left( 
		u
		\begin{pmatrix}	
			\frac{d_1(\alpha_1 - \alpha_3)}{d_2(\alpha_2 - \alpha_3)} \cdot d_2 (\alpha_1 - \alpha_2) & 0 & 0 \\
			0 & 0 & 0 \\
			0 & 0 & -d_3 (\alpha_1 - \alpha_3)
		\end{pmatrix},	
		u
		\begin{pmatrix}	
			0 & 0 & 0 \\
			0 & \frac{d_2(\alpha_2 - \alpha_3)}{d_1(\alpha_1 - \alpha_3)} \cdot d_1 (\alpha_2 - \alpha_1) & 0 \\
			0 & 0 & -d_3 (\alpha_2 - \alpha_3)
		\end{pmatrix}
	\right).
\end{eqnarray*}
\end{footnotesize}
Owing to $q_k(A, B) = q_k((I, V) \cdot (A, B))$,
for any primes ${\mathfrak p}$ of $k$ that completely splits in $k_1 = k_2$, we must have
\begin{eqnarray*}
q_{k_{\mathfrak p}}\left( 
		\begin{pmatrix}	
			d_2 (\alpha_1 - \alpha_2) & 0 \\
			0 & d_3 (\alpha_1 - \alpha_3)
		\end{pmatrix}
	\right) 
 &=& q_{k_{\mathfrak p}} \left( u  
		\begin{pmatrix}	
			\frac{d_1(\alpha_1 - \alpha_3)}{d_2(\alpha_2 - \alpha_3)} \cdot d_2 (\alpha_1 - \alpha_2) & 0 \\
			0 & -d_3 (\alpha_1 - \alpha_3)
		\end{pmatrix}
	\right), \\
q_{k_{\mathfrak p}}\left( 
		\begin{pmatrix}	
			d_1 (\alpha_2 - \alpha_1) & 0 \\
			0 & d_3 (\alpha_2 - \alpha_3)
		\end{pmatrix}	
	\right) 
 &=& q_{k_{\mathfrak p}}\left(
		u
		\begin{pmatrix}	
			\frac{d_2(\alpha_2 - \alpha_3)}{d_1(\alpha_1 - \alpha_3)} \cdot d_1 (\alpha_2 - \alpha_1) & 0 \\
			0 & -d_3 (\alpha_2 - \alpha_3)
		\end{pmatrix}
	\right).
\end{eqnarray*}
This implies that there exists $\beta \in k_i$ such that $- d_1(\alpha_1-\alpha_3)/d_2(\alpha_2~-~\alpha_3)~=~\beta^2$
and $\left( W, 
		-u^{-1} V
	\right) 
	\cdot
	(A, B)
		=
	(A, B)$ if we put:
\begin{eqnarray*}
		W = w^{-1} 
		\begin{pmatrix}	
			0 & \beta & 0 \\
			\pm 1/\beta & 0 & 0 \\
			0 & 0 & 1 
		\end{pmatrix} w.
\end{eqnarray*}

If $k_1 = k_2 = k$, then $W \in GL_3(k)$ has the required property.
Otherwise, $k_1 = k_2$ is quadratic over $k$. 
If the signature of $\pm 1/\beta$ is chosen so that $\beta$ and $\pm 1/\beta$ are conjugate over $k$, 
$W \in GL_3(k)$ is obtained.

\item ($\tilde{V}^3 = V^3 = u^3 I$)
In this case, $k_1 = k_2 = k_3 = k$ or a Galois cubic field over $k$. As shown in the above example,
\begin{eqnarray*}
V
&=& u
\begin{pmatrix}
	\alpha_1 & -1 \\
	\alpha_2 & -1
\end{pmatrix}^{-1}
\begin{pmatrix}
\frac{\alpha_3 - \alpha_1}{\alpha_2 - \alpha_3} & 0 \\
0 & \frac{\alpha_1 - \alpha_2}{\alpha_3 - \alpha_1}
\end{pmatrix}
\begin{pmatrix}
	\alpha_2 & -1 \\
	\alpha_3 & -1
\end{pmatrix}, \\
&=& u
\begin{pmatrix}
	\alpha_1 & -1 \\
	\alpha_2 & -1
\end{pmatrix}^{-1}
\begin{pmatrix}
	0 & \frac{\alpha_3 - \alpha_1}{\alpha_2 - \alpha_3} \\
	-\frac{\alpha_2 - \alpha_3}{\alpha_3 - \alpha_1} & -1
\end{pmatrix}
\begin{pmatrix}
	\alpha_1 & -1 \\
	\alpha_2 & -1
\end{pmatrix}.
\end{eqnarray*}
Hence,
\begin{small}
\begin{eqnarray*}
	& & \hspace{-10mm}
	\left( w, 
		\begin{pmatrix}	
			\alpha_1 & -1 \\
			\alpha_2 & -1 
		\end{pmatrix} V
	\right) 
	\cdot
	(A, B) \nonumber \\
	&=&
	\left( I, u
\begin{pmatrix}
	0 & \frac{\alpha_3 - \alpha_1}{\alpha_2 - \alpha_3} \\
	-\frac{\alpha_2 - \alpha_3}{\alpha_3 - \alpha_1} & -1
\end{pmatrix}
	\right) 
	\cdot
	\left( 
		\begin{pmatrix}	
			0 & 0 & 0 \\
			0 & d_2 (\alpha_1 - \alpha_2) & 0 \\
			0 & 0 & d_3 (\alpha_1 - \alpha_3)
		\end{pmatrix},
		\begin{pmatrix}	
			d_1 (\alpha_2 - \alpha_1) & 0 & 0 \\
			0 & 0 & 0 \\
			0 & 0 & d_3 (\alpha_2 - \alpha_3)
		\end{pmatrix}	
	\right) \\
	&=&
	\left( -u
		\begin{pmatrix}	
			\frac{d_1(\alpha_3 - \alpha_1)}{d_2(\alpha_2 - \alpha_3)} \cdot d_2 (\alpha_1 - \alpha_2) & 0 & 0 \\
			0 & 0 & 0 \\
			0 & 0 & d_3 (\alpha_1 - \alpha_3)
		\end{pmatrix},	
		-u
		\begin{pmatrix}	
			d_1 (\alpha_2 - \alpha_1) & 0 & 0 \\
			0 & \frac{d_2(\alpha_1 - \alpha_2)}{d_3(\alpha_3 - \alpha_1)} \cdot d_3 (\alpha_2 - \alpha_3) & 0 \\
			0 & 0 & 0
		\end{pmatrix}
	\right).
\end{eqnarray*}
\end{small}

Then, there exits $\beta_{i_1} \in k_1 = k_2 = k_3$
such that $d_{i_2} (\alpha_{i_1} - \alpha_{i_2}) / d_{i_3} (\alpha_{i_1}~-~\alpha_{i_3})~=~\beta_{i_1}^2$ 
for every $(i_1, i_2, i_3) = (1, 2, 3), (2, 3, 1), (3, 1, 2)$.
$\left( W, 
		-u^{-1} V
	\right)~\cdot~(A, B)~=~(A, B)$ and $W \in GL_3(k)$ hold, if we put:
\begin{eqnarray*}
		W = w^{-1}
		\begin{pmatrix}	
			0 & \beta_3 & 0 \\
			0 & 0 & \beta_1 \\
			\beta_2 & 0 & 0 
		\end{pmatrix} 
		w.
\end{eqnarray*}
Consequently, the lemma is proved.
\end{itemize}

\end{proof}

\section{Proof of Theorem \ref{thm:main result over RationalField}}
\label{Proof of Theorem ref{thm:main result over RationalField}}

In what follows, we assume that $k = \RationalField$. 
In order to prove the theorem, we will first show that $q_\RationalField(A_1, B_1) = q_\RationalField(A_2, B_2)$
implies that either of the determinants is a constant multiple of the other.

\begin{prop}\label{prop:same det(Ax+By)}
If $(A_i, B_i) \in V_\RationalField$ ($i = 1, 2$) satisfy the conditions (\ref{item: assumption (a)}), (b') and (c') of Theorem \ref{thm:main result over RationalField}, 
there are coprime integers $r_1, r_2$ such that $r_1^{-1} \det (A_1 x - B_1 y) = r_2^{-1} \det (A_2 x - B_2 y)$.
\end{prop}

\begin{proof}
Using the action of $GL_2(\RationalField)$ on $V_\RationalField$, we may assume that $\det A_1 \neq 0$.
If ${\rm Disc}(A_1, B_1) = 0$,
let $\alpha \in \RationalField$ be the multiple root of 
$\det(A_1 x - B_1) = 0$.
In this case, $A_1 \alpha - B_1$ has rank $\leq 1$ (Lemma \ref{lem: diagonalization over F}).
Owing to $q_\RationalField(A_1 \alpha - B_1) = q_\RationalField(A_2 \alpha - B_2)$, $A_2 \alpha - B_2$
has the same rank.  Therefore, $\alpha$ is a multiple root of $\det(A_2 x - B_2) = 0$.
Let $\beta \neq \alpha$ be another root of $\det(A_1 x - B_1) = 0$.
Since $(A_1, B_1)$ is anisotropic over $\RationalField$,
$C \in {\rm Sym}^2 (\RationalField^2)^*$ with $A_1 \beta - B_1 \sim_{\RationalField} C \perp [ 0 ]$ must be anisotropic over $\RationalField$. 
Hence, $A_2 \beta - B_2$ must have rank 2, and $\det(A_2 \beta - B_2) = 0$.
Thus, the proposition is proved if ${\rm Disc}(A_1, B_1)=0$.
The same holds if ${\rm Disc}(A_2, B_2)=0$.

We next assume that ${\rm Disc}(A_i, B_i) \neq 0$ ($i = 1, 2$).
Let $\alpha \in \bar{\RationalField}$ be a root of $\det(A_1 x - B_1) = 0$
and $K$ be the Galois closure of $k(\alpha)$ over $\RationalField$. 
If $C \in {\rm Sym}^2 (K^2)^*$ with $A_1 \alpha - B_1 \sim_{K} C \perp [ 0 ]$ is anisotropic over $K$, then
there are a finite prime $p$ of $\RationalField$ and an embedding $\iota: K \hookrightarrow \RationalField_p$ such that
$\iota(C)$ is anisotropic over $\RationalField_p$.
In this case, as a result of Corollary \ref{cor: anisotropic iff anisotropic},
$\det(A_2 \alpha - B_2) = 0$ is proved
by $q_{\RationalField_p}(A_2 \iota(\alpha) - B_2) = q_{\RationalField_p}(A_1 \iota(\alpha) - B_1)
 = \{ 0 \} \cup q_{\RationalField_p}(\iota(C))$.

Thus, by Lemma \ref{lem: equivalence conditions on isotropy},
the equations $\det(A_i x - B_i) = 0$ ($i = 1, 2$)
have at least two common roots $\alpha \neq \beta$.
In addition, if there exist distinct $\gamma_1 \neq \gamma_2$ such that $\det(A_i \gamma_i - B_i) = 0$ ($i = 1, 2$),
both $\gamma_i$ belong to $\RationalField$ and both $A_i \gamma_i - B_i$ are isotropic over $\RationalField$.
Thus, by using the action of $G_\RationalField$,
we may assume the following:
\begin{eqnarray}
(A_1, B_1) &=& 
	\left(
		\begin{pmatrix}
			a_{11} & a_{12} & 0 \\
			a_{12} & a_{22} & 0 \\
			0 & 0 & a_{33}
		\end{pmatrix},
		\begin{pmatrix}
			0 & c_1 & 0 \\
			c_1 & 0 & 0 \\
			0 & 0 & 0
		\end{pmatrix}
	\right), \\
(A_2, B_2) &=& 
	\left(
		\begin{pmatrix}
			0 & c_2 & 0 \\
			c_2 & 0 & 0 \\
			0 & 0 & 0
		\end{pmatrix},
		\begin{pmatrix}
			b_{11} & b_{12} & 0 \\
			b_{12} & b_{22} & 0 \\
			0 & 0 & b_{33}
		\end{pmatrix}
	\right).
\end{eqnarray}
Furthermore, by using the action of $G_\RationalField$,
we may assume that $a_{11} = b_{11} = 1$ and $c_1 = c_2 = -1/2$.
For the anisotropy of $(A_i, B_i)$ over $\RationalField$,
$-a_{33}, -b_{33}, -a_{22} a_{33}, -b_{22} b_{33} \notin (\RationalField^\times)^2$ is required.
In this case, if we put $g(x) := (x - \alpha)(x - \beta) \in \RationalField[x]$, then
\begin{eqnarray*}
	4 \det ( A_1 x - B_1 ) &=& -\{ 4 (a_{12}^2 - a_{22}) x^2 + 4 x a_{12} + 1 \} a_{33} x = -4 (a_{12}^2 - a_{22}) a_{33} x g(x), \\
	4 \det ( A_2 x -  B_2 ) &=& \{ x^2 + 4 x b_{12} + 4 (b_{12}^2 - b_{22}) \} b_{33} = b_{33} g(x).
\end{eqnarray*}
In particular, we have $4 (a_{12}^2 - a_{22}) b_{12} = a_{12}$ and $4 (b_{12}^2 - b_{22}) a_{12} = b_{12}$.
From this, we obtain
$(a_{22} - a_{12}^2)/(b_{22} - b_{12}^2) = a_{22} / b_{22} = a_{12}^2 / b_{12}^2 \in (\RationalField^\times)^2$.
Furthermore, for any $0 \neq x \in \RationalField$ we have
 \begin{eqnarray}
 A_1 x - B_1
 	&\sim_\RationalField& \begin{pmatrix}
 			x & 0 & 0 \\
 			0 & -4 (a_{12}^2 - a_{22}) g(x) / x & 0 \\
 			0 & 0 & a_{33} x
 		\end{pmatrix}, \\
 A_2 x - B_2
 	&\sim_\RationalField& 
 		\begin{pmatrix}
 			-1 & 0 & 0 \\
 			0 & g(x) & 0 \\
 			0 & 0 & -b_{33}
 		\end{pmatrix}.
 \end{eqnarray}

Let $P_0$ be the set of odd primes $p$ of $\RationalField$
such that $p$ completely splits in $\RationalField[x] / (g(x))$ (\IE $a_{22}, b_{22} \in (\RationalField_{p}^\times)^2$).
For a fixed $p \in P_0$,
we will denote the roots of $g(x) = 0$ in $\RationalField_{p}$ by $\alpha_{p}$ and $\beta_{p}$.
Let $P \subset P_0$ be the subset consisting of all $p \in P$ with 
$4 b_{12} \in \IntegerRing_p$ and 
$a_{33}$, $b_{33}$, $4 (a_{12}^2 - a_{22}), 4(b_{12}^2 - b_{22}) \in \IntegerRing_p^\times$.
In this case, for any $p \in P$, $\alpha_p$ and $\beta_p$ belong to $\IntegerRing_p^\times$.
By setting $z$ to an element of $\RationalField_p$ close to $\alpha_{p}$ (\textit{resp.} $\beta_{p}$),
$g(z) \in p \IntegerRing^\times$ and $4 (a_{12}^2 - a_{22}) / z \in \beta_p^{-1} (\IntegerRing_p^\times)^2$
(\textit{resp.} $\alpha_{p}^{-1} (\IntegerRing_p^\times)^2$) can be assumed.
$\left( \frac{-a_{33} }{p} \right) = \left( \frac{-b_{33}}{p} \right)$ is then obtained from $q_{\RationalField_p}(A_1 x - B_1) = q_{\RationalField_p}(A_2 x - B_2)$. In addition, either of the following holds:
\begin{enumerate}[(i)]
	\item $\left( \frac{-a_{33} }{p} \right) = \left( \frac{-b_{33}}{p} \right) = 1$.
	\item $\left( \frac{ -\alpha_p }{p} \right) = \left( \frac{-\beta_p }{p} \right) = 1$. 
\end{enumerate}

Let $K \subset \bar{\RationalField}$ be the extension of $\RationalField$ obtained by attaching $\sqrt{b_{22}}$ to 
$\RationalField$, \IE the splitting field of $g(x)$ over $\RationalField$.
If $K(\sqrt{-b_{33}})$ is denoted by $F_1$, we have $F_1 \supsetneq K$, since $-b_{33}, -b_{22} b_{33} \notin (\RationalField^\times)^2$.
Let $F_2$ be the extensions of $K$ that are obtained by attaching the roots of $g(-x^2) = 0$ to $K$.
$F_2 = K$ is proved as follows;
if $F_2 \supsetneq K$, let $F_3$ be the composition $F_1 F_2$.
Let $Q_i$ ($i = 1, 2, 3$) be the set of all primes of $K$ that completely splits in $F_i$. 
Since the extension $F_i / K$ is Galois,
the Kronecker density 
$d_{i} := \lim_{s \downarrow 1} \frac{ \sum_{p \in Q_i} p^{-s} }{ \log (1/(s - 1)) }$
equals $1 / [F_i : K]$ (Theorem 8.41 (2), \cite{Kato2011}).
However, from $p \in P \Rightarrow$ (i) or (ii), we obtain $d_1 + d_2 \geq 1$.
Owing to $d_1 = 1/2$, in order that $d_3 > 0$, $d_2 = 1$, \IE $K = F_2$ is required.

As a result, $-\alpha = \delta_1^2, -\beta = \delta_2^2$ for some $\delta_1, \delta_2 \in K^\times$,
and $\delta_1+\delta_2, \delta_1\delta_2 \in \RationalField$ may be assumed.
We also have the following:  
\begin{eqnarray*}
a_{12} = \frac{\delta_1^{-2} + \delta_2^{-2}}{4}, \quad a_{22} = \left(\frac{\delta_1^{-2} - \delta_2^{-2}}{4} \right)^2, \quad
b_{12} = \frac{\delta_1^2 + \delta_2^2}{4}, \quad b_{22} = \left(\frac{\delta_1^2 - \delta_2^2}{4} \right)^2. 
\end{eqnarray*}

We shall show that $a_{22}, b_{22} \in (\RationalField^\times)^2$, and therefore, $\delta_1, \delta_2 \in \RationalField$, owing to $K = F_2 = \RationalField$;
if we put $c := (\delta_1^{-1} + \delta_2^{-1})/2$ and 
$d := \delta_1^{-1} \delta_2^{-1}/2$, then $c, d$ are rational, and satisfy $d^2 = a_{12}^2 - a_{22}$, $a_{12} + d = c^{2}$ and $a_{12} - d = a_{22}/c^2$.
Furthermore, $A_1({\mathbf x}) = 0$
if and only if ${\mathbf x} = (x_1, x_2, x_3) \in \RationalField$
satisfies $(x_1 + (a_{12} + d) x_2)(x_1 + (a_{12} - d) x_2) = -a_{33} x_3^2$.  Hence,
if $x_1 + (a_{12} + d) x_2 \ne 0$, there are $s, 0 \ne t \in \RationalField$ such that 
$-2d x_1 = \{ a_{12} - d + (a_{12} + d) a_{33} s^2 \} t$, 
$2d x_2 = (1 + a_{33} s^2) t$ and $x_3 = st$.
Thus, for a field $F \supset \RationalField$, if $T_{1, F}, T_{2, F}$ are defined as follows,
$T_{1, \RationalField} = T_{2, \RationalField}$ is obtained by considering the representations $(0, *)$:
% x_1 + (a_{12} + d) x_2 = y, 
% x_1 + (a_{12} - d) x_2 = -a_{33} x_3^2/y
% => 
% 2 d x_1 = -(a_{12} - d) y - (a_{12} + d) a_{33} x_3^2/y,
% 2 d x_2 = (y + a_{33} x_3^2/y)
%
% x_1 + (a_{12} + d) x_2 = 0, x_3 = 0
% => 
% -x_1*x_2/4 = (a_{12} + d)*x_2^2/4
\begin{eqnarray*}
T_{1, F}
 &:=&
\left\{ \left( 1 + a_{33} s^2 \right) \left(a_{12} - d + (a_{12} + d) a_{33} s^2 \right) t^2: 
\begin{matrix}
s, t \in F, t \neq 0 
\end{matrix}
\right\} \\
 &=&
\left\{ \left( 1 + a_{33} s^2 \right) \left( a_{22}/c^2 + a_{33} c^2 s^2 \right) t^2: 
\begin{matrix}
s, t \in F, t \neq 0 
\end{matrix}
\right\}, \\
T_{2, F} &:=&
q_F(g_1) \cup q_F(g_2).
\end{eqnarray*}
where 
$g_1, g_2$ are binary quadratic forms defined by $g_1(x, y) := x^2 + b_{33} y^2$, $g_2(x, y) := b_{22} x^2 + b_{33} y^2$.
From $\left( \frac{-a_{33} }{p} \right) = \left( \frac{-b_{33}}{p} \right)$ for any $p \in P$,
we also have $a_{33} / b_{33} \in (K^\times)^2$, hence,
$a_{33} / b_{33} \in (\RationalField^\times)^2$ or $a_{33} / b_{22} b_{33} \in (\RationalField^\times)^2$.

$T_{1, F} \subset T_{2, F}$ implies that for any $s \in \RationalField$,
either of the following holds:
\begin{itemize}
	\item $(1 + a_{33} s^2) ( a_{22} / c^2 + a_{33} c^2 s^2 ) \in q_\RationalField(g_1)$,
	\item $(1 + a_{33} s^2) ( a_{22} / c^2 + a_{33} c^2 s^2 ) \in q_\RationalField(g_2)$.
\end{itemize}
In case of $a_{33} / b_{33} \in (\RationalField^\times)^2$,  
each item implies the following, which is obtained by using the composition of binary quadratic forms:
\begin{itemize}
	\item $a_{22} / c^2 + a_{33} c^2 s^2 \in q_{\RationalField}(g_1)$,
	\item $1 + a_{33} s^2 \in q_{\RationalField}(g_2)$.
\end{itemize}
However, this is impossible if $a_{22} \notin (\RationalField^\times)^2$, 
because there are finite primes $p_1 \ne p_2$ 
such that $q_{\RationalField_{p_1}}(g_1) \not\subset q_{\RationalField_{p_1}}(g_2)$
and
$q_{\RationalField_{p_2}}(g_2) \not\subset q_{\RationalField_{p_2}}(g_1)$,
and it is possible to choose $(s_1, s_2) \in \RationalField^2$ with the following properties (hence $s_1 \ne 0$),
by using the Chinese remainder theorem:
\begin{itemize}
\item 
$s_1^2 + a_{33} s_2^2 \in q_{\RationalField_{p_1}}(g_1) \setminus q_{\RationalField_{p_1}}(g_2)$, 

\item $a_{22} s_1^2 / c^2 + a_{33} c^2 s_2^2 \in q_{\RationalField_{p_2}}(g_2) \setminus q_{\RationalField_{p_2}}(g_1)$.
\end{itemize}
Therefore, $a_{22} \in (\RationalField^\times)^2$ must hold,
which is similarly proved in case of $a_{33} / b_{22} b_{33} \in (\RationalField^\times)^2$.
%In particular, $a_{12} \pm d \in (\RationalField^\times)^2$ is obtained as a result.

Therefore, the following $(A_i \delta_1^2 + B_i, A_i \delta_2^2 + B_i) \in {\rm Sym}^2 (\RationalField^3)^* \otimes_{\RationalField} \RationalField^2$ ($i = 1, 2$) have the identical simultaneous representations over $\RationalField$.
\begin{small}
\begin{eqnarray*}
	(A_1 \delta_1^2 + B_1, A_1 \delta_2^2 + B_1)
	&=&
	\left( \delta_1^2
	\begin{pmatrix}
		1 & -\frac{\delta_1^{-2} - \delta_2^{-2}}{4} & 0 \\
		-\frac{\delta_1^{-2} - \delta_2^{-2}}{4} & \left(\frac{\delta_1^{-2} - \delta_2^{-2}}{4} \right)^2 & 0 \\
		0 & 0 & a_{33}
	\end{pmatrix},
	\delta_2^2
	\begin{pmatrix}
		1 & \frac{\delta_1^{-2} - \delta_2^{-2}}{4} & 0 \\
		\frac{\delta_1^{-2} - \delta_2^{-2}}{4} & \left(\frac{\delta_1^{-2} - \delta_2^{-2}}{4} \right)^2 & 0 \\
		0 & 0 & a_{33}
	\end{pmatrix}
	\right), \\
	(A_2 \delta_1^2 + B_2, A_2 \delta_2^2 + B_2)
	&=&
	\left( 
	\begin{pmatrix}
		1 & -\frac{\delta_1^{2} - \delta_2^{2}}{4} & 0 \\
		-\frac{\delta_1^{2} - \delta_2^{2}}{4} & \left(\frac{\delta_1^{2} - \delta_2^{2}}{4} \right)^2 & 0 \\
		0 & 0 & a_{33}
	\end{pmatrix},
	\begin{pmatrix}
		1 & \frac{\delta_1^{2} - \delta_2^{2}}{4} & 0 \\
		\frac{\delta_1^{2} - \delta_2^{2}}{4} & \left(\frac{\delta_1^{2} - \delta_2^{2}}{4} \right)^2 & 0 \\
		0 & 0 & a_{33}
	\end{pmatrix}
	\right).
\end{eqnarray*}
\end{small}

%Let $m \in \RationalField$ be the largest positive number such that 
%$a_{12} = m \tilde{a}_{12}$ and $a_{22} = m^2 \tilde{a}_{22}$
%hold for some $\tilde{a}_{12}, \tilde{a}_{22} \in \IntegerRing$.
%Then, 
%$\tilde{d} = d / m$ satisfies $\tilde{d}^2 = \tilde{a}_{12}^2 - \tilde{a}_{22}$.
%If a prime $p$ divides $\tilde{d}$, $p$ does not divide $\tilde{a}_{12}$, owing to the choice of $m$. 
%If this $p$ is odd, $(a_{12}-d)/(a_{12}+d) = (\tilde{a}_{12}-\tilde{d})/(\tilde{a}_{12}+\tilde{d}) = 1$ mod $p$.
%Hence, $x^2 + a_{33} y^2$ is isotropic over $\RationalField_p$,
%because otherwise, any elements of $T_{1, \RationalField_p}(q)$ %outside of $p \IntegerRing_p$
%belong to the class $(a_{12} + d) (\RationalField_p^\times)^2$.
%Similarly, it can be proved that 
%$x^2 + a_{33} y^2$ must be isotropic over $\RationalField_2$,
%if $(\tilde{a}_{12}-\tilde{d})/(\tilde{a}_{12}+\tilde{d})$ is odd (hence, 1 mod 8 because of $(\tilde{a}_{12}-\tilde{d})/(\tilde{a}_{12}+\tilde{d}) \in (\RationalField^\times)^2$).
%It is straightforward to confirm that some integers $\eta_1, \eta_2$ with $\gcd(\eta_1, \eta_2) = 1$, and $\epsilon = 1$ or $2$ satisfy
%$\epsilon \eta^2_i = \delta_i^{-2}/2m$ and $\tilde{d} = \epsilon \eta_1 \eta_2$.
%If $(\tilde{a}_{12} - \tilde{d}) / (\tilde{a}_{12} + \tilde{d}) \in \IntegerRing_p^\times$ for some prime $p | \tilde{a}_{33}$,
%$T_{1, \RationalField_p} \supset (\RationalField_p^\times)^2$, hence, 
%$T_{1, \RationalField} = T_{2, \RationalField}$ is impossible.
%Therefore, $\tilde{a}_{33} | (\tilde{a}_{12}^2 - \tilde{d}^2)$, and $\gcd(\tilde{d}, \tilde{a}_{33}) = 1$.
Let $D$ be the square free integer in $a_{33} (\RationalField^\times)^2$.
By using the action of $GL_3(\RationalField)$, $(A_i \delta_1^2 + B_i, A_i \delta_2^2 + B_i)$ ($i = 1, 2$) are transformed into the following $(\tilde{A}_i, \tilde{B}_i)$, respectively:
\begin{small}
\begin{eqnarray*}
	(\tilde{A}_1, \tilde{B}_1) &:=& 
	\left( 
	\begin{pmatrix}
		1 & 0 & 0 \\
		0 & 0 & 0 \\
		0 & 0 & D
	\end{pmatrix}, (\delta_2 / \delta_1)^2
	\begin{pmatrix}
		0 & 0 & 0 \\
		0 & 1 & 0 \\
		0 & 0 & D
	\end{pmatrix}
	\right), \\
	(\tilde{A}_2, \tilde{B}_2) &:=&
	\left( 
	\begin{pmatrix}
		1 & 0 & 0 \\
		0 & 0 & 0 \\
		0 & 0 & D
	\end{pmatrix},
	\begin{pmatrix}
		0 & 0 & 0 \\
		0 & 1 & 0 \\
		0 & 0 & D
	\end{pmatrix}
	\right).
\end{eqnarray*}
\end{small}

From $q_\RationalField(\tilde{A}_1, \tilde{B}_1) = q_\RationalField(\tilde{A}_2, \tilde{B}_2)$,
for any $(q, q) \in q_\RationalField(\tilde{A}_i, \tilde{B}_i)$ and $k \in \IntegerRing$, 
$(q, (\delta_1/\delta_2)^{2k} q) \in q_\RationalField(\tilde{A}_i, \tilde{B}_i)$ ($i = 1, 2$) holds. 
Equivalently, for any $0 \ne (u_1, u_3) \in \RationalField^2$ and $k \in \IntegerRing$,
there exists $u_2 \in \RationalField$ such that $u_1^2 + D u_3^2 = (\delta_2/\delta_1)^{2k} (u_2^2 + D u_3^2)$.
Hence, some $\zeta \in \RationalField(\sqrt{-D})$ with the norm 1
satisfies $\zeta (u_1 + \sqrt{-D} u_3) = (\delta_2/\delta_1)^{k}(u_2 + \sqrt{-D} u_3)$.
As a result of Hilbert's theorem 90, some $0 \ne (s, t) \in \IntegerRing^2$
satisfies $\zeta = (s + \sqrt{-D} t)/(s - \sqrt{-D} t)$. Therefore, 
\begin{eqnarray}\label{eq: formula for proposition}
	\{ u_3 (s^2 - D t^2) + 2 u_1 s t \}/(s^2 + D t^2) = (\delta_2/\delta_1)^{k} u_3.
\end{eqnarray}

Now $(\delta_2/\delta_1)^2 \ne 1$, because otherwise $a_{22} = 0$.
If $\eta_1, \eta_2$ are coprime integers with $\delta_2/\delta_1 = \eta_2/\eta_1$,
some prime $p$ divides either of $\eta_1, \eta_2$.
Hence, if $k$ is chosen so that $(\delta_2/\delta_1)^{k}$ is sufficiently close to 0 in $\RationalField_p$, 
Eq.(\ref{eq: formula for proposition})
implies that
the quadratic form $f_{u_1, u_3}(x_1, x_2) := u_3 (x_1^2 - D x_2^2) + 2 u_1 x_1 x_2$ is isotropic over $\RationalField_p$.
A contradiction is obtained by choosing $0 \ne (u_1, u_3) \in \RationalField^2$ with $u_1^2 + D u_3^2 \notin (\RationalField_p^\times)^2$. 
\end{proof}
We can now proceed to the proof of Theorem \ref{thm:main result over RationalField}.
As in the previous section, we denote the quartic and cubic $\RationalField$-algebras assigned to $(A_i, B_i)$ by 
$(Q_\RationalField(A_i, B_i), \langle 1, \xi_{i,1}, \xi_{i,2}, \xi_{i,3} \rangle)$ and 
$(R_\RationalField(A_i, B_i), \langle 1, \omega_{i,1}, \omega_{i,2} \rangle)$.

\begin{proof}[Proof of Theorem \ref{thm:main result over RationalField}]
We fix $(A_i, B_i) \in V_\RationalField$ ($i =1, 2$) as stated in Theorem \ref{thm:main result over RationalField}.
From Proposition \ref{prop:same det(Ax+By)}, 
we have $r_1^{-1} \det (A_1 x + B_1 y) = r_2^{-1} \det (A_2 x + B_2 y)$ for some coprime integers $r_1, r_2$.
In what follows, we denote $4 \det (A_i x - B_i y) / r_i$ by $\tilde{f}_{det}(x, y) = a x^3 + b x^2 y + c x y^2 + d y^3$ ($a, b, c, d \in \RationalField$).

% First we shall show that the last statement is obtained from the previous ones of this proposition.
% From Lemma \ref{lem:same content}, the contents of $Q(A_1, B_1)$ and $Q(A_2, B_2)$ must be the same number.
% Furthermore, we may assume they equal 1 by using the natural action of $GL(2, \IntegerRing)$ on $(Sym^2 \IntegerRing^3 \otimes \IntegerRing^2)^*$.
% Considering $4 | {\rm tr}(\alpha)$ holds for any $\alpha \in 1 + 4 Q(A_i, B_i)$, $t_i$ divide the content of $1 + 4 Q_i$ \IE 4 for both $i = 1, 2$.
% Hence $(r_1, r_2)$ must be one of the above three.

From Lemma \ref{lem:generator of Q otimes RationalField}, 
some $h_1 := (h_{1,1}, h_{1,2}, h_{1,3}) \in \RationalField^3$ satisfies $\tilde{f}_{det}( B_1(h_1), A_1(h_1) )\neq~0$.
From the assumption (c'), 
there exists $h_2 := (h_{2,1}, h_{2,2}, h_{2,3}) \in \RationalField^3$ such that $(A_1, B_1)(h_1) = (A_2, B_2)(h_2)$, hence,
$\tilde{f}_{det}( B_2(h_2), A_2(h_2) ) \neq 0$.
Using 
these $h_1$, $h_2$ and arbitrarily chosen $h_{1,0}, h_{2,0} \in \RationalField$,
we define $\alpha_i := h_{i,0} + \sum_{j=1}^3 h_{i,j} \xi_{i,j} \in Q_\RationalField(A_i, B_i)$.
If we put $q_A := A_1(h_1) = A_2(h_2)$ and $q_B := B_1(h_1) = B_2(h_2)$, then $\alpha_i$ 
has the characteristic polynomial ${\rm ch}_{\alpha_i}(x) := x^4 + a_{i,3} x^3 + a_{i,2} x^2 + a_{i,1} x + a_{i,0}$ as follows:
\begin{eqnarray*}
{\rm ch}_{\alpha_i}^{res} \left(x + \frac{ a_{i,2} + r_i (c q_A - b q_B) }{3} \right) 
&=& x^3 + r_i (c q_A - b q_B) x^2 + r_i^2 \{ b d q_A^2 + (3 a d - b c) q_A q_B  + a c q_B^2 \} x \nonumber \\
& &  + r_i^3 \{ a d^2 q_A^3 - (b^2 d - 2 a c d) q_A^2 q_B + (a c^2 - 2 a b d) q_A q_B^2 - a^2 d q_B^3 \}.
\end{eqnarray*}
Hence, 
\begin{eqnarray}\label{eq:relation between ch(beta_i)^{res}(x)}
	r_1^{-3} {\rm ch}_{\alpha_1}^{res} \left( r_1 x + \frac{ a_{1,2} }{3} \right) = r_2^{-3} {\rm ch}_{\alpha_2}^{res} \left( r_2 x + \frac{ a_{2,2} }{3} \right).
\end{eqnarray}
Let $W_1$, $W_2$ and $V_i$ be the rational matrix determined by
\begin{eqnarray*}
	W_i
	\begin{pmatrix}
		\overline{\xi_1} \\ \overline{\xi_2} \\ \overline{\xi_3} 
	\end{pmatrix} 
 &=&
	\begin{pmatrix}
		1  & 0 & 0 \\
		a_{i,3}  & 1 & 0 \\
		a_{i,2}  & a_{i,3} & 1
	\end{pmatrix}
	\begin{pmatrix}
		\overline{\alpha_i} \\ \overline{\alpha_i^2} \\ \overline{\alpha_i^3} 
	\end{pmatrix}, \\
V_i
 &:=& \frac{1}{ r_i \tilde{f}_{det}(q_B, -q_A)}
\begin{pmatrix}
 1 & 0  \\
r_i  (c A(h) - b B(h)) /3 & 1
\end{pmatrix}
\begin{pmatrix}
 q_B & -q_A \\
-r_i (c q_A q_B + d q_A^2) & - r_i (a q_B^2 + b q_A q_B) 
\end{pmatrix}.
\end{eqnarray*}

Since 
$\begin{pmatrix}
	r_1 & 0 \\
	0 & 1
\end{pmatrix} V_1
=
\begin{pmatrix}
	r_2 & 0 \\
	0 & 1
\end{pmatrix} V_2
$,
the following $(\tilde{A}_i, \tilde{B}_i)$ ($i = 1, 2$)
have the identical sets of simultaneous representations over $\RationalField$ owing to the (c'):
\begin{eqnarray}
(\tilde{A}_i, \tilde{B}_i) 
:= 
\left( W_i, 
\begin{pmatrix}
	r_i & 0 \\
	0 & 1
\end{pmatrix} V_i \right)
(A_i, B_i).
\end{eqnarray}
From $4 \det (\tilde{A}_i x - \tilde{B}_i) = {\rm ch}_{\alpha_i}^{res} \left( r_i x + a_{i,2} / 3 \right)$,
$\det (\tilde{A}_1 x - \tilde{B}_1) = (r_1/r_2)^{3} \det (\tilde{A}_2 x - \tilde{B}_2)$ is obtained.
Even if $\alpha_i$ is replaced by $\alpha_i - \trace{\alpha_i}/4$ (${\rm Tr}$ is the trace function), all the above hold.
% The replacement corresponds to the transformation $(\tilde{A}_i, \tilde{B}_i) \mapsto (V, 1) \cdot (\tilde{A}_i, \tilde{B}_i)$ by 
% \begin{eqnarray}
% 	V :=
% 	\begin{pmatrix}
% 		        1 & 0 & 0 \\
% 		       -a_{i,3}/2 & 1 & 0 \\
% 	            a_{i,3}^2/8 & -a_{i,3}/4 & 1
% 	\end{pmatrix}.
% \end{eqnarray}
% Now that $M_1$, $M_2 \in GL_3(\IntegerRing_p)$ for all but finitely many primes $p$, 
The proof of the theorem is completed by the following Proposition \ref{prop:case of monogenic}.
\end{proof}

\begin{prop}\label{prop:case of monogenic}
We assume that $r_1, r_2$ are coprime integers
and the following $(A_i, B_i) \in V_\RationalField$ ($i = 1, 2$)
satisfy 
the assumptions (b') and (c') of Theorem \ref{thm:main result over RationalField} and $\det (A_1 x - B_1) = (r_1/r_2)^{3} \det (A_2 x - B_2)$:
\begin{small}
\begin{eqnarray}\label{eq:definition of (A_i, B_i)}
(A_i, B_i) := 
	\left(
	\begin{pmatrix}
		       0 &       0     & r_i/2 \\
               0 & - r_i & 0 \\
          r_i/2 & 0 & 0
	\end{pmatrix},
	-
	\begin{pmatrix}
		       1 &       0     & a_{i,2} / 6 \\
               0 & 2 a_{i,2} / 3 & a_{i,1}/2 \\
	        a_{i,2} / 6 & a_{i,1}/2  & a_{i,0}
	\end{pmatrix}
	\right).
\end{eqnarray}
\end{small}
Then, either of the following holds:
\begin{enumerate}[1.]
\item 
there exist coprime integers $u_1, u_2$ such that $r_i = u_i^2$ ($i = 1, 2$),
and 
$(A_1, B_1) = (w, I) \cdot (A_2, B_2)$ 
for the following $w \in GL_3(\RationalField)$:
\begin{eqnarray}
	w = 
	\begin{pmatrix}
		1 & 0 & 0 \\
		0 & u_1/u_2 & 0 \\
		0 & 0 & (u_1/u_2)^2
	\end{pmatrix}.
\end{eqnarray}

\item 
There exist $s_1, s_2 \in \RationalField$ such that $(r_2 / r_1) (s_1^4 + a_{1,2} s_1^2 + a_{1,1} s_1 + a_{1,0}) = (r_1 / r_2) (s_2^4 + a_{2,2} s_2^2 + a_{2,1} s_2 + a_{2,0}) = c^2$ for some $c \in \RationalField^\times$,
and $(w, (r_2/r_1) I) \cdot (A_1, B_1) = (A_2, B_2)$
for the following $w \in GL_3(\RationalField)$:
\begin{eqnarray}\label{eq:case of t neq 0}
	w = 
 	\begin{pmatrix}
 		       1 & 0 & 0 \\
 		       2 s_2 & 1 & 0 \\
 	             s_2^2 & s_2 & 1
 	\end{pmatrix}^{-1}
	\begin{pmatrix}
		0 & 0 & 1/c \\
		0 & 1 & 0 \\
		c & 0 & 0
	\end{pmatrix}
 	\begin{pmatrix}
 		       1 & 0 & 0 \\
 		       2 s_1 & 1 & 0 \\
 	             s_1^2 & s_1 & 1 
 	\end{pmatrix}.
\end{eqnarray}

\end{enumerate}
The above 1.\ always holds, if ${\rm Disc}(A_i, B_i) = 0$ for (either of) $i = 1, 2$. 

\end{prop}

For the proof, the following lemma is used.
\begin{lem}\label{lem:isomorphic k-algebras} 
Let $k$ be a global field with ${\rm char}\ k \neq 2$.
Assume that $f_1(x)$, $f_2(x) \in k[x]$ are monomial quartic polynomials 
with no roots in $k$, no multiple roots in $\bar{k}$,
and $k[x] / (f_1^{res}(x))$ and $k[x] / (f_2^{res}(x))$ are isomorphic as $k$-algebras.
Furthermore, assume that $f_1(x)$, $f_2(x)$ have a root in $k_{\mathfrak p}$ with respect to the same ${\mathfrak p} \in P$,
where $P$ is the set of all the primes of $k$ that completely splits in $k[x] / (f_i^{res}(x))$. 
In this case, %for any ${\mathfrak p} \in P$,
%${\mathfrak p}$ completely splits in $k[x]/(f_1(x))$ if and only if it does in $k[x]/(f_2(x))$. In addition, 
$k[x] / (f_1(x))$ and $k[x] / (f_2(x))$ 
are isomorphic as $k$-algebras.
\end{lem}

\begin{proof}
It may be assumed that the coefficient of $x^3$ in $f_i(x)$ equals $0$. 
Fix a prime ${\mathfrak p} \in P$ so that both of $f_1(x)$ and $f_2(x)$ have a root in $k_{\mathfrak p}$.
Let $\alpha_{i, {\mathfrak p}} \in k_{\mathfrak p}$ ($i = 1, 2$) be the root.
If the roots of $g_{i, {\mathfrak p}}(x) := f_i(x) / (x - \alpha_{i, {\mathfrak p}}) \in k_{\mathfrak p}[x]$ are denoted by 
$\beta_{i,1}, \beta_{i,2}, \beta_{i,3} \in k_{\mathfrak p}$, then
\begin{eqnarray*}
% \beta_1, \beta_2, \beta_3: roots of g_{i,p}(x)
 \alpha_{i, {\mathfrak p}} \beta_{i,1} + \beta_{i,2} \beta_{i,3}
 &=& -\alpha_{i, {\mathfrak p}} (\alpha_{i, {\mathfrak p}} + \beta_{i,2} + \beta_{i,3}) + \beta_{i,2} \beta_{i,3} \\
 &=& -2 \alpha_{i, {\mathfrak p}}^2 + (\alpha_{i, {\mathfrak p}} - \beta_{i, 2}) (\alpha_{i, {\mathfrak p}} - \beta_{i, 3}) \\ 
 &=& -2 \alpha_{i, {\mathfrak p}}^2 - f_i^\prime(\alpha_{i, {\mathfrak p}}) /(\beta_{i, 1} - \alpha_{i, {\mathfrak p}}),
\end{eqnarray*}
where  $f_i^\prime(x)$ is the first derivative of $f_i(x)$ with respect to $x$.
Thus, there exists $0 \neq C \in k_{\mathfrak p}$ such that
\begin{eqnarray}
	g_{i, {\mathfrak p}}(x)
	= C (x - \alpha_{i, {\mathfrak p}})^3
%\frac{ (x - \alpha_{i, p})^3 }{ f_i^{res}(- 2 \alpha_{i, p}^2) }
		f_i^{res} \left( - \frac{ f_i^\prime(\alpha_{i, {\mathfrak p}}) }{ x - \alpha_{i, {\mathfrak p}} } - 2 \alpha_{i, {\mathfrak p}}^2 \right).
\end{eqnarray}
From the assumption about $f_1^{res}(x)$ and $f_2^{res}(x)$, 
${\mathfrak p}$ completely splits in $k[x] / (f_1(x))$ if and only if it does in $k[x] / (f_2(x))$.

Let $f_i(x) = \prod_{j=1}^m g_{ij}(x)$ ($g_{ij} \in k[x]$) be a factorization in $k$.
By assumption, $k[x] / (f_i(x))$ is a direct sum of $k[x]/(g_{ij}(x))$ with degree $2$ or 4 over $k$.
For each $g_{ij}(x)$, we 
fix an embedding $\iota_{ij} : k[x]/(g_{ij}(x)) \hookrightarrow \bar{k}$
and let $K_i$ be the composite field of $\iota_{ij} (k[x]/(g_{ij}(x)))$ ($j = 1, \cdots, m$).
$K_1$, $K_2$ are quadratic or quartic fields over $k$,
and any prime ${\mathfrak p}$ of $k$ completely splits in $K_1$ if and only if it does over $K_2$.
Hence, $K_1$, $K_2$ have the identical Galois closure over $k$ (Theorem 8.8, \cite{Kato2011}).
%Note that $y^2 \neq 4 a_{i,0}$ holds because otherwise, $f_i(x) = (x^2 + a_3 x/2 + y/2)^2$ follows
%from $(4 y + a_{i, 3}^2 - 4 a_{i, 2})(y^2 - 4 a_{i,0}) = (a_{i, 3} y - 2 a_{i, 1})^2$.

If both of $K_i$ are Galois over $k$, then $K_1 = K_2$.
If either of $K_1$, $K_2$ is not Galois over $k$, both must be a quartic field not Galois over $k$.
Even in this case, $K_1$ and $K_2$ are isomorphic over $k$. This can be seen as follows;
first suppose that $f_i^{res}(x) = 0$ has a root $u_i$ in $k$.
In this case, 
$f_i(x) := x^4 + a_{i, 2} x^2 + a_{i, 1} x + a_{i, 0}$ is decomposed 
as follows:
\begin{eqnarray}
	f_i(x)
	 = 
	(x^2 + u_i/2)^2 - (u_i^2 - 4 a_{i,0}) \left( \frac{ a_{i, 1} }{ u_i^2 - 4 a_{i,0} } x - 1/2 \right)^2.
\end{eqnarray}
Even if $u_i^2 = 4 a_{i,0}$, $a_{i,1}=0$ follows from $f_i^{res}(x) = (x - a_{i, 2})(x^2 - 4 a_{i, 0}) - a_{i, 1}^2$. 
Therefore $K_1$, $K_2$ are quadratic extensions of a quadratic field over $k$.
Since they have the same Galois closure, the quadratic field is common, 
and $K_1, K_2$ are conjugate over $k$. 

Next, suppose that $f_i^{res}(x)$ does not have a root in $k$.
In this case, the Galois closure $F$ of $K_1$, $K_2$ contains a cubic field isomorphic to $k[x]/(f_i^{res}(x))$. 
Since ${\rm Gal}(F/k)$ is isomorphic to a subgroup of $S_4$, this happens only when ${\rm Gal}(F / k) \cong S_4$ or $A_4$.
Since all the subgroups of $S_4$ (\textit{resp.} $A_4$) of index $4$ 
are conjugate to $S_3 = \langle (1, 2, 3), (1, 2) \rangle$ (\textit{resp.} $A_3  = \langle (1, 2, 3) \rangle$),
$K_1, K_2$ are conjugate over $k$. 

Consequently, if $k[x] / (g_{ij}(x)) = K_i$ for all $1 \leq j \leq m$, $k[x] / (f_{1}(x))$, $k[x] / (f_{2}(x))$
are isomorphic as $k$-algebras.
It remains to verify the case of 
$K_1 = F_1 \cdot F_2$ and $K_2 = F_1 \oplus F_2$, 
where $F_1, F_2$ are distinct quadratic fields.
They correspond to the $G_k$-orbits 
containing
\begin{itemize}
\item $k(\sqrt{d_1}, \sqrt{d_2})$,
\begin{eqnarray}
(A, B) :=
\left(
\begin{pmatrix}
	-d_1 & 0 & 0 \\
	0 & 1 & 0 \\
	0 & 0 & 0 \\
\end{pmatrix}
\begin{pmatrix}
	d_2 & 0 & 0 \\
	0 & 0 & 0 \\
	0 & 0 & -1 \\
\end{pmatrix} \right).
\end{eqnarray}

\item $k(\sqrt{d_1}) \oplus k(\sqrt{d_2})$,
\begin{eqnarray}
(A, B) :=
\left(
\begin{pmatrix}
	-d_1 & 0 & 0 \\
	0 & -d_2 & 0 \\
	0 & 0 & 1 \\
\end{pmatrix}
\begin{pmatrix}
	0 & -1/2 & 0 \\
	-1/2 & 0 & 0 \\
	0 & 0 & 0 \\
\end{pmatrix} \right).
\end{eqnarray}

\end{itemize}

However, $\det(A x - B y)$ of the former has only rational roots,
whereas that of the latter has irrational roots.
Therefore, $f_1^{res}(x) = f_2^{res}(x)$ cannot happen. Hence, this case can be eliminated.

\end{proof}

\begin{proof}[Proof of Proposition \ref{prop:case of monogenic}]
We first put $f_i(x) := x^4 + a_{i, 2} x^2 + a_{i, 1} x + a_{i, 0}$.
It follows from $\det(A_1 x - B_1y) = (r_1/r_2)^3 \det(A_2 x - B_2 y)$ that
\begin{eqnarray*}
(1/r_1)^2 (4 a_{1,0} + a_{1,2}^2/3) &=& (1/r_2)^2 (4 a_{2,0} + a_{2,2}^2/3), \label{eq: equation 1 between a_{i,0} + a_{i,2}} \\
(1/r_1)^3 (a_{1,1}^2 - 8 a_{1,2} a_{1,0}/3 + 2 a_{1,2}^3 / 27) &=& (1/r_2)^3 (a_{2,1}^2 - 8 a_{2,2} a_{2,0}/3 + 2 a_{2,2}^3 / 27).
\end{eqnarray*}
Hence, $r_1^{-1} a_{1,2} = r_2^{-1} a_{2,2}$ implies that
$r_1^{-2} a_{1,0} = r_2^{-2} a_{2,0}$ and
$r_1^{-3} a_{1,1}^2 = r_2^{-3} a_{2,1}^2$. Thus, in this case,
case 1.\ occurs.  Therefore, in what follows we assume that $r_1^{-1} a_{1,2} \neq r_2^{-1} a_{2,2}$.

In case of ${\rm Disc}(f_1^{res}) = {\rm Disc}(f_2^{res}) = 0$,
the multiple root of $f_i^{res}(x) = \det (A_i x - B_i)$ for $i = 1, 2$, 
equals $x = 2 a_{1, 2} / 3 r_1 = 2 a_{2, 2} / 3 r_2$, which is seen by checking when $A_i x - B_i$ in Eq.(\ref{eq:definition of (A_i, B_i)}) is rank $1$.  In particular, $r_1^{-1} a_{1,2} = r_2^{-1} a_{2,2}$ follows in this case.

Since ${\rm Disc}(f_1) = 0$ implies ${\rm Disc}(f_1^{res}) = 0$,
we may now assume that both $f_i$ and $f_i^{res}$ have no multiple roots.
By Corollary \ref{cor: (A, B) is isotropic over k}, 
$\RationalField[x] / (f_i(x))$ is a direct sum of number fields of degree greater than $1$ over $\RationalField$.
Let $p$ be a finite prime that completely splits in $\RationalField[x] / (f_i^{res}(x))$.
$(A_i, B_i)$ is isotropic over $\RationalField_p$ 
if and only if $\RationalField_p$ contains a root of $f_i(x) := x^4 + a_{i, 2} x^2 + a_{i, 1} x + a_{i, 0} = 0$.
Hence, 
$f_1(x) = 0$ has a root in $\RationalField_p$ if and only if $f_2(x) = 0$ does.
By Lemma \ref{lem:isomorphic k-algebras},
$\RationalField[x] / (f_1(x))$ and $\RationalField[x] / (f_2(x))$ are isomorphic as $\RationalField$-algebras.
Thus, there exists $x_2 \in \RationalField[x] / (f_1(x))$
such that $x \mapsto x_2$ provides an isomorphism $\RationalField[x] / (f_2(x)) \rightarrow \RationalField[x] / (f_1(x))$.

Now $\QA{\RationalField}{A_1}{B_1}$ is a $\RationalField$-algebra isomorphic to $\RationalField[x] / (f_1(x))$
with the basis $\langle 1, r_1 x, r_1 (x^2 + a_{1,2}), r_1 (x^3 + a_{1,2} x + a_{1,1}) \rangle$.
We assume that $x_2$ is represented as $h_0 + r_1 \{ h_1 x + h_2 x^2 + h_3 (x^3 + a_{1,2} x) \}$ in $\QA{\RationalField}{A_1}{B_1}$,
 using some $h_0 \in \RationalField$ and $h := (h_{1}, h_{2}, h_{3}) \in \RationalField^3$.
If we put $b_0 = -(1/r_i)^3 (a_{i,1}^2 - 8 a_{i,2} a_{i,0}/3 + 2 a_{i,2}^3 / 27)$ and $b_1 = - (1/r_i)^2 (4 a_{i,0} + a_{i,2}^2/3)$, 
then 
$$\tilde{f}_{det}(x, y) := 4 r_i^{-3} \det (A_i x - B_i y) = x^3 + b_1 x y^2 + b_0 y.$$
In this case, by using the formulas given in the proof of Lemma \ref{prop: monogenic ring},
$(W, V) \in G_\RationalField$ satisfying $(W, V) \cdot (A_1, B_1) = (A_2, B_2)$ is obtained as follows
(herein, $\bar{x}$, $\bar{x_2}$ are the classes of $x$, $x_2$ in $\QA{\RationalField}{A_1}{B_1} / \RationalField$):
\begin{eqnarray*}
	W
	\begin{pmatrix}
		 \overline{r_1 x} \\ \overline{r_1 x^2} \\ \overline{r_1(x^3 + a_{1,2} x)}
	\end{pmatrix}
 &=&
	\begin{pmatrix}
		\overline{x_2} \\ \overline{x_2^2} \\ \overline{x_2^3 + a_{2,2} x_2}
	\end{pmatrix},
\end{eqnarray*}
\begin{eqnarray*}
V
 &:=& (\det W)^{-1}
\begin{pmatrix}
 r_2 & 0  \\
\frac{ r_1^3 b_1 A_1(h) }{3} & 1
\end{pmatrix}
\begin{pmatrix}
 B_1(h) & -A_1(h)  \\
 -r_1^3 ( b_1 A_1(h) B_1(h) + b_0 A_1(h)^2 ) & -r_1^3 B_1(h)^2
\end{pmatrix} \\
 &=& (\det W)^{-1}
\begin{pmatrix}
 r_2 B_1(h) & - r_2 A_1(h)  \\
 -r_1^3 ( \frac{2}{3} b_1 A_1(h) B_1(h) + b_0 A_1(h)^2 ) & - r_1^3 ( \frac{1}{3} b_1 A_1(h)^2 + B_1(h)^2)
\end{pmatrix}.
\end{eqnarray*}
These matrices have the determinants $\det W = r_1^3 ( B_1(h)^3 + b_1 A_1(h)^2 B_1(h) + b_0 A_1(h)^3 )$
and $\det V = -r_2 \det (W)^{-1}$.
%Therefore, 
%$\det W^2 \det ((x, - y) V \tr{(A_1, B_1)}) = \det (A_2 x - B_2 y) =  (r_2 / r_1)^3 \det (A_1 x - B_1 y)$ is obtained.

From $q_\RationalField(A_1, B_1) = q_\RationalField(A_2, B_2)$,
$q_\RationalField(A_1, B_1) = q_\RationalField((I, V) \cdot (A_1, B_1))$ is obtained.
In addition, $V$ satisfies:
$$
(\det W)^2 \tilde{f}_{det}((x, y)\tilde{V}) = (r_2/r_1)^3 \tilde{f}_{det}(x, y), \quad \tilde{V} :=
\begin{pmatrix}
	1 & 0 \\
	0 & -1 \\
\end{pmatrix}
V
\begin{pmatrix}
	1 & 0 \\
	0 & -1 \\
\end{pmatrix}.
$$
By Lemma~\ref{lem:automorphism of cubic polynomial}, 
$\tilde{V}^n = u^n I$ holds for $u :=  (r_2/r_1)^3 / (\det V \det W^2) = r_2 \det V/r_1^3$ and either of $n = 1, 2, 3$.
If $n \ne 1$, by Lemma~\ref{lem:automorphism of (A, B) with V},
there exists $W_2 \in GL_3(\RationalField)$ such that 
$(W_2, -u^{-1} V) \cdot (A_1, B_1)~=~(A_1, B_1)$.
Thus, we may assume that $V$ is a scalar multiple of $I$,
by replacing $(W, V)$ with $(W, V) (W_2, -u^{-1} V)$, 
and the above $x_2$ by another element of $\RationalField[x]/(f_1(x))$ with the characteristic polynomial $f_2(x)$. 
Thus,
$A_1(h) = 0$, $B_1(h) = - r_1^{-3} r_2$, $\det W = - r_1^{-6} r_2^3$, 
and $V = r_1^3 r_2^{-1} I$. 
% (W, r_1^4 r_2^{-2} I) \cdot (r_1^{-1} A_1, r_1^{-1} B_1) = (r_2^{-1} A_2, r_2^{-1} B_2)
Therefore, 
$(r_1 W, (r_1/ r_2) I)$ also maps $(A_1, B_1)$ to $(A_2, B_2)$.

Since $A_1(h) = 0$, 
there are $0 \neq C \in \RationalField$ and $0 \neq (s, t) \in \RationalField$ such that $h = C (s^2, s t, t^2)$.
%In the former case, 
%$h_1^2 = r_1^{-3} r_2$ is obtained from $\tilde{B}_1(h) = - r_1^{-3} r_2$,
%hence there exist coprime integers $u_1, u_2$ such that $r_i = u_i^2$ ($i = 1, 2$).
%If we put $\tilde{W} := \frac{u_1^3}{ u_2 } W$,
%$(\tilde{W}, I) \cdot (\tilde{A}_1, \tilde{B}_1) = (\tilde{A}_2, \tilde{B}_2)$ 
%and $(1, 0, 0) \tilde{W} = \frac{u_1^3}{ u_2 } (h_1, 0, 0) = (\pm 1, 0, 0)$ hold.
%The above $y$ equals $r_1 h_1 x = \pm u_2 / u_1 x$.
%Hence all the statements are obtained from $a_{1, k} = (\pm u_1 / u_2)^{4 - k} a_{2, k}$ ($k = 0, 1, 2$).
If we put $f_i(X, Y) := Y^4 f_i(X/Y)$ ($i = 1, 2$), then
$C^2 f_1(s, t) = r_1^{-3} r_2$ follows from $B_1(h) = - r_1^{-3} r_2$.
Since $x$, $x_2$ satisfies $\trace{x^2} = - 2 a_{1,2}$, $\trace{x^3} = - 3 a_{1,1}$ and $\trace{x} = \trace{x_2} = 0$ in $\RationalField[x] / (f_1(x))$, we obtain 
$x_2 = C r_1 ( s^2 x + s t (x^2 + a_{1, 2}/2) + t^2 (x^3 + a_{1,2} x + 3 a_{1,1}/4))$.
Hence, 
\begin{eqnarray*}
	& & \hspace{-10mm}
 	(t x - s)\{ t x_2 + C r_1 (s^3 + a_{1,2} s t^2/2 + a_{1,1} t^3 / 4) \} \\
	&=& C r_1 (t x - s) \left\{ s^3 + s^2 t x + s t^2 (x^2 + a_{1, 2}) + t^3 (x^3 + a_{1,2} x + a_{1,1}) \right\} \\
	&=& -C r_1 \left\{ s^4 - t^4 x^4 + a_{1, 2} t^2 (s^2 - t^2 x^2) + a_{1,1} t^3 (s - t x) \right\} \\
	&=& -C r_1 f_1(s, t). 
\end{eqnarray*}
If we put $\tilde{C} := -C r_1 f_1(s, t)$, 
$\tilde{C}^2 = (r_2 / r_1) f_1(s, t)$ follows from $C^2 f_1(s, t) = r_1^{-3} r_2$.
Hence, if $t = 0$, then $r_1 / r_2 \in (\RationalField^\times)^2$, and $x_2 = c x$ holds for $c := C r_1 s^2$,
which satisfies $c^2 = C^2 r_1^2 s^4 = r_2/r_1$.
As a result, $a_{2, 2} = c^2 a_{1, 2}$, $a_{2, 1} = c^3 a_{1, 1}$ and $a_{2, 0} = c^4 a_{1, 0}$ are obtained,
which implies the case 1.

If $t \neq 0$, 
the characteristic polynomials of $\tilde{x} := t x - s$, $\tilde{x}_2 := \tilde{C} (t x - s)^{-1}$
%nd $\tilde{x}_2 := t x_2 + C r_1 (s^3 - a_{1,2} s t^2 - 2 a_{1,1} t^3)$
are as follows:
\begin{eqnarray*}
	{\rm ch}_{\tilde{x}}(X)
 	&=& f_1(X + s, t)
 	= X^4 + 4 s X^3 + (6 s^2 + a_{1, 2} t^2) X^2 + (4 s^3 + 2 a_{1, 2} s t^2 + a_{1, 1} t^3) X + f_1(s, t), \\
 	{\rm ch}_{\tilde{x}_2}(X)
 	&=& f_1(s, t)^{-1} X^4 {\rm ch}_{\tilde{x}}(\tilde{C} / X) \\
 	&=& f_1(s, t)^{-1} \{ \tilde{C}^4  + 4 \tilde{C}^3 s X + \tilde{C}^2 (6 s^2 + a_{1, 2} t^2) X^2 + \tilde{C} (4 s^3 + 2 a_{1, 2} s t^2 + a_{1, 1} t^3) X^3 + f_1(s, t) X^4 \}. \nonumber
\end{eqnarray*}

Therefore, in this case, %if $r_1/r_2  \in (\RationalField^\times)^2$ is proved,
Eq.(\ref{eq:case of t neq 0}) is obtained 
by putting $s_1 = s$, $s_2 = \tilde{C} (s^3 + a_{1,2} s t^2 / 2 + a_{1,1} t^3 / 4) / f_1(s, t)$ and $c = \tilde{C}$.
Thus, the proposition is proved.

\end{proof}

%From the proof of Theorem \ref{}, the following corollary follows:
%\begin{cor}\label{cor:same A(h) and B(h)}
%Suppose that $(A, B) \in V_\RationalField$ satisfies the  (\ref{item: assumption (a)})--(\ref{item: assumption (c)}) of Theorem \ref{thm:main result over RationalField},
%and that there exists ${\mathbf x}_1, {\mathbf x}_2 \in \RationalField^3$ ($i = 1, 2$)
%such that 
%$A({\mathbf x}_1) = A({\mathbf x}_2) =: q_A$, $B({\mathbf x}_1) = B({\mathbf x}_2) =: q_B$ and
%$\det (q_B A - q_A B) \neq 0$.
%In this case, 
%$(w, 1) \cdot (A, B) = (A, B)$ and ${\mathbf x}_1 w = {\mathbf x}_2$ hold for some $w \in GL_3(\RationalField)$.
%
%\end{cor}

\section{Case of ${\rm Disc}(A_i, B_i) = 0$ (proofs of Propositions \ref{thm:proposition 1}, \ref{thm:theorem 1})}
\label{Case of Disc(A_i, B_i) = 0 (proofs of Propositions 1, 2)}

From the known result in the binary case, it is immediately obtained that 
$q_\IntegerRing(x_1^2 - x_1 x_2 + x_2^2, x_3^2) = q_\IntegerRing(x_1^2 + 3x_2^2, x_3^2)$.
In the case of (ii), if we put $(A_1, B_1) = (x_1^2 - x_1 x_2 + x_2^2, (x_1 + x_2 + 3 x_3)^2)$
and $(A_2, B_2) = (x_1^2 + 3x_2^2, (x_1 + 3 x_3)^2)$, we have
\begin{eqnarray*}
(A_2, B_2)(x_1, x_2, x_3) &=& (A_1, B_1)(x_1 + x_2, 2 x_2, -x_2 + x_3) \\
&=& (A_1, B_1)(2 x_1, x_1 + x_2, -x_1 + x_3) \\
&=& (A_1, B_1)(x_1 + x_2, x_1-x_2,  -x_1 - x_3).
\end{eqnarray*}

Hence, $q_\IntegerRing(A_2, B_2) \subset q_\IntegerRing(A_1, B_1)$, 
and the converse is also true, since any $(y_1, y_2) \in \IntegerRing^2$ 
can be represented in either of the ways $(x_1 + x_2, 2 x_2)$, $(2 x_1, x_1 + x_2)$ or $(x_1 + x_2, x_1-x_2)$
for some $x_1, x_2 \in \IntegerRing$.

Proposition \ref{thm:theorem 1} can be also proved in an elementary way.

\begin{proof}[Proof of Proposition \ref{thm:theorem 1}]
As proved in Proposition \ref{prop:same det(Ax+By)},
$q_\IntegerRing(A_1, B_1) = q_\IntegerRing(A_2, B_2)$ 
implies that all the roots of $\det (A_1 x - B_1 y)$
and $\det (A_2 x - B_2 y)$ are common.
In particular, ${\rm Disc}(A_1, B_1) = 0$ leads to ${\rm Disc}(A_2, B_2) = 0$.

By the action of $G_\IntegerRing$, $(A_1, B_1)$ can be transformed into:
\begin{eqnarray*}
(A_1, B_1) = (M, 1) \cdot
\left(
\begin{pmatrix}
	a_{11} & a_{12} & 0 \\
	a_{12} & a_{22} & 0 \\
	0 & 0 & 0
\end{pmatrix},
\begin{pmatrix}
	0 & 0 & 0 \\
	0 & 0 & 0 \\
	0 & 0 & 1
\end{pmatrix}
\right),\
M
:= 
\begin{pmatrix}
  1 & 0 & m_1 \\ 
  0 & 1 & m_2 \\ 
  0 & 0 & m_3 
\end{pmatrix},
\end{eqnarray*}
where $m_1, m_2$ and $0 \neq m_3 \in \IntegerRing$ may be assumed to have the greatest common divisor~1.

By comparing the representations $(0, *)$ of $(A_1, B_1)$ and $(A_2, B_2)$, it is seen that $(A_2,B_2)$ can be simultaneously transformed into:
\begin{eqnarray*}
(A_2, B_2) = (\tilde{M}, 1) \cdot
\left(
\begin{pmatrix}
	\tilde{a}_{11} & \tilde{a}_{12} & 0 \\
	\tilde{a}_{12} & \tilde{a}_{22} & 0 \\
	0 & 0 & 0
\end{pmatrix},
\begin{pmatrix}
	0 & 0 & 0 \\
	0 & 0 & 0 \\
	0 & 0 & 1
\end{pmatrix}
\right),\
\tilde{M}
:= 
\begin{pmatrix}
  1 & 0 & \tilde{m}_1 \\ 
  0 & 1 & \tilde{m}_2 \\ 
  0 & 0 & m_3 
\end{pmatrix}.
\end{eqnarray*}
The greatest common divisor of $\tilde{m}_1, \tilde{m}_2$, $0 \neq m_3 \in \IntegerRing$ is $1$, owing to $q_{\IntegerRing}(B_1) = q_{\IntegerRing}(B_2)$.
Furthermore, 
either of the following may be assumed owing to $q_{\IntegerRing}(A_1) = q_{\IntegerRing}(A_2)$:
\begin{enumerate}[(I)]
	\item
$\begin{pmatrix}
	a_{11} & a_{12} \\
	a_{12} & a_{22}
\end{pmatrix}
=
\begin{pmatrix}
	1 & -1/2 \\
 -1/2 & 1
\end{pmatrix},\
\begin{pmatrix}
	\tilde{a}_{11} & \tilde{a}_{12} \\
	\tilde{a}_{12} & \tilde{a}_{22}
\end{pmatrix}
=
\begin{pmatrix}
	1 & 0 \\
	0 & 3
\end{pmatrix}$.
%w = \tilde{M} 
%\begin{pmatrix}
%  1 & 0 & 0 \\ 
%  1 & 2 & 0 \\ 
%  0 & 0 & 1 
%\end{pmatrix}
%M^{-1} =
%\begin{pmatrix}
%  1 & 0 & (\tilde{m}_1 - m_1) / m_3 \\ 
%  1 & 2 & (\tilde{m}_2 - m_1 - 2 m_2) / m_3 \\ 
%  0 & 0 & 1 
%\end{pmatrix}
	\item
$\tilde{A} := \begin{pmatrix}
	a_{11} & a_{12} \\
	a_{12} & a_{22}
\end{pmatrix}
=
\begin{pmatrix}
	\tilde{a}_{11} & \tilde{a}_{12} \\
	\tilde{a}_{12} & \tilde{a}_{22}
\end{pmatrix},\
0 < a_{11} \leq a_{22},
0 \leq - 2 a_{12} \leq a_{11}$, \IE $\tilde{A}$ is reduced.
\end{enumerate}

For any $m \neq 0$ and $n \in \IntegerRing$, 
the class of $n$ in $\IntegerRing/ m \IntegerRing$ is denoted by $\Mod{n}{m}$.
In case (I), 
the following is proved by 
considering the representations $(A_i(h), B_i(h)) = (1, *)$, $(3, *)$:
\begin{eqnarray*}
	\left\{ \pm \Mod{m_1}{m_3}, \pm \Mod{m_2}{m_3}, \pm \Mod{m_1+m_2}{m_3} \right\}
	&=&
	\left\{ \pm \Mod{\tilde{m}_1}{m_3} \right\}, \\
	\left\{ \pm \Mod{m_1 - m_2}{m_3}, \pm \Mod{m_1 + 2 m_2}{m_3}, \pm \Mod{2 m_1 + m_2}{m_3} \right\}
	&=&
	\left\{ \pm \Mod{\tilde{m}_2}{m_3} \right\}.
\end{eqnarray*}
The above can hold only when either of the following holds:
\begin{itemize}
\item $\Mod{m_1}{m_3} = \Mod{m_2}{m_3} = \Mod{\tilde{m_1}}{m_3} = \Mod{\tilde{m_2}}{m_3} = \Mod{0}{m_3}$ or 
\item $m_3 = 3$, $\Mod{m_1}{3} = \Mod{m_2}{3} = \pm \Mod{\tilde{m_1}}{3}$, $\Mod{\tilde{m_2}}{3} = \Mod{0}{3}$.
\end{itemize}
We note that each corresponds to the cases (i), (ii), respectively. Thus, the theorem is proved in case (I).

In case (II), 
$\Mod{m_1}{m_3} = \Mod{\tilde{m}_1}{m_3}$, $\Mod{m_2}{m_3} = \pm \Mod{\tilde{m}_2}{m_3}$ may be assumed,
by changing the basis of $\IntegerRing^3$ without losing the property (II).
%If $\Mod{m_1}{m_3} = -\Mod{m_1}{m_3}$ or $\Mod{m_2}{m_3} = -\Mod{m_2}{m_3}$, then  
In what follows, we assume that 
 $\Mod{m_1}{m_3} \neq -\Mod{m_1}{m_3}$ and $\Mod{m_2}{m_3} = -\Mod{\tilde{m}_2}{m_3} \neq -\Mod{m_2}{m_3}$,
because otherwise, $(A_1, B_1)$ and $(A_2, B_2)$ are equivalent by the action of $GL_3(\IntegerRing)$.
Furthermore, 
the multiplicity of the representation $a_{11} + a_{22} + 2a_{12}$ of
$\tilde{A}$ must be greater than $1$, otherwise we would have $\Mod{m_1 + m_2}{m_3} = \Mod{\tilde{m}_1 + \tilde{m}_2}{m_3}$, 
and $\Mod{m_2}{m_3} = \Mod{\tilde{m}_2}{m_3}$.
This can happen only if $a_{12} = 0$ or $a_{12} = -a_{11}/2$.
If $a_{12} = 0$, then
$(A_1, B_1)$ and $(A_2, B_2)$ are equivalent by the action of $GL_3(\IntegerRing)$.
If $a_{12} = -a_{11}/2$, 
then we see from the representations $(A_i(h), B_i(h)) = (a_{11} + a_{22} + 2a_{12}, *)$ that either of the following is required:
\begin{itemize}
	\item 
 	$\left\{ \pm \Mod{m_2}{m_3}, \pm \Mod{m_1 + m_2}{m_3} \right\}
	=
	\left\{ \pm \Mod{m_2}{m_3}, \pm \Mod{m_1 - m_2}{m_3} \right\}$.
	\item 
	$a_{11} = a_{22} = - 2 a_{12}$ and \\
	$\left\{ \pm \Mod{m_1}{m_3}, \pm \Mod{m_2}{m_3}, \pm \Mod{m_1 + m_2}{m_3} \right\}
	=
	\left\{ \pm \Mod{m_1}{m_3}, \pm \Mod{m_2}{m_3}, \pm \Mod{m_1 - m_2}{m_3} \right\}$.
\end{itemize}
If $\Mod{m_1 + m_2}{m_3} = \pm \Mod{m_1 - m_2}{m_3}$, then 
$\Mod{m_1}{m_3} = -\Mod{m_1}{m_3}$ or $\Mod{m_2}{m_3} = -\Mod{m_2}{m_3}$ must hold, which contradicts with the above assumption.
However, in the former of the above cases, $\Mod{m_1 + m_2}{m_3} = \pm \Mod{m_1 - m_2}{m_3}$ always holds.
In the latter, $\Mod{m_1 + m_2}{m_3} \neq \pm \Mod{m_1 - m_2}{m_3}$ implies that
$\Mod{m_1 + m_2}{m_3} = \pm \Mod{m_i}{m_3}$ and $\Mod{m_1 - m_2}{m_3} = \pm \Mod{m_j}{m_3}$ hold for either $(i, j) = (1, 2)$ or $(i,j)=(2, 1)$.
From $\Mod{m_1}{m_3} \neq -\Mod{m_1}{m_3}$ and $\Mod{m_2}{m_3} \neq -\Mod{m_2}{m_3}$,
the following is obtained:
$$(\Mod{m_1 + m_2}{m_3}, \Mod{m_1 - m_2}{m_3}) = (\Mod{-m_1}{m_3}, \Mod{m_2}{m_3}) \text{ or } (\Mod{-m_2}{m_3}, \Mod{-m_1}{m_3}).$$

Comparing the representation $(a_{11} + a_{22} - 2 a_{12}, *)$, we have
\begin{eqnarray*}
	& & \hspace{-20mm} 
	\{ \pm \Mod{m_1 - m_2}{m_3}, \pm \Mod{m_1 + 2 m_2}{m_3}, \pm \Mod{2 m_1 + m_2}{m_3} \} \\
	&=&
	\{ \pm \Mod{m_1 + m_2}{m_3}, \pm \Mod{m_1 - 2 m_2}{m_3}, \pm \Mod{2 m_1 - m_2}{m_3} \}.
\end{eqnarray*}

If $(\Mod{m_1 + m_2}{m_3}, \Mod{m_1 - m_2}{m_3}) = (\Mod{-m_1}{m_3}, \Mod{m_2}{m_3})$, 
the following is obtained from 
$\Mod{m_1}{m_3} = \Mod{2 m_2}{m_3}$ and $\Mod{m_2}{m_3} = -\Mod{2 m_1}{m_3}$:
\begin{eqnarray}\label{eq: {m2,0} equiv {m1} }
	\{ \pm \Mod{m_2}{m_3}, \Mod{0}{m_3} \}
	&=&
	\{ \pm \Mod{m_1}{m_3}, \Mod{0}{m_3} \}.
\end{eqnarray}

Even if $(\Mod{m_1 + m_2}{m_3}, \Mod{m_1 - m_2}{m_3}) = (\Mod{-m_2}{m_3}, \Mod{-m_1}{m_3})$, 
Eq.(\ref{eq: {m2,0} equiv {m1} }) is obtained from 
$\Mod{m_1}{m_3} = -\Mod{2 m_2}{m_3}$ and
$\Mod{m_2}{m_3} = \Mod{2 m_1}{m_3}$.
%\begin{eqnarray}
%	\{ \pm \Mod{m_1}{m_3}, \Mod{0}{m_3} \}
%	&=&
%	\{ \pm \Mod{m_2}{m_3}, \Mod{0}{m_3} \}.
%\end{eqnarray}

Eq.(\ref{eq: {m2,0} equiv {m1} }) implies $\Mod{m_1}{m_3} = \pm \Mod{m_2}{m_3}$, hence this case is impossible.
\end{proof}

Proposition \ref{thm:proposition 1} is proved in the remaining part.

\begin{proof}[Proof of Proposition \ref{thm:proposition 1}]
Since the ``only if'' part is clear, we shall prove the ``if'' part;
for any positive-definite $f_1, f_2 \in {\rm Sym}^2 (\RealField^3)^*$ with $q_\IntegerRing(f_1) = q_\IntegerRing(f_2)$,
take $\lambda_{1}, \ldots, \lambda_{s} \in \RealField$ linearly independent over $\RationalField$ and 
positive-definite $A_{1j}$, $A_{2j} \in {\rm Sym}^2 (\RationalField^3)^*$ ($1 \leq j \leq s$) satisfying $f_1 = \sum_{j=1}^s \lambda_j A_{1j}$, $f_2 = \sum_{j=1}^s \lambda_j A_{2j}$ 
as in Lemma \ref{lem:decomposition of positive definite symmety matrices}.
Furthermore, 
In this case, $q_\IntegerRing(A_{11}, \ldots, A_{1s}) = q_\IntegerRing(A_{21}, \ldots, A_{2s})$ holds.
Therefore, the proposition is obtained by proving the following (**):

\begin{description}
	\item[(**)]
If both of $(A_{i1}, A_{i2}, \ldots, A_{is}) \in {\rm Sym}^2 (\RationalField^3)^* \otimes_{\RationalField} \RationalField^s$ ($i = 1, 2$) satisfy all of
\begin{enumerate}[(a)]
\item $A_{i1}, \ldots, A_{is}$ span a linear space of dimension more than 1 over $\RationalField$,
\item $\sum_{j = 1}^s c_j A_{ij}$ is positive definite for some $c_1, \ldots, c_s \in \RationalField$,
\item $q_\IntegerRing(A_{11}, \ldots, A_{1s}) = q_\IntegerRing(A_{21}, \ldots, A_{2s})$,
\end{enumerate}
there exists $w \in GL_3(\IntegerRing)$ such that $(A_{11}, \ldots, A_{1s}) = (w, 1) \cdot (A_{21}, \ldots, A_{2s})$,
or there exist $w_1, w_2 \in GL_3(\IntegerRing)$ and $v \in GL_s(\RationalField)$ such that $\{ (A_{11}, \ldots, A_{1s}), (A_{21}, \ldots, A_{2s}) \} $ equals either of the following:
\begin{enumerate}[(i)]
\item 
$\left\{  
	(w_1, v) \cdot ( 0, \cdots, 0, x_1^2 - x_1 x_2 + x_2^2, x_3^2 ),
	(w_2, v) \cdot ( 0, \cdots, 0, x_1^2 + 3 x_2^2, x_3^2 )
\right\}$,

\item $\left\{  
	\begin{matrix} (w_1, v) \cdot ( 0, \cdots, 0, x_1^2 - x_1 x_2 + x_2^2,  (x_1 + x_2 + 3 x_3)^2 ), \\
	(w_2, v) \cdot ( 0, \cdots, 0, x_1^2 + 3 x_2^2, (x_1 + 3 x_3)^2 ) \end{matrix} 
\right\}$,
\end{enumerate}
where $0$ is the ternary quadratic form that maps all the ${\mathbf x} \in \IntegerRing^3$ to 0, and
$GL_3(\IntegerRing) \times GL_s(\RationalField) \ni (w, v)$
acts on ${\rm Sym}^2 (\RationalField^3)^* \otimes_{\RationalField} \RationalField^s$
by 
\begin{eqnarray*}
\left(w, (v_{ij})
\right) \cdot (A_1 \ldots A_s) = 
\left( \sum_{j=1}^s v_{1j} A_j({\mathbf x} w), \sum_{j=1}^s v_{2j} A_j({\mathbf x} w), \ldots, \sum_{j=1}^s v_{sj} A_j({\mathbf x} w) \right).
\end{eqnarray*}
\end{description}

The case of $s = 2$ is equivalent to the assumption (*).
In order to prove (**) by induction, we assume (**) is true for $s = 2, \ldots, n$ 
for some $n \geq 2$.
If $A_{11}, \ldots, A_{1 n+1}$ span a linear space of dimension $m < n+1$, by using the action of $\{ 1 \} \times GL_{n + 1}(\RationalField)$,
$(A_{11}, \ldots, A_{1 n+1})$ is mapped to $(\tilde{A}_1, \ldots, \tilde{A}_{m}, 0 \ldots, 0)$ for some $\tilde{A}_j \in {\rm Sym}^2 (\RationalField^3)^*$.
From $q_\IntegerRing(A_{11}, \ldots, A_{1s}) = q_\IntegerRing(A_{21}, \ldots, A_{2s})$,
$(A_{21}, \ldots, A_{2s})$ is simultaneously mapped to $(\tilde{B}_1, \ldots, \tilde{B}_{m}, 0 \ldots, 0)$ for some $\tilde{B}_j \in {\rm Sym}^2 (\RationalField^3)^*$.
The proposition is proved by induction in this case.
Thus we assume the dimension of the space spanned by $A_{11}, \ldots, A_{1 n+1}$ is exactly $n+1$.
By using the action of $GL_3(\IntegerRing) \times GL_{n+1}(\RationalField)$,
either of the following may be assumed:
\begin{enumerate}[(I)]
	\item $A_{1j} = A_{2j}$ ($1 \leq j \leq n$) and $\sum_{j=1}^n c_j A_{1j} \succ 0$ for some $c_j \in \RealField$,
	\item $n = 2$ and $(A_{11}, A_{12}) = (x_1^2 - x_1 x_2 + x_2^2, x_3^2)$, $(A_{21}, A_{22}) = (x_1^2 + 3 x_2^2, x_3^2)$,
	\item $n = 2$ and $(A_{11}, A_{12}) = (x_1^2 - x_1 x_2 + x_2^2, (x_1 + x_2 + 3 x_3)^2)$, $(A_{21}, A_{22}) = (x_1^2 + 3 x_2^2, (x_1 + 3 x_3)^2)$.
\end{enumerate}
In the case (I), 
define $H := \{ g \in GL_3(\IntegerRing) : A_{ij} = g \cdot A_{ij} \text{ for all } 1 \leq j \leq n \}$. 
For any $c \in \RationalField$, 
owing to 
$q_\IntegerRing(A_{11}, \ldots, A_{1 n-1}, A_{1 n} + c A_{1 n+1}) = q_\IntegerRing(A_{21}, \ldots, A_{2 n-1}, A_{2 n} + c A_{2 n+1})$
and the assumption of induction, 
there exists 
$h \in H$
such that $(h, 1) \cdot (A_{11}, \ldots, A_{1 n-1}, A_{1 n} + c A_{1 n+1}) = (A_{21}, \ldots, A_{2 n-1}, A_{2 n} + c A_{2 n+1})$.
Since $H$ is a finite group, some $h \in H$ satisfies $h \cdot (A_{1 n} + c_i A_{1 n+1})= A_{2 n} + c_i A_{2 n+1}$ for some distinct $c_1 \neq c_2$. 
Hence $h \cdot A_{1j} = A_{2j}$ holds for all $1 \leq j \leq n + 1$.

In the cases (II) and (III), owing to $A_{11} \not\sim A_{21}$,
${\rm Disc}(A_{11}, A_{12} + d A_{13}) = {\rm Disc}(A_{21}, A_{22} + d A_{23}) = 0$ 
is required for any $d \in \RationalField$.
Therefore $A_{i1}, A_{i2}, A_{i3}$ are linearly dependent over $\RationalField$ for each $i = 1, 2$.
Hence, (**) is true in any case.
\end{proof}

\paragraph{Acknowledgments}
The author would like to extend her gratitude to Professor T. Kamiyama of KEK. 
He gave her an opportunity to apply the theory of quadratic forms to a real scientific problem.
She also appreciates 
Professor A. Earnest of Southern Illinois University for sending her his lecture slides containing some references and a copy of Kaplansky's letter to Schiemann,
Professor T. Oda of the university of Tokyo, Dr. S. Harashita of Yokohama National University and Dr. K. Gunji of 
Chiba Institute of Technology for discussion during a regular seminar.
For various checks of our program and computation of genera and spinor genera, we used the algebra system Magma \cite{MR1484478}.

\appendix

\section{The algorithm to obtain all the positive-definite quadratic forms with a given set of representations}
\label{Algorithm}

In this section, the method to obtain all the positive-definite quadratic forms with the identical set of representations over $\IntegerRing$
is explained. 
For such a quadratic form $f$ and a given $M > 0$, 
all the elements $q_1, \ldots, q_t$ of $q_\IntegerRing(f)$ less than $M$ can be computed \cite{Pohst81}.
Therefore, the problem is reduced to enumeration of all the quadratic forms
with the property $q_\IntegerRing(f) \cap [0, M] = \{ q_1, \ldots, q_t \}$.
The method  described in Table~\ref{Recursive procedure} can be applied to any $N$-ary positive-definite quadratic forms as long as $N \leq 4$,
which equals the theoretical optimum, because infinitely many solutions may exist if $N > 4$.

In what follows, a set of vectors $v_1, \cdots, v_n$ of $\IntegerRing^N$ 
is said to be \textit{primitive} if it is a subset of some basis of $\IntegerRing^N$.
An $N$-ary quadratic form $S$ is \textit{Minkowski-reduced}
if the following holds for any $1 \leq n \leq N$:
\begin{eqnarray}\label{eq:definition of Minkowski-reduction}
	S ({\mathbf e}_n)
	=
	\min \left\{
		S(v) :
		\begin{matrix}
		v \in \IntegerRing^N \text{ such that } \langle {\mathbf e}_1, \ldots, {\mathbf e}_{n-1}, v \rangle \\
		\text{ is a primitive set of } \IntegerRing^N
		\end{matrix}
	\right\},
\end{eqnarray}
where ${\mathbf e}_n$ is a vector with $1$ in its $n$-th component and $0$ in the remaining components.
It is known that if $N \leq 4$, then $S = (s_{ij})$ is Minkowski-reduced
if and only if the following inequalities hold (\CF Lemma 1.2 of chap. 12, Cassels (1978)\nocite{Cassels78}):
\begin{eqnarray}\label{eq:condition3}
	\begin{cases}
		0 < s_{11} \leq \cdots \leq s_{NN},\\
		s_{jj} \leq S(v) \text{ for any } 1 \leq j \leq N \text{ and vectors } v \text{ with the entries } \\
		v_i = -1, 0, 1\ (1 \leq i < j),  v_j = 1, v_k = 0\ (j < k \leq N).
	\end{cases}
\end{eqnarray}

The following is frequently added to the definition of Minkowski reduction:
\begin{eqnarray}\label{eq:condition4}
s_{i i+1}  \leq 0\ (1 \leq i < N).
\end{eqnarray}
In what follows,
Eq.(\ref{eq:condition3}), (\ref{eq:condition4}) are adopted as the inequalities of the Minkowski reduction for $N \leq 4$.

Table~\ref{Recursive procedure} presents a recursive procedure for generating all candidates of $N$-ary quadratic forms
from a sorted set $\Lambda := \langle q_1, \cdots, q_t \rangle \subset \RealField_{> 0}$.
If the recursive procedure begins with the arguments $m=n=1$, $q_{min} = q_{max} = q_1$,
all positive-definite quadratic forms satisfying the following
in addition to Eq.(\ref{eq:condition3}), (\ref{eq:condition4}) are enumerated in the output array $Ans$:
\begin{eqnarray}
	\begin{cases}
				\text{any } q \in \Lambda \text{ with } q \leq s_{NN} \text{ belongs to } q_\IntegerRing(S), \\
				s_{nn}, s_{mm} + s_{nn} + 2 s_{mn} \in \Lambda \text{ for any } 1 \leq m, n \leq N.
	\end{cases}
\end{eqnarray}

After the execution of the algorithm, it is possible to reduce the number of candidate solutions in $Ans$ 
by
\begin{enumerate}[(a)]
\item
checking if
$\Lambda = q_\IntegerRing(S) \cap [0, q_{t}]$ holds,
\item removing either $S$ or $S_2$ from $Ans$,
if they are equivalent over $\IntegerRing$.
\end{enumerate}

\begin{table}
	\caption{A recursive procedure to obtain all the positive-definite quadratic forms with a given set $\Lambda$ of  representations}
	\label{Recursive procedure}
	\begin{tabular}{ccl}
\hline
	\multicolumn{3}{c}{{\bf void func($\Lambda, N, S, m, n, q_{min}, q_{max}, Ans$)}} \\
	(Input) \\
	$\Lambda$ & : & a sorted sequence $\langle q_1, \ldots, q_t \rangle$ of positive numbers \\
%	$M$ & : & a natural number \\
	$1 \leq N \leq 4$ & : & number of variables, \\
	$S$ & : & a quadratic form $(s_{ij})_{1 \leq i, j \leq N}$ \\
	$m, n$ & : & integers $1 \leq m \leq n \leq N$ indicating that the algorithm is determining \\
	& & the $(m, n)$-entry of $S$ \\
	$q_{min}, q_{max}$ & : & numbers indicating \\
	& & 	$\begin{cases}
					q_{min} \leq s_{nn} \leq q_{max} & \text{if } m = n, \\
					q_{min} \leq s_{mm} + s_{nn} + 2 s_{mn} \leq q_{max} & \text{otherwise.}
			\end{cases}$ \\
	(Output) \\
	$Ans$ & : & an array of $N \times N$ Minkowski-reduced symmetric matrices $\tilde{S} := (\tilde{s}_{ij})$ \\
    & & that satisfy \\
	& &
			$\begin{cases}
				\tilde{s}_{NN} \leq q_t,\
				\Lambda \cap [0, \tilde{s}_{NN}] \subset q_\IntegerRing(\tilde{S}), \\
				\tilde{s}_{nn}, \tilde{s}_{mm} + \tilde{s}_{nn} + 2 \tilde{s}_{mn} \in \Lambda \text{ for any } 1 \leq m, n \leq N.
			\end{cases}$
	\end{tabular}
	\begin{tabular}{p{1mm}l}
\\
	(Start) \\
	1: & Take integers $I$ and $J$ such that $\Lambda \cap [q_{min}, q_{max}] = \langle q_I, \cdots, q_J \rangle$. \\
	2: & for $l=I$ to $J$ do \\
	3: & \hspace{5mm}	if $m=n$ then \\
	4: & \hspace{10mm}	          $s_{n n} := q_{l}$. \\
	5: & \hspace{5mm}	else \\
	6: & \hspace{10mm}            $s_{m n} := s_{n m} := \frac{1}{2}(q_{l} - s_{m m} - s_{n n})$. \\
	7: & \hspace{5mm}   end if \\
	8: & \hspace{5mm}	if $m=1$ then \\
	9: & \hspace{10mm}			if $n \geq N$ then \\
	10: & \hspace{20mm}					Insert $S$ in $Ans$. \\ % $/\star$ $(g_{ij})$ is a candidate of the solution. $\star/$\\
	11: & \hspace{10mm}			else  \\
	12: & \hspace{15mm}			$T := (s_{ij})_{1 \leq i, j \leq n}$.  /* an $n$-by-$n$ submatrix of $S$ */. \\
	13: & \hspace{15mm} 					${t_2} := \max \left\{ 1 \leq i \leq t :
												\begin{matrix}
													q_{1}, \ldots, q_{i-1} \in q_\IntegerRing(T)
\end{matrix} \right\}$. \\
	14: & \hspace{15mm}    				func($\Lambda, N, S, n+1, n+1, s_{nn}, q_{t_2}, Ans$). \\
	15: & \hspace{10mm} end if \\
	16: & \hspace{5mm}	else  \\
	17: & \hspace{10mm}		Take $p_{min}$, $p_{max}$ such that $s_{m-1 n}$ and the entries $s_{m-1 m-1}$, $s_{nn}$ of $S$ determined in \\
	& \hspace{10mm} the previous steps fulfill Eq.(\ref{eq:condition3}), (\ref{eq:condition4}) iff $p_{min} \leq s_{m-1 m-1} + s_{nn}  + 2 s_{m-1 n} \leq p_{max}$. \\
	18: & \hspace{10mm}   	func($\Lambda, N, S, m-1, n, S, p_{min}, p_{max}, Ans$). \\
	19: & \hspace{5mm} end if \\
	20: & end for \\
\hline
	\end{tabular}
% 	\footnotetext[1]{
% 	Proposition \ref{lem:no sublattice} claims that there always exists $t_2 < \infty$,
% 	even if $\Lambda$ is replaced by $\Lambda(S_0)$ of some metric tensor $S_0$.
% 	}
	\footnotetext[1]{
	Such $p_{min}$, $p_{max}$ are determined uniquely
	because Eq.(\ref{eq:condition3}), (\ref{eq:condition4})
	define a polyhedral convex cone.
	}
\end{table}

As a consequence, if $q_t$ is sufficiently large,
all the $N$-ary positive-definite quadratic form $S$ that satisfy $\Lambda = q_\IntegerRing(S) \cap [0, q_t]$
are contained in the output array.
If we set $\Lambda := q_\IntegerRing(S_0) \cap [0, q_t]$ for some $S_0$, 
all the $N$-ary quadratic forms $S$ with $q_\IntegerRing(S) = q_\IntegerRing(S_0)$ are in the output.
%If $q_t \in \Lambda$ is not sufficiently large,
%all the $S$
%satisfying $\Lambda = q_\IntegerRing(S) \cap [0, q_{t}]$
%might not be obtained.
%(this is so called the dominant zone problem (\CF Shirley (1980)\nocite{Shirley80})).
%However, 
After the execution of the algorithm,
it is possible to determine
whether or not $q_t$ is large enough to obtain all such $S$,
just by checking if $t_2 < t$ in line 13 and $p_{max} \leq q_t$ in line 17 hold.
If both are true, then
all the quadratic forms $S$ satisfying $\Lambda = q_\IntegerRing(S) \cap [0, q_t]$ are contained in $Ans$.
%Therefore, the dominant zone problem can be easily avoided in this case.

The algorithm is completed in a finite number of steps,
even if
$q_\IntegerRing(S_0)$ of some $S_0$ is used instead of the finite set $\Lambda$
(in fact, this can be programmed
by adding elements of $q_\IntegerRing(S_0)$ in $\Lambda$ when they are necessary).
Even though $t = \infty$ in such cases,
$t_2$ in line 13 of Table~\ref{Recursive procedure} is always finite, 
as a consequence of the following proposition:

\begin{prop}\label{prop:no sublattice}
If $S$, $S_2$ are positive-definite quadratic forms over $\RealField$ of degree $N$ and $N_2$, respectively, 
with  $1 \leq N_2 < \min \{ 4, N \}$, then
 $q_\RationalField(S) \not\subset q_\RationalField(S_2)$, hence $q_\IntegerRing(S) \not\subset q_\IntegerRing(S_2)$.
\end{prop}

\begin{proof}
It may be assumed that $S$ and $S_2$ have rational entries,
since due to Lemma~\ref{lem:decomposition of positive definite symmety matrices},
they are simultaneously represented as finite sums $S = \sum_{j=1}^2 \lambda_j T_j$, $S_2 = \sum_{j=1}^s \lambda_j T_{2j}$, 
where $\lambda_1, \ldots, \lambda_s \in \RealField$ are linearly independent over $\RationalField$
and every $T_j$, $T_{2j}$ is rational and positive-definite.
In the rational case, the proposition follows from Lemma~\ref{lem:no sublattice}.
\end{proof}

\begin{lem}\label{lem:decomposition of positive definite symmety matrices}
Let $S_i$ ($1 \leq i \leq m$) be $N_i$-ary positive-definite quadratic forms over $\RealField$.
There are $\lambda_{1}, \ldots, \lambda_{s} \in \RealField$ linearly independent over $\RationalField$
and $N_i$-ary positive-definite quadratic forms $T_{ij}$ over $\RationalField$ ($1 \leq i \leq m$, $1 \leq j \leq s$)
such that every $S_i$ is represented as a finite sum $S_i = \sum_{j=1}^s \lambda_j T_{ij}$.
\end{lem}
\begin{proof}
If the condition that $T_{ij}$ is positive-definite is removed,
it is clear that such $\lambda_j$ and $T_{ij}$ exist.
Since $\sum_{j=1}^s \lambda_j T_{ij}$ is positive-definite for all of $i = 1, \ldots, m$,
there exists $\epsilon > 0$ such that 
$\sum_{j=1}^s c_j T_{ij}$ satisfies the same property
for any $(c_1, \ldots, c_s) \in \RationalField^s$ with $\abs{ c_j - \lambda_j } < \epsilon$ ($j = 1, \ldots, s$).
Hence, if we choose $U = (u_{j_1 j_2}) \in GL_s(\RationalField)$ so that $\abs{ u_{j_1 j_2} - \lambda_{j_2} } < \epsilon$ holds 
for all $j_1, j_2 = 1, \ldots, s$,
the new $(\lambda_1, \ldots, \lambda_s)$ and $T_{ij}$ replaced by 
$(\lambda_1, \ldots, \lambda_s) U^{-1}$ and $\sum_{j_2=1}^s u_{j j_2} T_{i j_2}$, 
satisfy all the required properties.

\end{proof}

For any $N$-ary quadratic form $S$ over a field $k$,
$S$ is \textit{singular}, if the determinant of the corresponding symmetric matrix is 0.
$S$ is \textit{isotropic} over $k$, if $0 \in q_k(S)$,
Otherwise $S$ is \textit{anisotropic} over $k$.

\begin{lem}\label{lem:no sublattice}
For any positive integers $N$, $N_2$ with $1 \leq N_2 < \min \{ 4, N \}$,
we assume that an $N$-ary rational quadratic form $S$ is non-singular
and an $N_2$-ary rational quadratic form $S_2$ is anisotropic over $\RationalField$.
Then, $q_\RationalField(S) \not\subset q_\RationalField(S_2)$.
\end{lem}
\begin{proof}
We may assume that $N_2 + 1 = N = 4$, as the other cases easily follow from this. 
Since $S$ is not singular,
it satisfies $q_{\RationalField_p}(S) \supset \RationalField_p^\times$ for any $p \neq \infty$.
In addition,
% Furthermore, by Lemma \ref{lem:decomposition of positive definite symmety matrices},
% it may be assumed that $S$ and $S_2$ are rational, $S$ is non-singular and $S_2$ is anisotropic.
there exists a finite prime $p$
such that
$q_{\RationalField_p}(S_2) \not\supset \RationalField_p^\times$
(\CF Corollary 2 of Theorem 4.1 in Chapter 6, Cassels (1978)).
If $q_{\RationalField}(S) \subset q_{\RationalField}(S_2)$,
%$\Lambda_{S_2}(\RationalField) \supset \Lambda_{S}(\RationalField)$, therefore
then $q_{\RationalField_p}(S) \subset q_{\RationalField_p}(S_2)$ is required for any $p$.
This is a contradiction.
\end{proof}

In line 13 of Table \ref{Recursive procedure}, $s_{n+1 n+1} \leq q_{t_2}$ is assumed,
which is proved as follows;
with regard to the symmetric matrix $T$ defined in line 12,
$q_{t_2} \in \Lambda$ does not belong to $q_\IntegerRing(T)$.
Hence, if $S$ is an extension of $T$ searched for, there exists $v \in \IntegerRing^N$ such that
${\mathbf e}_1, \ldots, {\mathbf e}_n, v$
are linearly independent and $S(v) = q_{t_2}$.
Let us recall that the $n$-th \textit{successive minimum} $\lambda_n$ of $S$ is defined as follows
\begin{eqnarray*}\label{eq: definition of successive minimum}
	\lambda_n :=
	\min \left\{
		\max\{ S(v_i) : 1 \leq i \leq n \} : 
		v_1, \ldots, v_n \in \IntegerRing^N \text{ are linearly independent over $\RationalField$}
	\right\}.
\end{eqnarray*}
If $1 \leq n \leq N \leq 4$, 
the above $v_1, \ldots, v_n$ can be chosen from a primitive set of $\IntegerRing^N$
(\CF Wan der Waerden (1956)\nocite{Waerden56}).
%As a result, the $n$'th diagonal entry of a Minkowski-reduced form equals its $n$-th successive minimum.
Therefore, $s_{n+1 n+1} \leq q_{t_2}$ holds in line 13.

\bibliographystyle{plain}
\bibliography{document8}

\end{document}